\documentclass[letterpaper,12pt,reqno]{amsart}
\usepackage{stmaryrd}
\usepackage{latexsym,amsfonts,amssymb}
\usepackage{mathrsfs}
\usepackage{amsmath}
\usepackage{amsthm}
\usepackage[dvips]{graphics}
\usepackage{color}
\usepackage{eucal}

\usepackage{multicol}
\usepackage[normalem]{ulem}
\usepackage[utf8]{inputenc}

\newtheorem{theorem}{Theorem}[section]
\newtheorem{lemma}[theorem]{Lemma}
\newtheorem{proposition}[theorem]{Proposition}
\newtheorem{corollary}[theorem]{Corollary}

\theoremstyle{definition}
\newtheorem{definition}[theorem]{Definition}
\newtheorem{example}[theorem]{Example}

\newtheorem{remark}[theorem]{Remark}
\newtheorem{remarks}[theorem]{Remarks}

\newtheorem{conventions}[theorem]{Convention}
\numberwithin{equation}{theorem}

\def\wt{\text{\rm wt}\,  }

\def\wt{\text{\rm wt}\,  }

\def\mod{\text{\rm mod}\,  }

\def\mfq{{\mathfrak q(n)}}

\def\al{\alpha}

\def\de{\delta}

\def\la{\lambda}

\def\La{\Lambda}

\def\Om{\Omega}

\def \bZ{\mathbb Z}
\def \bN{\mathbb N}

\def\span{{\text{\rm span}}}
\def\wt{{\text{\rm wt}}}

\def\diag{{\text{\rm diag}}}

\def\sbi{{s_{\bar i}}}
\def\sbii{{s_{\overline{ i+1}}}}

\def\sZ{{\mathcal Z}}

\def\sE{{\mathcal E}}
\def\sF{{\mathcal F}}

\def\la{{\lambda}}
\def\al{{\alpha}}

\def\sR{{\mathcal R}}

\def\sX{{\mathcal X}}
\def\sK{{\mathcal K}}
\def\sTn{{\mathscr T_\up(n)}}
\def\sTnr{{\mathscr T_\up(n,r)}}
\def\sVn{{\mathscr V_\up(n)}}
\def\scrV{{\mathscr V}}
\def\wsTn{{\widehat{\mathscr T}_\up(n)}}

\def\sfK{{\mathsf K}}
\def\sfE{{\mathsf E}}
\def\sfF{{\mathsf F}}

\def\bfr{{\mathbf r}}
\def\bfc{{\mathbf c}}

\newcommand{\ul}{\underline}
\newcommand{\ol}{\overline}
\def\fks{{\mathfrak{s}}}

\def\fkm{{\mathfrak m}}
\def\fkb{{\mathfrak b}}
\def\fkc{{\mathfrak c}}
\def\fka{{\mathfrak a}}
\def\fkF{{\mathfrak F}}
\def\fkE{{\mathfrak E}}
\def\fkB{{\mathfrak B}}
\def\fkL{{\mathfrak L}}
\def\fkM{{\mathfrak M}}

\def\sck{{\textsc{k}}}
\def\sce{{\textsc{e}}}
\def\scf{{\textsc{f}}}

\def\ro{{\rm ro}}
\def\co{{\rm co}}

\newcommand{\Z}{\mathbb{Z}}
\newcommand{\0}{\bar 0}

\newcommand{\ov}{\overline}
\newcommand{\bi}{{\bar i}}
\newcommand{\bj}{{\bar j}}
\newcommand{\ds}{\displaystyle}
\newcommand{\Q}{\mathbb Q}
\newcommand{\qUq}{{U_\up(\mathfrak q_n)}}
\newcommand{\qUqZ}{{U_{\up,\sZ}(\mathfrak q_n)}}

\newcommand{\N}{\mathbb N}
\newcommand{\mc}{\mathcal}

\def\sce{{\textsc e}}
\def\scf{{\textsc f}}
\def\sck{{\textsc k}}
\def\k{{\mathsf K}}
\def\e{{\mathsf E}}
\def\f{{\mathsf F}}

\def\q{{{\upsilon}}}
\def\v{{\upsilon}}
\def\up{{\upsilon}}
\def\ups{{{\upsilon}}}

\def\i{{h}}

\def\SW{\mathscr{A}_\up(n)}
\def\SWr{\mathscr{A}_\up(n,r)}
\def\noma#1{X^{#1}}

\def\nomab#1{X^{[ #1 ]}}

\def\bfa{{{\bf a}}}
\def\bfb{{{\bf b}}}
\def\bfaa{{{\bf c}}}
\def\bfe{{{\bf e}}}
\def\bfj{{{\bf j}}}

\def \pari{\partial_i}
\def \parj{\partial_j}
\def \chii{\chi_i}
\def \chij{\chi_j}
\def \dej{\delta_j}
\def \dei{\delta_i}

\def \parbi{\partial_{\bar i}}
\def \parbj{\partial_{\bar j}}
\def \chibi{\chi_{\bar i}}
\def \chibj{\chi_{\bar j}}
\def \debj{\delta_{\bar j}}
\def \debi{\delta_{\bar i}}

\def \ebi{\mathcal{E}_{\bar i}}

\def \fbi{\mathcal{F}_{\bar i}}

\def \kbi{\mathcal{K}_{\bar i}}
\def \kbj{\mathcal{K}_{\bar j}}

\def \cei{\mathcal{E}_i}

\def \cki{\mathcal{K}_i}

\def \parb{\partial}

\def \upv{{\upsilon}}

\def\scrA{{\mathscr A}}
\def\scrL{{\mathscr L}}

\def \ceii{\mathcal{E}_{{i+1}}}
\def \ckii{\mathcal{K}_{{i+1}}}

\def \kbi{\mathcal{K}_{\bar i}}
\def \kbii{\mathcal{K}_{\ov{i+1}}}
\def \kbj{\mathcal{K}_{\bar j}}
\def \ebii{\mathcal{E}_{\ov{ i+1}}}

\def\sqA{{{}^\square\!A}}

\def\OE{{\mathsf E}}
\def\OG{{\mathsf G}}
\newcommand{\Xije}{{\OE_{i,j}}}
\newcommand{\Xijo}{{\overline{\OE}_{i,j}}}

\def\bfl{{\mathbf 0}}
\def\co{{\text{\rm co}}}
\def\ro{{\text{\rm ro}}}

\def\sfX{{\mathsf X}}
\def\sfY{{\mathsf Y}}
\begin{document}
\title{Quantum  queer supergroups 
 via $\up$-differential operators
        }
\date{}

\author{Jie Du, Yanan Lin  and Zhongguo Zhou}

\address{J. D., School of Mathematics and Statistics,
University of New South Wales, Sydney NSW 2052, Australia}
\email{j.du@unsw.edu.au}
\address{Y. L., School of Mathematical Sciences, Xiamen University, Xiamen 361005, China}
\email{ynlin@xmu.edu.cn}
\address{Z. Z., College of Science, Hohai University, Nanjing, China}
\email{zhgzhou@hhu.edu.cn}


\thanks{The work was partially supported by the UNSW Science FRG and the Natural Science Foundation of
China ($\#$11871404)}

\maketitle
\begin{abstract}
By using certain quantum differential operators, we construct a super representation for the quantum queer supergroup  $\qUq$. The underlying space of this representation is a deformed polynomial superalgebra in $2n^2$ variables whose homogeneous components can be used as the underlying spaces of queer $q$-Schur superalgebras. We then extend the representation to its formal power series algebra which contains a (super) submodule isomorphic to the regular representation  of $\qUq$.  A monomial basis $\mathfrak M$ for $\qUq$ plays  a key role in proving the isomorphism. In this way, we may present the quantum queer supergroup $\qUq$ by another new  basis $\mathfrak L$ together with some explicit multiplication formulas by the generators. As an application, similar presentations are obtained for queer $q$-Schur superalgebras via the above mentioned homogeneous components.

The existence of the bases $\mathfrak{M}$ and $\mathfrak{L}$ and the new presentation show that the seminal construction of quantum $\mathfrak{gl}_n$ established by Beilinson--Lusztig--MacPherson thirty years ago extends to this ``queer''  quantum supergroup via a completely different approach.
\end{abstract}
\maketitle



\tableofcontents

\section{Introduction}
The general linear Lie algebra $\mathfrak{gl}_n$ has two super analogues: the general linear Lie superalgebra $\mathfrak{gl}_{m|n}$ and the queer Lie superalgebra $\mathfrak q_n$. 
Generally speaking, there is a degree of resemblance between the Lie superalgebra $\mathfrak{gl}_{m|n}$ and the Lie algebra $\mathfrak{gl}_{n}$. But the queer Lie superalgebra $\mathfrak q_n$ differs drastically from the rest of the entire family of classical Lie superalgebras. For example, this is the only classical Lie superalgebra whose Cartan subsuperalgebra has a nontrivial odd part, which results in very interesting phenomenon on 
the highest weight space of a finite dimensional irreducible supermodule (see, e.g., \cite[\S1.5.4]{CW}). Another example is that the celebrated Schur--Weyl duality, which can be easily established for $\mathfrak{gl}_{m|n}$, becomes the Schur--Weyl--Sergeev duality, which involves the Sergeev superalgebra or the Hecke--Clifford superalgebra in the quantum case (see \cite{Se, Ol, DW1}). 

It is known that the representation theory of Lie superalgebras is much more complicated than the corresponding theory of Lie algebras. The study of  the representations of the queer quantum supergroup $\qUq$ faces even more challenges.
In the series of papers \cite{GJKK, GJKK2, GJKKK, GJKKK2}, D. Grantcharov et al.
investigated the representation theory of the quantum superalgebra $\qUq$ in the quantum characteristic 0 case (i.e., the non-root-of-unity case) and discovered that the category of finite dimensional $\qUq$-supermodules is not semisimple. However, the full subcategory $\mathcal O_{\text{int}}^{\geq0}$ of the so-called {\it tensor supermodules} (or polynomial super representations in the sense of \cite{DLZ}) is semisimple. By overcoming several challenges, they further established the crystal basis theory for the supermodules in this semisimple category. These latest developments may be regarded as new contributions to the combinatorial representation theory.

To explore the representation theory of $\qUq$ in positive quantum characteristics, especially the category of polynomial (or tensor) supermodules at a root of unity, one has to face new challenges. First, the quantum queer supergroup $\qUq$ will be replaced by the corresponding (super) hyperalgebra through a Lusztig type form defined by a PBW type basis. A better understanding of the structure of the hyperalgebra often requires a new presentation for $\qUq$.
Thus, an integral monomial basis and a triangular relation between the bases will be crucial to such a new construction. Note that the existence of such a triangular relation plays a key role in constructing a global crystal (or canonical) basis for the $\pm$-parts. Second, lifting the quantum Schur--Weyl--Sergeev duality to the integral level is the key to link the polynomial super representations of $\qUq$ with those of the Hecke-Clifford superalgebras via the queer $q$-Schur superalgebras (see \cite{DW1, DW2} and cf. \cite{DZ}).

In this paper, we will tackle the first challenges mentioned above. A model one could follow is the beautiful work by Beilinson--Lusztig--MacPherson (BLM) \cite{BLM} in which they discovered a new realisation for quantum $\mathfrak{gl}_n$ via finite dimensional $q$-Schur algebras. 
Here a key step is the discovery of a new basis arising from certain spanning sets for $q$-Schur algebras together with a triangularly related integral monomial basis.
The BLM work has been generalised by Gu and the first author \cite{DG} to the quantum supergroup $U_\up(\mathfrak{gl}_{m|n})$. However, due to the involvement of a Hecke--Clifford superalgebra, attempts for $\qUq$ via queer $q$-Schur superalgebras  were not successful. 

Fortunately, a new approach via quantum differential operators has recently been developed in a pilot study \cite{DZ}, where a new realisation for $U_\up(\mathfrak{gl}_{m|n})$ is obtained without using $q$-Schur superalgebras. The new approach is to construct directly the regular representation of the quantum supergroup on a certain polynomial superalgebra whose homogeneous components resemble $q$-Schur superalgebras. Thanks to \cite{DW2} by Wan and the first author,  where a basis for a queer $q$-Schur superalgebra is constructed in terms of a matrix labelling set, we are able to introduce a queer (or deformed) polynomial superalgebra on which $\qUq$ acts via certain quantum differential operators.
In this way, the regular representation for $\qUq$ is constructed via two new bases---a monomial basis $\fkM$ and a BLM type basis\footnote{A related PBW type basis will be denoted by $\fkB$ in \S2. Thus, the three bases $\mathfrak{B, L, M}$ are named after the BLM work.}  $\fkL$---and a new realisation is achieved for $\qUq$ and for queer $q$-Schur superalgebra as well. 

The new realisations for both $\qUq$ and queer $q$-Schur algebras are useful in addressing the second challenge mentioned above. In a forthcoming paper, we will establish a partial integral Schur--Weyl--Sergeev duality and the associated polynomial representation theory in positive quantum characteristics (cf. \cite{DLZ}).

We organise the paper as follows. Section 2 is preliminary. By using a special ordering on root vectors, we define in Section 3 a PBW type basis $\fkB$ for $\qUq$ (and for its Lusztig type form $\qUqZ$ as well). We order the basis by an order relation on the matrix index set and  revisit  certain commutation formulas for root vectors with respect to the order relation. A monomial basis $\fkM$
is constructed in Section 4. The key to the proof is a triangular relation between the monomial basis and the PBW basis. 

The second half of the paper is devoted to the construction of the regular representation of $\qUq$.

 We introduce a deformed polynomial superalgebra $\SW$ in Section 5. Several linear maps on $\SW$, including $n$ (partial) $\up$-differential operators, are defined and their relations are discussed so that a $\qUq$-supermodule structure on $\SW$ is built (Theorem \ref{supalg}). This is the most technical part of the paper. Using the Hopf algebra structure on $\qUq$, we then extend in Section 6 the action of $\qUq$ on $\SW$ to the $n$-fold tensor product $\sTn$ by some explicit action formulas (Lemma \ref{genmul}). Note that every homogeneous component $\sTnr$ of $\sTn$ is a finite dimensional weight supermodule, which will be proved in Section 11 to be isomorphic to the regular representation of the quantum queer  Schur superalgebra (Theorem \ref{reaque}). In Section 7, we further extend the $\qUq$-action to the formal power series analog $\wsTn$ of $\sTn$. This is a $\qUq$-module without an inherited supermodule structure. We then extract a subspace $\sVn$ from $\wsTn$ which has an inherited superspace structure,  and display certain action formulas of generators (Theorem \ref{mulfor}) on the BLM type basis $\fkL$ of $\sVn$. 
 We prove in Section 8 that $\sVn$ is a $\qUq$-supermodule.
 
By analysing the action formulas, we use the order relation on the defining basis for $\sVn$ to determine in Section 9 the leading terms in the action formulas on the basis elements by the divided powers of generators. These are the key to a triangular relation between the monomial basis $\fkM$ and the BLM type basis $\fkL$ discussed in Section 10.  In this way, we prove that $\sVn$ is a cyclic $\qUq$-supermodule isomorphic to the regular representation of $\qUq$. In the last section, we prove that $\sTnr$  is isomorphic to the regular representation of the queer $q$-Schur superalgebra
$\mathcal Q_\up(n,r)$. Thus, we obtain new realisations for both $\qUq$ and $\mathcal Q_\up(n,r)$ (Theorems \ref{mthm} and \ref{Qnr}).


\section{The quantum  queer supergroup $\qUq$ and its weight supermodules}

Let $n\geq 2$ be a positive integer. The following definition is taken from \cite[Definition 1.1]{GJKKK} (cf. \cite{DW1}) for the algebra structure and \cite[\S 4]{Ol} or \cite[(1.2)]{GJKKK} for the coalgebra structure.
\begin{definition}\label{presentqUq}
The queer quantum supergroup $\qUq$ is a Hopf superalgebra over $\Q(\q)$ whose unital associative superalgebra is generated
by even generators: $\k^{\pm1}_i, \e_j, \f_j$, and odd generators: $\k_{\bi}, \e_{\bj},\f_{\bj},$ for $1\leq i\leq n, 1\leq j\leq n-1$, subject to the following relations:
\begin{itemize}
 \item[(QQ1)]
 $\k_i\k^{-1}_i=1=\k^{-1}_i\k_i$,~
 $\k_i\k_j=\k_j\k_i,~ \k_i\k_{\bar j}=\k_{\bj}\k_i,$

\noindent $\k_{\bar i}\k_{\bar j}+\k_{\bar j}\k_{\bar i}=\ds2\delta_{ij}\frac{\k_i^2-\k_i^{-2}}{\q^2-\q^{-2}}$ for all $1\leq i,j\leq n$;

\vspace{0.1in}

 \item[(QQ2)]
 $\k_i\e_j=\q^{\de_{i,j}-\de_{i,j+1}}\e_j\k_i,\quad \k_i\f_j=\q^{-(\de_{i,j}-\de_{i,j+1})}\f_j\k_i,$

\noindent  $\k_i\e_{\bar j}=\q^{\de_{i,j}-\de_{i,j+1}}\e_{\bar j}\k_i,\quad \k_i\f_{\bar j}=\q^{-(\de_{i,j}-\de_{i,j+1})}\f_{\bar j}\k_i$,

\noindent
 for all $1\leq i,j\leq n,j\neq n$;

\vspace{0.1in}

 \item[(QQ3)] $\k_{\bi}\e_j=\e_j\k_{\bi}$, $\k_{\bi}\f_j=\f_j\k_{\bi}$, where $i\neq j, j+1$;

 \noindent $\k_{\bj}\e_j-\q\e_j\k_{\bj}=\e_{\bj}\k_j^{-1},\quad \q\k_{\ol{j+1}}\e_{j}-\e_{j}\k_{\ol{j+1}}=-\k^{-1}_{j+1}\e_{\ov{j}},$

 \noindent  $\k_{\bj}\f_j-\q\f_j\k_{\bj}=-\f_{\bj}\k_j,\quad \q\k_{\ol{j+1}}\f_{j}-\f_{j}\k_{\ol{j+1}}=\k_{j+1}\f_{\ov{j}},$




\vspace{0.1in}

 \item[(QQ4)]
 $\e_i\f_j-\f_j\e_i=\ds\delta_{ij}\frac{\k_i\k^{-1}_{i+1}-\k^{-1}_{i}\k_{i+1}}{\q-\q^{-1}}$,




\noindent  $\e_i\f_{\bj}-\f_{\bj}\e_i=\ds\delta_{ij}(\k^{-1}_{i+1}\k_{\bi}-\k_{\ov{i+1}}\k^{-1}_i)$,


\noindent  $\e_{\bi}\f_j-\f_j\e_{\bi}=\delta_{ij}(\k_{i+1}\k_{\bi}-\k_{\ov{i+1}}\k_i)$,  for all $1\leq i,j\leq n-1$;


\vspace{0.1in}

 \item[(QQ5)]
\noindent  $\e_i\e_j=\e_j\e_i$, $\f_i\f_j=\f_j\f_i$
for  $|i-j|>1$,

 \noindent $\e_i\e_{\bar i}=\e_{\bar i}\e_i$, $\f_i\f_{\bar i}=\f_{\bar i}\f_i,$

\noindent $\e_i\e_{i+1}-\q\e_{i+1}\e_i=\e_{\bi}\e_{\ov{i+1}}+\q\e_{\ov{i+1}}\e_{\bi},$

\noindent  $\q\f_{i+1}\f_i-\f_i\f_{i+1}=\f_{\bar i}\f_{\ov{i+1}}+\q\f_{\ov{i+1}}\f_{\bar i}$, where $1\leq i,j<n$;

 \vspace{0.1in}

  \item[(QQ6)]
 $\e_i^2\e_{j}-(\q+\q^{-1})\e_i\e_{j}\e_i+\e_{j}\e^2_i=0$,
 $\f_i^2\f_j-(\q+\q^{-1})\f_i\f_j\f_i+\f_j\f^2_i=0$,

\noindent $\e_i^2\e_{\ov{j}}-(\q+\q^{-1})\e_i\e_{\ov{j}}\e_i+\e_{\ov{j}}\e^2_i=0$,
$\f_i^2\f_{\ov{j}}-(\q+\q^{-1})\f_i\f_{\ov{j}}\f_i+\f_{\ov{j}}\f^2_i=0$,

 \noindent where $|i-j|=1$ and $1\leq i,j<n$.
\end{itemize}
The coalgebra structure on $\qUq$ has a comultiplication $\triangle$  defined  by the rules\footnote{See \cite[p. 199]{jc} for a comparison between this comultiplication and the usual one in the non-super case. 
The image on other odd generators can be found in 
\cite[p. 838]{GJKK} after adjusting notation as in \cite[(5.3)]{DW1}.}:
\begin{equation}\label{comult}
\begin{aligned}
& \Delta(\k_i)=\k_i\otimes \k_i,\quad
\Delta(\e_i)=\e_i\otimes \widetilde{\k}_i^{-1}+1\otimes \e_{i} ,\\
&   \Delta(\f_i)=\f_i\otimes 1+  \widetilde{\k}_i\otimes \f_i ,\quad
\Delta(\k_{\bar 1})=\k_{\bar 1}\otimes \k_{1}+\k_1^{-1}\otimes \k_{\bar 1},
\end{aligned}
\end{equation}
where $\widetilde{\k}_i=\k_i\k_{i+1}^{-1}.$
\end{definition}

\begin{remark}\label{gen rel}
(1) By identifying $\up,\sfK_i,\sfE_j,\sfF_j,\sfK_{\bar i},\sfE_{\bar j},\sfF_{\bar, j}$ with $q,q^{k_i},e_j,f_j,k_{\bar i},e_{\bar j},f_{\bar j}$, respectively, the relations in (QQ1)--(QQ6) are identical with those in \cite[(1.1)]{GJKKK}.

(2) As pointed out at the end of {\cite[Definition 1.1]{GJKKK}}, the algebra $\qUq$ is generated by even generators $\k^{\pm1}_i, \e_j, \f_j$, for $1\leq i\leq n,1\leq j<n$, together with the odd generators $\k_{\bar1}$.\footnote{This can be seen easily as follows. By (QQ3), we see that
$\e_{\bar1}=(\k_{\bar1}\e_1-\q\e_1\k_{\bar1})\k_1$ and $\f_{\bar1}=(-\k_{\bar1}\f_1+\q\f_1\k_{\bar1})\k_1^{-1}.$
Thus, by (QQ4), $\k_{\bar 2}=(\k^{-1}_{2}\k_{\bar1}-\e_1\f_{\bar1}+\f_{\bar 1}\e_1)\k_1$. Then, by (QQ3), $\e_{\bar2},\f_{\bar2}$ are defined. Inductively, we see that the missing generators can all be derived from the given relations.}

(3) The relations form a subset of relations given in \cite{DW1,DW2}; the missing relations can all be derived from the relations above (see \cite[Remark 1.2]{GJKKK2}).
\end{remark}

Recall from \cite[(5.6)]{DW1} that the superalgebra $\qUq$ admits a ring anti-involution $\Om$
given by 
\begin{equation}\label{Om}
\aligned
\Om(\up)&=\up^{-1},\quad\Om(\sfK_i)=\sfK_i^{-1},\quad\Om(\sfK_{\bar i})=\sfK_{\bar i},\\
\Om(\sfE_i)&=\sfF_i,\quad\Om(\sfF_i)=\sfE_i, \quad\Om(\sfE_{\bar i})=\sfF_{\bar i},\quad \Om(\sfF_{\bar i})=\sfE_{\bar i}.
\endaligned
\end{equation}
Recall also from \cite[(5.8) ]{DW1} the (even and odd) root vectors\footnote{We change the notation $X_{i,j}, \overline{X}_{i,j}$ there back to the usual notation $\Xije, \Xijo$ for root vectors to avoid notational confusion with the generators of the queer polynomial superalgebra in \S4.}
 $\Xije$ and $  \Xijo=\mathsf E_{i,\bar j},  i\neq j$, where, for $\varepsilon_{ij}:=1$ if $i<j$ and $-1$ if $i>j$,
 and  $|j-i|>1$,
$$\aligned
\OE_{h,h+1}&=\sfE_h, \OE_{h+1,h}=\sfF_h, \OE_{i,j}=\OE_{i,k}\OE_{k,j}-\up^{\varepsilon_{ij}}\OE_{k,j}\OE_{i,k};\\
\overline{\OE}_{h,h+1}&=\sfE_{\bar h}, \overline{\OE}_{h+1,h}=\sfF_{\bar h},
 \; \ol{\OE}_{i,j}=\begin{cases}\OE_{i,k}\ol{\OE}_{k,j}-\up^{}\ol{\OE}_{k,j}\OE_{i,k},&\text{ if }i<j;\\
\ol{ \OE}_{i,k}\OE_{k,j}-\up^{-1}\OE_{k,j}\ol{\OE}_{i,k},&\text{ if }i>j.\end{cases}
\endaligned.$$

We also write $\overline{\OE}_{i,i}:=\sfK_{\bar i}=\overline{\sfK}_i$ for consistency. We have by  \cite[(5.10)]{DW1}
\begin{equation}\label{Om2}
\Om(\OE_{i,j})=\OE_{j,i},\quad\Om(\overline{\OE}_{i,j})=\overline{\OE}_{j,i}\;\;(1\leq i\neq j\leq n).
\end{equation}

A $\qUq$-supermodule $M$ is called a {\it weight supermodule}, if $M$ has a weight space decomposition $M=\bigoplus_{\la\in\Z^n}M_\la$ where
$$M_\la=\{m\in M\mid \sfK_i.m=\up^{\la_i}m\;\forall 1\leq i\leq n\}.$$
If $0\neq x\in M_\la$, we write $\wt(x)=\la$. Define the set of weights of $M$ by 
$$\wt(M):=\{\la\in\Z^n\mid M_\la\neq0\}.$$ Call $M$ a {\it polynomial weight supermodule} if $M$ is a weight module and $\wt(M)\subseteq \N^n$.

The following fact will be used in the last section.
\begin{lemma}\label{poly2}
Let $M$ be a polynomial weight supermodule and $\la\in\wt(M)$. Then
$\sfE_h.M_\la=0$ if $\la_{h+1}=0$ and $\sfF_h.M_\la=0$ if $\la_{h}=0$.
\end{lemma}

\begin{proof}If $\sfE_i.M_\la\neq0$, then $\sfE_i.M_\la\subseteq M_{\la+\al_i}$, where $\al_i=\bfe_i-\bfe_{i+1}$. However, $\la+\al_i$ is not a polynomial weight if $\la_{i+1}=0$. Hence, $\sfE_i.M_\la=0$ must be true. The proof for the other case is similar.
\end{proof}

\noindent
{\bf Some notation.} Let $\sZ=\mathbb Z[\up,\up^{-1}]$ be the integral Laurent polynomial ring.
For $c\in\Z$, let $[c]=\ds\frac{\q^c-\q^{-c}}{\q-\q^{-1}}$ and 
define, for $m\geq 1$,
$$
[m]^!=[m] [m-1]\cdots[1],\quad [0]^!=1.
$$
and
$$
\begin{bmatrix}
c\\m
\end{bmatrix}
=\frac{[c][c-1]\cdots[c-m+1]}{[m]^!},\quad
\begin{bmatrix}
c\\0
\end{bmatrix}
=1.
$$
Generally, for an element $Z$ in an associative $\Q(\q)$-algebra $\mc A$ and $m\in\N$, define its (quantum) divided powers $Z^{(m)}$ by setting
$$
Z^{(m)}=\frac{Z^m}{[m]^!}.
$$
If $Z$ is invertible, define, for $1\leq i,j\leq n$ and $t\geq 1,c\in\Z$,

\begin{equation}\label{Kt}
\aligned
\begin{bmatrix}
Z;c\\ t
\end{bmatrix}
&=\prod^t_{s=1}\frac{Z\q^{c-s+1}-Z^{-1}\q^{-c+s-1}}{\q^s-\q^{-s}},
\quad \text{and }\;\begin{bmatrix}
Z;c\\ 0
\end{bmatrix}=1.
\endaligned
\end{equation}

Let $M_n(\N)$ be the set of $n\times n$ matrices over non-negative integers, and let
\begin{equation}\label{MnNZ}
\aligned
&M_n(\N|\N)=\{(A^0|A^1)\mid A^0, A^1\in M_n(\N)\},\\
&M_n(\N|\Z_2)=\{(A^0|A^1)\mid A^0\in M_n(\N),A^1\in  M_n(\Z_2)\},\\
&M_n(\bN|\bZ_2)_r=\{A\in M_n(\bN|\bZ_2)\mid \Sigma_{i,j}(a_{i,j}^0+a_{i,j}^1)=r\}.
\endaligned
\end{equation}
Here the two $n\times n$ matrices $A^0,A^1$ have the form $A^0=(a_{i,j}^0)$ and $A^1=(a_{i,j}^1)$,
$\Z_2=\{0,1\}$ is regarded as a subset\footnote{When $\Z_2$ is used to define a superspace, it is regarded as an abelian group.} of $\N$, and $|$ is used to separate the even and odd parts in a  superstructure. We may identify $M_n(\N|\N)$ as the set of all $n\times 2n$ matrices. Thus,
 $a_{i,j}^0$ (resp. $a_{i,j}^1$) is the $(i,j)^0$-entry or $(i,j)$-entry (resp. $(i,j)^1$-entry, or $(i,\bar j)$-entry) of $A$, where\footnote{The reader should not confuse this notation with the subscripts in $K_{\bar i}, E_{\bar h}, F_{\bar h}$ in Definition \ref{presentqUq}, where we didn't assume $\bar i=n+i$. However, it is not harmful to understand $K_{\bar i}, E_{\bar h}, F_{\bar h}$ as $K_{n+i}, E_{n+ h}, F_{n+ h}$, or $\overline{K}_{i}, \overline{E}_{h}, \overline{F}_{h}$. The latter notation was used in \cite{DW1}.} 
\begin{equation}\label{bar on i}
\bar j=n+j\text{ for all }1\leq j\leq n.
\end{equation}

For convenience, we sometimes identify a matrix $A=(A^0|A^1)\in M_n(\N|\N)$ with the {\it square} matrix 
\begin{equation}\label{squareA}
\sqA=\bigg({A^0\;A^1\atop A^1\;A^0}\bigg).\end{equation}

\section{A PBW type basis and the Lusztig type form $\qUqZ$}

We now define a PBW type basis $\frak B$ for $\qUq$. Let
\begin{equation}\label{MnNZ'}
\aligned
M_n(\N|\N)'&=\{((a_{i,j}^0)|(a_{i,j}^1))\in M_n(\N|\N)\mid a^0_{ii}=0,\;\forall 1\leq i\leq n\}\\
M_n(\N|\Z_2)'&=M_n(\N|\N)'\cap M_n(\N|\Z_2).
\endaligned
\end{equation}
Linearly order the index set
$\{(i,j),(i,\bar j)\}_{ 1\leq i,j\leq n}\backslash\{(i,i)\}_{ 1\leq i\leq n}$  for root vectors as follows:
\begin{equation}\label{PBWorder}
\aligned
&\underline{(n-1,n),\ldots,(1,n),(1,\bar n),\ldots,(n,\bar n)},\cdots, \underline{(1,2), (1,\overline{2}),\cdots,(n,\overline{2})},\underline{(1,\bar 1),\cdots(n,\bar 1)},\\
&\underline{(2,1),(3,1),\cdots,(n,1)},\underline{(3,2),\cdots,(n,2)},\cdots,\underline{(n,n-1)}
\endaligned
\end{equation}

If we call the $j$-th column of the upper (resp., lower) triangular part of $A^0$ the $j^+$-column (resp., $j^-$-column) of $A$, and call the $j$-th column of $A^1$ the $\bar j$-th column of $A$. Then the ordering in \eqref{PBWorder} is arranged by $(j^+_\uparrow,\bar j_\downarrow)$-column indices for 
$j=n,n-1,\ldots,1$ and follows by $j^-_\downarrow$-column indices for $j=1,2,\ldots, n-1$. Here the arrows $\uparrow, \downarrow$ indicate the column indices are read upwards, downwards, respectively.

As an example, we indicate the order in the following $3\times 3$ matrix
$$A=\left(\left. \begin{array}{ccc}
         *& 6 & 2 \\
         13 & *& 1 \\
         14 & 15 & *\\
       \end{array}\right|\right.
\left.\begin{array}{ccc}
         10& 7&3 \\
         11& 8&4 \\
         12 & 9 &5\\
       \end{array} \right).
$$

Associated with $ A\in M_n(\bN|\bZ_2)',\,\bfj\in \mathbb Z^{n}$, we define 
\begin{equation}\label{PBW}
{\mathfrak b}^{A,\bfj}=\sfK^\bfj \OE_{A}^{n^+,\bar n}\cdots \OE_{A}^{2^+,\bar 2}\OE_A^{1^+,\bar1}\OE_{A}^{1^-}\cdots \OE_{A}^{(n-1)^-},
\end{equation}
where $\sfK^\bfj=\sfK_1^{\bfj_1}\sfK_2^{\bfj_2}\cdots \sfK_n^{\bfj_n}$ is called a $\sfK$-{\it segment} of ${\mathfrak b}^{A,\bfj}$,
\begin{equation}\label{EAj+}
\OE_{A}^{j^+,\bar j}=\begin{cases} \overline{\OE}_{1,1}^{a_{1,1}^1}\overline{\OE}_{2,1}^{a_{2,1}^1}\cdots \overline{\OE}_{n,1}^{a_{n,1}^1}, &\text{ if }j=1;\\
\OE_{j-1,j}^{(a_{j-1,j}^0)}\OE_{j-2,j}^{(a_{j-2,j}^0)} \cdots \OE_{1,j}^{(a_{1,j}^0)}
 \overline{\OE}_{1,j}^{a_{1,j}^1}\overline{\OE}_{2,j}^{a_{2,j}^1}\cdots \overline{\OE}_{n,j}^{a_{n,j}^1},&\text{ if }2\leq j\leq n.
 \end{cases}
 \end{equation}
is called a $(j^+,\bar j)$-{\it segment}, where the product $\OE_{j-1,j}^{(a_{j-1,j}^0)}\OE_{j-2,j}^{(a_{j-2,j}^0)} \cdots \OE_{1,j}^{(a_{1,j}^0)}
$ is called a $j^{+}$-{\it segment} and the rest product is called a $\bar j$-segment,\footnote{We dropped the brackets $(\;\;)$ since $a^1_{i,j}\in\Z_2$. This simplified notation is a good reminder of the condition.} and 
\begin{equation}\label{EAj-}
\OE_A^{j^-}=\OE_{j+1,j}^{(a_{j+1,j}^0)}\OE_{j+2,j}^{(a_{j+2,j}^0)}\cdots \OE_{n,j}^{(a_{n,j}^0)}\;(1\leq j\leq n-1)
\end{equation}
 is called a $j^-$-{\it segment}. (We ignored the subscript arrows for clarity.)
Note that $\OE_{A}^{j^+,\bar j},\OE_A^{j^-} $ are defined on each underlined section in \eqref{PBWorder}. Hence,
$\fkb^{A,\bfl}$ is a product taken over the ordering \eqref{PBWorder}.

\begin{remark}\label{rule}
 We remark the general rules for the product \eqref{PBW}. For two positive (resp. negative) even root vectors, the one with a larger column index is on the left (resp. right); if they have the same column index, then the one with a larger row index is on the left (resp. right). Odd root vectors are always put on the right of the even positive root vectors with the same column index and the row indices are increasing from left to right.
\end{remark}

For example, if $n=3$ and 
\begin{equation}\label{3by3}
\begin{aligned}
A=\left(\left. \begin{array}{ccc}
         0        &      a_{12}^0   & a_{13}^0\\
         a_{21}^0 &      0 & a_{23}^0 \\
         a_{31}^0 & a_{32}^0 & 0
       \end{array}\right|\right.
\left.\begin{array}{ccc}
         a_{11}^1 & a_{12}^1 & a_{13}^1 \\
         a_{21}^1 & a_{22}^1 & a_{23}^1 \\
         a_{31}^1 & a_{32}^1 & a_{33}^1
       \end{array} \right),
\end{aligned}
\end{equation}
then
$$\aligned
\fkb^{A,\bfl}&=\OE_{2,3}^{(a_{2,3}^0)}\OE_{1,3}^{(a_{1,3}^0)}\ol{\OE}_{1,3}^{a_{1,3}^1}\ol{\OE}_{2,3}^{a_{2,3}^1}\ol{\OE}_{3,3}^{a_{3,3}^1}\cdot
\OE_{1,2}^{(a_{1,2}^0)}\ol{\OE}_{1,2}^{a_{1,2}^1}\ol{\OE}_{2,2}^{a_{2,2}^1}\ol{\OE}_{3,2}^{a_{3,2}^1}\\
&\quad\,\cdot\ol{\OE}_{1,1}^{a_{1,1}^1}\ol{\OE}_{2,1}^{a_{2,1}^1}\ol{\OE}_{3,1}^{a_{3,1}^1}\cdot
\OE_{2,1}^{(a_{2,1}^0)}\OE_{3,1}^{(a_{3,1}^0)}\cdot\OE_{3,2}^{(a_{3,2}^0)}.
\endaligned$$

\begin{lemma}\label{The PBW basis}
 The set
$$\mathfrak B=
\{\fkb^{A,\bfj}  \mid A\in
M_n(\bN|\bZ_2)',\,\bfj\in \mathbb Z^{n}\}
$$
forms a basis, a {\rm PBW} type basis,  for  $\qUq$.
\end{lemma}
\begin{proof}  By \cite[(5.2), Lem. 5.7]{DW1}, each of the $\OE_{i,j},\overline{\OE}_{i,j}$ is a scalar multiple of $\sfK_j$ times a corresponding $L_{i,j}$. Since $\sfK_i$ commutes with every root vector up to a scalar (see \cite[(5.9)]{DW1}, it follows that 
$$\OE_{A}^{j^+,\bar j}=f(\up)\sfK^{\bfj'}L_{j-1,j}^{(a_{j-1,j}^0)}L_{j-2,j}^{(a_{j-2,j}^0)} \cdots L_{1,j}^{(a_{1,j}^0)}
 L_{-1,j}^{a_{1,j}^1}L_{-1,j}^{a_{2,j}^1}\cdots L_{-n,j}^{a_{n,j}^1}$$
 for some nonzero $f(\up)\in\mathbb Q(\up),\bfj'\in\mathbb Z^{m+n}$. 
 Now the assertion follows from \cite[Thm 6.2]{Ol} or \cite[Proposition 5.5]{DW1}.
\end{proof}

We now introduce an integral basis for the Lusztig type form
of $\qUq$. Following \cite[\S8]{DW1}, let $\qUqZ$ be the $\sZ$-subalgebra of $\qUq$ generated by
\begin{equation}\label{Z-generators}
\sfK_i,\left[\begin{matrix}\sfK_i\\t\end{matrix}\right], \sfE_j^{(m)},\sfF_j^{(m)}, \sfK_{\bar i}, \sfE_{\bar j}, \sfF_{\bar j}\quad
(1\leq i,j\leq n, j\neq n, t,m\in\mathbb N),
\end{equation} 
where $\left[\begin{matrix}\sfK_i\\t\end{matrix}\right]=\left[\begin{matrix}\sfK_i;0\\t\end{matrix}\right]$ as defined in \eqref{Kt}.  By the definition of root vectors and \cite[(5.8), Prop. 7.4(3)]{DW1}, all root vectors $\OE_{i,j}, \ol{\OE}_{i,j}$ and their divided powers $\OE_{i,j}^{(m)}$ are in $\qUqZ$. Moreover, if we introduce the following degree function:
$$\deg(\OE_{i,j}^{(m)})=2m|i-j|,\;\;\deg( \ol{\OE}_{i,j})=2|i-j| (i\neq j);\quad \deg(\sfK_{\bar i})=1,\;\;\deg(\sfK_i)=0,$$
then, for each $A\in
M_n(\bN|\bZ_2)'$,
$$\deg(\fkb^{A,\bfj})=\sum_{i=1}^na_{i,i}^1+\sum_{1\leq i\neq j\leq n}2(a_{i,j}+a_{j,i})|i-j|=:\deg(A).$$
The following result is stated in \cite[Remark 8.3]{DW1}.

\begin{lemma}Let $\mathcal G=\{\OE_{i,j}^{(m)},\ol{\OE}_{i,j},\sfK_{\bar i}\mid
1\leq i,j\leq n, i\neq j, m\in\mathbb N\}$. For any two elements $\sfX,\sfY\in\mathcal G$, there exists  some $a\in\mathbb Z$ such that
$$\sfX\sfY=\up^a\sfY\sfX+f,$$
where $f$ is a linear combination of monomials of degree $<\deg(\sfX\sfY)$.
\end{lemma}

For $\tau,\la\in\mathbb N^n$, let 
\begin{equation}\label{Kla}
\sfK^\tau=\prod_{i=1}^n\sfK_i^{\tau_i}\quad\text{ and }\quad\left[\begin{matrix}\sfK\\\la\end{matrix}\right]=\prod_{i=1}^n\left[\begin{matrix}\sfK_i\\\la_i\end{matrix}\right].
\end{equation}

\begin{proposition}\label{ZbasisB}
The set
$$\mathfrak B_\sZ=\left\{\sfK^\tau \left[\begin{matrix}\sfK\\\la\end{matrix}\right]\mathfrak b^{A,\bfl}\,\bigg| \,A\in M_n(\N|\Z_2)', \la\in\mathbb N^n, \tau\in(\Z_2)^n\right\}$$
forms a $\sZ$-basis for $\qUqZ$.
\end{proposition}
\begin{proof} Clearly, span$_\sZ(\mathfrak B_\sZ)\subseteq \qUqZ$. Conversely, by the commutation formulas in the lemma above, every element in the basis given in \cite[Prop. 8.2(2)]{DW1} can be written as a linear combination of basis elements in  $\mathfrak B_\sZ$. Thus,
 $\qUqZ\subseteq\text{span}_\sZ(\mathfrak B_\sZ)$.
\end{proof}

\begin{remark}In \cite[Prop. 8.2]{DW1}, Wan and the first author constructed PBW bases for the Lusztig form $\qUqZ$ which are compatible with the triangular decomposition $\qUqZ=U_{\up,\sZ}^-\otimes U_{\up,\sZ}^0\otimes U_{\up,\sZ}^+$. In other words, these bases are products of bases for the triangular parts. This is not the case for bases $\mathfrak B$ and $\mathfrak B_\sZ$.
\end{remark}

In order to construct a monomial basis for $\qUq$ via the generators in \eqref{Z-generators}, we need a new order relation on the basis $\mathfrak B$ with respect to which the transition matrix between the two bases is upper triangular. This requires more accurate analysis on certain commutation formulas.

Consider the lexicographical order on $\N^m$: for $\bfa,\bfb\in\N^m$, 
\begin{equation}\label{lexic}
\bfa<\bfb\iff a_1<b_1 \text{ or }a_i=b_i \;(1\leq i<k) \text{ and }a_k<b_k, \text{ for some }k>1.
\end{equation}
To specify the order more precisely, we say that $\bfa<\bfb$ {\it at the $k$th component}.

Define a map 
$$\vec{(\;\;)}:\ M_n(\bN|\bZ_2)'  \longrightarrow \bN^{2n^2-n}, A\longmapsto \vec A,$$ where
\begin{equation}\label{vecA}
\begin{aligned}
\vec A&= (\underline{a_{n,n}^1,\cdots, a_{1,n}^1, a_{1,n}^0,\cdots,a_{n-1,n}^0},
     \underline{a_{n,n-1}^1,\cdots, a_{1,n-1}^1, a_{1,n-1}^0,\cdots, a_{n-2,n-1}^0}, \\
    &\cdots,\underline{a_{n,2}^1\cdots a_{1,2}^1, a_{1,2}^0},\underline{a_{n,1}^1,\cdots, a_{1,1}^1}, \underline{a_{n,1}^0,\cdots a_{2,1}^0},\underline{ a_{n,2}^0,\cdots, a_{3,2}^0},\cdots,  \underline{a_{n,n-1}^0}).
\end{aligned}
\end{equation}
Here we divide the sequence $\vec A$ in $2n-1$ sections as underlined above. Note that the entry ordering in $\vec A$ can be obtained from the ordering in \eqref{PBWorder} by reversing every underlined section. 

 We will label the sections by their column indices so that $\vec A$ is the sequence consisting of $({\bar j}_\uparrow, j^{+}_\downarrow)$-columns for $j=n,n-1,\ldots,1$, followed by $j^-_\uparrow$-columns for $j=1,2,\ldots,n-1$. Here the arrows $_{\uparrow,\downarrow}$ indicate the direction we read the entries of the column. For example, 
 a $({\bar j}_\uparrow, j^{+}_\downarrow)$-column consists of the $j$-th column of $A^1$, reading from bottom to top, and the $j$-th column of the upper triangular part of $A^0$, reading downwards. (Note that,  for $j=1$, the $j_\downarrow^+$-column is empty.)
 

For example, if $n=3$, then $\vec A=(a_1,a_2,\ldots,a_{15})$ where 
\begin{equation}\label{order A}
A=\left(\left. \begin{array}{ccc}
         * & a_9 & a_4 \\
         a_{14} & * & a_5 \\
         a_{13} & a_{15} & *
       \end{array}\right|\right.
\left.\begin{array}{ccc}
         a_{12} & a_8 & a_3 \\
         a_{11} & a_7 & a_2 \\
         a_{10} & a_6 & a_1
       \end{array} \right).
\end{equation}
\begin{conventions}\label{conv}
We say that $\vec A$ (or $A$) {\it starts at $a^i_{h,k}$} ($i\in \Z_2$) if $a^i_{h,k}>0$ and entries of $\vec A$ before $a^i_{h,k}$ are all zero. We say that $\vec A$ (or $A$) has the {\it leading entry}  $a^i_{h,k}$ if $\vec A$ starts at $a^i_{h,k}$.
We say $\vec A$ (or $A$) {\it starts after $a^i_{h,k}$} if  $a^i_{h,k}=0$ and entries of $\vec A$ before $a^i_{h,k}$ are all zero. In other word, the leading entry of $A$ in this case occurs after the entry 
$a^i_{h,k}$. A column $j^-, j^+$ or $\bar j$ is a leading column if it contains the leading entry of $\vec A$.
\end{conventions}
We extend the map $\vec{(\;\;)}$ to $M_n(\bN|\bZ_2)$ by setting $\vec A:=\vec A'$, where $A'\in M_n(\bN|\bZ_2)'$
is obtained from $A=(A^0|A^1)$ by replacing the diagonal of $A^0$ with zeros.
Define the pre-order $\preceq$ on $M_n(\bN|\bZ_2)$ by setting, for $A, B\in M_n(\bN|\bZ_2)$, 
\begin{equation}\label{prec}
B\preceq A\iff \vec B \leq \vec A,
\end{equation}
 where $``\leq"$ is the lexicographical order defined in \eqref{lexic}. We will also say below that $A\prec B$ {\it at the $(i,j)$-entry} instead of ``at the $k$th component''.

In order to establish a triangular relation between the PBW basis and a monomial basis relative to the order $\preceq$ on $M_n(\bN|\bZ_2)$, we need to rewrite certain commutation formulas for root vectors, given in \cite{DW1}, in terms of the PBW basis elements. In each formula below, the leading term in the right hand sides is the largest element relative to the order $\preceq$. In other words, it has the form:
$$\fkb^{A,\bfj}+(\text{lower terms}),$$ 
where``lower terms'' means a linear combination of $\fkb^{B,\bfj'}$ with $B\prec A$
(i.e., $B\preceq A$ but $B\neq A$).
\begin{lemma}\label{comm formulas}
\begin{enumerate}
\item If $j<k<l$ and $s,t\in\mathbb N$, then
$$
\begin{cases}
(1^+)& \OE_{j,k}^{(s)}\OE_{k,l}^{(t)}=\OE_{k,l}^{(t-s)}\OE_{j,l}^{(s)}+(\text{lower terms}),\;\;\text{ if }s\leq t;\\
(1^-)& \OE_{l,k}^{(t)}\OE_{k,j}^{(s)}=\OE_{k,j}^{(s-t)}\OE_{l,j}^{(t)}+(\text{lower terms}),\;\;\text{ if }t\leq s.
\end{cases}$$
\item For $a<i<j$ or $i<j<a$, $\overline{\sfK}_a\OE_{i,j}^{(m)}=\OE_{i,j}^{(m)}\overline{\sfK}_a$, and
 $$\overline{\sfK}_i\OE_{i,j}^{(m)}=\up^m\sfK_i^{-1}\OE_{i,j}^{(m-1)}\overline{\OE}_{i,j}+\up^m\OE_{i,j}^{(m)}\overline{\sfK}_i.$$
 \item For $1\leq h<n$, $1\leq i<j\leq n$,  we have
 $$\sfF_h\OE_{i,j}^{(m)}=\begin{cases}\sfE_h^{(m)}\sfF_h-\Big[{\widetilde \sfK_h;1-m\atop 1}\Big]\sfE_h^{(m-1)},&\text{ if }i=h, j=h+1;\\
  \OE_{h,j}^{(m)}\sfF_h+\up^{2-m}\sfK_h\sfK_{h+1}^{-1}\OE_{h+1,j} \OE_{h,j}^{(m-1)},&\text{ if }i=h<h+1<j;\\
    \OE_{i,h+1}^{(m)}\sfF_h-\sfK_h^{-1}\sfK_{h+1} \OE_{i,h+1}^{(m-1)}\OE_{i,h},&\text{ if }i<h<h+1=j;\\
  \OE_{i,j}^{(m)}\sfF_h,&\text{ otherwise.}\\
\end{cases}$$
 
\item For $i<j$ and $i<k$,
$$\OE_{j,i}\overline{\OE}_{i,k}=\begin{cases}
\sfK_i\ol{\sfK}_{j}+\overline{\OE}_{i,j}\OE_{j,i}-\sfK_j\ol{\sfK}_{i},&\text{ if }k=j;\\
\overline{\OE}_{j,k}\sfK_{j}^{-1}\sfK_i+\overline{\OE}_{i,k}\OE_{j,i},&\text{ if }j<k;\\
\sfK_i\sfK_k\overline{\OE}_{j,k}+\overline{\OE}_{i,k}\OE_{j,i}+(\up^{-1}-\up)\sfK_k\OE_{j,k}\ol{\sfK}_{i},&\text{ if }j>k.
\end{cases}
$$
\end{enumerate}
\end{lemma}
\begin{proof}The first assertion in (1) follows from \cite[Prop.~7.4(3)]{DW1}. More precisely, if $1\leq s\leq t$, then we have
\begin{equation}\label{prv}
\OE_{j,k}^{(s)}\OE_{k,l}^{(t)}=\OE_{k,l}^{(t-s)}\OE_{j,l}^{(s)}+\sum_{a=0}^{s-1}\up^{(s-a)(t-a)}
\OE_{k,l}^{(t-a)}\OE_{j,l}^{(a)}\OE_{j,k}^{(s-a)},
\end{equation}
Here, we have singled out the leading term relative to $\preceq$ and all other terms in the right hand side are PBW basis elements as defined in \eqref{PBW}. This proves (1$^+$). By applying $\Omega$ to  \cite[Prop.~7.4(3)]{DW1}, a similar augment gives (1$^-$).

The first commutation relation in (2) follows from \cite[Proposition 7.10(1)]{DW1}, while the second follows from \cite[Proposition 7.10(2)]{DW1} and \cite[Proposition 7.8(1)]{DW1}. 

The first and fourth cases in (3) are special cases of \cite[Prop. 7.6]{DW1}(1)\&(2), respectively.
By  \cite[Prop. 7.6]{DW1}(3)\&(4),
$$\sfF_h\OE_{i,j}^{(m)}=\begin{cases}
  \OE_{h,j}^{(m)}\sfF_h+\up^{1-m}\OE_{h+1,j}\sfK_h\sfK_{h+1}^{-1} \OE_{h,j}^{(m-1)},&\text{ if }i=h<h+1<j;\\
    \OE_{i,h+1}^{(m)}\sfF_h-\up^{m-1}\sfK_h^{-1}\sfK_{h+1} \OE_{i,h}\OE_{i,h+1}^{(m-1)},&\text{ if }i<h<h+1=j.\\
\end{cases}$$
Here we need to apply $\Om$ in \eqref{Om} to \cite[Prop. 7.6(3)]{DW1} for $m=1,s\geq 1$ to obtain the first case here. By \cite[(5.9)]{DW1}, $\OE_{h+1,j}\sfK_h\sfK_{h+1}^{-1}=\up \sfK_h\sfK^{-1}_{h+1}\OE_{h+1,j}$ and, by \cite[Prop. 7.4(2)]{DW1}, 
$\OE_{i,h}\OE_{i,h+1}^{(m-1)}=\up^{-(m-1)}\OE_{i,h+1}^{(m-1)}\OE_{i,h}$. So (3) follows.

To see (4),  we extract the first, third and sixth cases from \cite[Proposition 6.4 (2)]{DW1},
and applying $\Omega$ to them gives
$$\overline{\OE}_{i,k}\OE_{j,i}=\begin{cases}
\OE_{j,i}\overline{\OE}_{i,j}-\sfK_i\ol{\sfK}_{j}+\ol{\sfK}_{i}\sfK_j,&\text{ if }k=j;\\
\OE_{j,i}\overline{\OE}_{i,k}-\overline{\OE}_{j,k}\sfK_{j}^{-1}\sfK_i,&\text{ if }i<j<k;\\
\OE_{j,i}\overline{\OE}_{i,k}-\sfK_i\sfK_k\overline{\OE}_{j,k}-(\up^{-1}-\up)\sfK_k\ol{\sfK}_{i}\OE_{j,k},&\text{ if }i<k<j.\\
\end{cases}
$$
Rewriting gives the required formula.
\end{proof}

\setcounter{equation}{0}

\section{Monomial bases for $\qUq$ and $\qUqZ$}
In this section,, we will construct a new basis for $\qUq$ in terms of the generators $\sfE_h^{(m)},\sfF_h^{(m)}, \sfK^{\pm1}_i,\sfE_{\bar h},\sfF_{\bar h},\sfK_{\bar 1}$; see Remark \ref{gen rel}. 
 We will first define some monomials $\fkm^{A,\bfj}$ to form the set
$\mathfrak{M}=\{\fkm^{A,\bfj}\mid A\in M_n(\bN|\bZ_2)',\bfj\in\Z^n\}$,
and then to establish a triangular relation with respect to the ordering $\preceq$ defined in \eqref{prec} between $\frak M$ and the PBW basis $\frak B$ in Lemma \ref{The PBW basis}.

For any $A=(A^0|A^1)\in M_n(\bN|\bZ_2)'$ with $A^0=(a_{i,j}^0)$, $A^1=(a_{i,j}^1)$, and $1\leq i, j\leq n$, define odd monomials $\fkF^1_{i,j}=\fkF^1_{i,j}(A)$ by
\begin{equation}\label{Fij}
\fkF^1_{1,j}=\k_{\bar 1}^{a_{1,j}^1},\quad\fkF^1_{i,j}=\f_{ i-1}^{a_{i,j}^1} \cdots \f_{ 1}^{a_{i,j}^1}  \k_{\bar 1}^{a_{i,j}^1}\;\;(2\leq i\leq n).
\end{equation}
and even monomials $\fkE_{j-1}^0=\fkE_{j-1}^0(A)$ and $\fkF_j^0=\fkF_j^0(A)$ by $\fkE_{0}^0=1$,
\begin{equation}\label{MonBs0}
\aligned
\fkE_{j-1}^0&= \e_{1}^{(a_{1,j}^0+|\bfc_j^1|)} \e_{2}^{(a_{1,j}^0+a_{2,j}^0+|\bfc_j^1|)} \cdots
\e_{j-1}^{(\sum_{1\leq s\leq j-1}a_{sj}^0+|\bfc_j^1|)}\;\;(2\leq j\leq n),\\
\fkF_j^0&=\sfF_{n-1}^{(a_{n,j}^0)} \sfF_{n-2}^{(a_{n-1,j}^0+a_{n,j}^0)}\cdots \sfF_j^{(\sum_{j+1\leq s\leq n}a_{s,j}^0)}
\;\;(1\leq j\leq n-1),\\
\endaligned
\end{equation}
where $\bfc_j^1$ is the $j$th column of $A^1$, and $|\bfc_j^1|=\sum_{1\leq s\leq n}a_{s,j}^1$. 

Define, for $A\in M_n(\bN|\bZ_2)',\bfj\in\Z^n$,
\begin{equation}\label{MonBs}
\aligned
\fkm^{A,\bfj}&=\sfK_1^{j_1}\cdots \sfK_n^{j_n}\fkm^{A,\bfl},\quad\text{ where}\\
\fkm^{A,\bfl}&=\bigg(\prod_{j=1}^n(\fkF_{1,n-j+1}^1\fkF_{2,n-j+1}^1\cdots \fkF_{n,n-j+1}^1\fkE_{n-j}^0)\bigg)\cdot\fkF_{1}^0\fkF_2^0\cdots\fkF_{n-1}^0.
\endaligned
\end{equation}
Here the order in the product $\prod$ is ordered naturally $1,2,\ldots n$ from left to right. Note that each factor $\fkF_{1,j}^1\fkF_{2,j}^1\cdots \fkF_{n,j}^1\fkE_{j-1}^0$ corresponds to the ordered $(\bar j_\downarrow,j^+_\downarrow)$-column (on exponents!), interpreted similarly to the  $(\bar j_\uparrow,j^+_\downarrow)$-column in \eqref{vecA}, and the entire product $\fkm^{A,\bfl}$ is taken over the ordering in \eqref{vecA} on sections.

For example, if $A$ is given as in \eqref{3by3}, 
then
\begin{equation}\label{tdef}\notag
\begin{aligned}
\fkm^{A,{\bf0}}&=
(\k_{\bar 1}^{a_{13}^1}
\f_{ 1}^{a_{23}^1}\k_{\bar 1}^{a_{23}^1}
\f_{2}^{a_{33}^1}\f_{ 1}^{a_{33}^1}\k_{\bar 1}^{a_{33}^1})
(\e_1^{(a_{13}^0+a_{13}^1+a_{23}^1+a_{33}^1)}
\e_2^{(a_{13}^0+a_{23}^0+a_{13}^1+a_{23}^1+a_{33}^1)})\\
&\quad\;(\k_{\bar 1}^{a_{12}^1}
\f_{ 1}^{a_{22}^1}\k_{\bar 1}^{a_{22}^1}
\f_{2}^{a_{32}^1}\f_{ 1}^{a_{32}^1}\k_{\bar 1}^{a_{32}^1})
\e_1^{(a_{12}^0+a_{12}^1+a_{22}^1+a_{32}^1)}\\
&\quad\;(\k_{\bar 1}^{a_{11}^1}
\f_{ 1}^{a_{21}^1}\k_{\bar 1}^{a_{21}^1}
\f_{2}^{a_{31}^1}\f_{ 1}^{a_{31}^1}\k_{\bar 1}^{a_{31}^1})\\
&\quad\;
(\f_2^{(a_{31}^0)} \f_1^{(a_{21}^0+a_{31}^0)}) \f_2^{(a_{32}^0)}.
\end{aligned}
\end{equation}

Now we compute the leading terms of certain monomials when written as linear combinations of the PBW basis elements in $\mathfrak B$. In the following, a statement like ``a lower term having a smaller $j^+$-segment'' means that the matrix associated the lower term is less than (under $\preceq$) that of the leading term at an entry in the $j^+_\downarrow$-column.

Let
$$\up^{\mathbb Z}\sfK^{\mathbb Z^n}=\{\up^a\sfK^\bfj\mid a\in\mathbb Z,\bfj\in\mathbb Z^n\}.$$
\begin{lemma}\label{leater}
\begin{enumerate}
\item For $1\leq i<j,$ if $1\leq a_i\leq \cdots\leq a_{j-1}$, then 
$$\e_{i}^{(a_i)}\e_{i+1}^{(a_{i+1})}\cdots\e_{j-1}^{(a_{j-1})}
={\OE}_{j-1,j}^{(a_{j-1}-a_{j-2})}\cdots{\OE}_{i+1,j}^{(a_{i+1}-a_{i})}{\OE}_{i,j}^{(a_{i})}+ \mbox{(lower terms)},$$
where each lower term has a smaller $j^{+}$-segment.
\item For $1\leq i<j$, if $a_i\geq \cdots\geq a_{j-1}\geq 1,$ then 
$$\f_{j-1}^{(a_{j-1})}\cdots\f_{i+1}^{(a_{i+1})}\f_{i}^{(a_i)}
={\OE}_{i+1,i}^{(a_{i}-a_{i+1})}\cdots{\OE}_{j-1,i}^{(a_{j}-a_{j-1})}{\OE}_{j,i}^{(a_{j-1})}+ \mbox{(lower terms)},$$
where each lower term has a smaller $i^-$-segment  (or $i^-_\uparrow$-column).
\item For $i,j\geq 2$ and $1\leq a_1\leq \cdots\leq a_{j-1}$, then there exists $g\in \q^{\mathbb Z}\sfK^{\mathbb Z^n}$ such that 
$$\f_{i-1}\cdots\f_{1}\overline{\sfK}_{1}\e_{1}^{(a_1)}\cdots\e_{j-1}^{(a_{j-1})}
=g{\OE}_{j-1,j}^{(a_{j-1}-a_{j-2})}\cdots{\OE}_{2,j}^{(a_{2}-a_{1})}{\OE}_{1,j}^{(a_{1}-1)}\overline{\OE}_{i,j}+ \mbox{(lower terms)},$$
where each lower term has a smaller $(j^+,\bar j)$-segment  (or $(\bar j_\uparrow,j^+_\downarrow)$-column).\footnote{Note that the ordering in $\vec A$ sectionally reverses the ordering used in defining a PBW basis.} 
\end{enumerate}
\end{lemma}
\begin{proof}Since $\OE_h=\sfE_{h,h+1}$, repeatedly applying Lemma \ref{comm formulas}($1^+$) together with \eqref{prv} yields
(1). More precisely, multiplying \eqref{prv} on the left by $\OE_{i,j}^{(r)}$ with $i<j$ and $r\leq s\leq t$ and noting \cite[Prop. 7.4(1)]{DW1} yields
$$
\aligned
\OE_{i,j}^{(r)}\OE_{j,k}^{(s)}\OE_{k,l}^{(t)}&=\OE_{k,l}^{(t-s)}(\OE_{i,j}^{(r)}\OE_{j,l}^{(s)})+\sum_{a=0}^{s-1}\up^{(r-b)(s-b)}
\OE_{k,l}^{(t-a)}(\OE_{i,j}^{(r)}\OE_{j,l}^{(a)})\OE_{j,k}^{(s-a)}\\
&=\OE_{k,l}^{(t-s)}\OE_{j,l}^{(s-r)}\OE_{i,l}^{(r)}+\sum_{b=0}^{r-1}\up^{(s-a)(t-a)}\OE_{k,l}^{(t-s)}\OE_{j,l}^{(t-b)}\OE_{i,l}^{(b)}\OE_{i,j}^{(r-b)}\\
&\quad\;+\sum_{a=0}^{s-1}\OE_{k,l}^{(t-a)}\bigg(\sum_{c=0}^{\text{min}(r,a)}\OE_{j,l}^{(a-c)}\OE_{i,l}^{(c)}\OE_{i,j}^{(r-c)}\bigg)\OE_{j,k}^{(s-a)}.
\endaligned
$$
Now, by applying \eqref{prv} to $\OE_{i,j}^{(r-c)}\OE_{j,k}^{(s-a)}$ one more time,
we obtain a linear combination of some PBW basis elements $\fkb^{A,\bf0}$ as given in Lemma \ref{The PBW basis}, where $A^1=0$ and $A^0$ is a strictly a upper triangular matrix. Clearly, the $l^+$-segment 
of every $\fkb^{A,\bf0}$ contains the leading entry of $\vec A$ and their corresponding $l^+_\downarrow$-column are $(r,s-r,t-s)$ for the first term and  $(b,t-b,t-s)$ or $(c,a-c,t-a)$ for the other terms, where
$0\leq b\leq r-1$, $0\leq a\leq s-1,0\leq c\leq\text{min}(r,a)$. Hence, the first term is clearly the leading or the largest term relative to the lexicographic order, i.e., the order $\preceq$, and the lower terms has the form (a smaller $l^+$-segment)$\fkc$. 
The general case can be seen similarly or by an inductive argument. This proves (1).

By Lemma \ref{comm formulas}($1^-$),
the proof for assertion (2) is similar.

It remains to prove (3). By (1), we have
\begin{equation}\label{eee}
\e_{1}^{(a_1)}\cdots\e_{j-1}^{(a_{j-1})}={\OE}_{j-1,j}^{(a_{j-1}-a_{j-2})}\cdots{\OE}_{2,j}^{(a_{2}-a_{1})}{\OE}_{1,j}^{(a_{1})}+ \mbox{ (lower terms)},
\end{equation}
where each lower term has a smaller $j^{+}$-segment.
Lemma \ref{comm formulas}(2) gives
$$\aligned
\overline{\sfK}_{1}{\OE}_{j-1,j}^{(a_{j-1}-a_{j-2})}&\cdots{\OE}_{2,j}^{(a_{2}-a_{1})}{\OE}_{1,j}^{(a_{1})}
={\OE}_{j-1,j}^{(a_{j-1}-a_{j-2})}\cdots{\OE}_{2,j}^{(a_{2}-a_{1})}\overline{\sfK}_{1}{\OE}_{1,j}^{(a_{1})}\\
&={\OE}_{j-1,j}^{(a_{j-1}-a_{j-2})}\cdots{\OE}_{2,j}^{(a_{2}-a_{1})}(\q^{a_1}\sfK_1^{-1}{\OE}_{1,j}^{(a_{1}-1)}\overline{\OE}_{1,j}+\up^{a_1}\OE_{1,j}^{(a_1)}\overline{\sfK}_1)\\
&=\q^{a_1}{\OE}_{j-1,j}^{(a_{j-1}-a_{j-2})}\cdots{\OE}_{2,j}^{(a_{2}-a_{1})}\sfK_1^{-1}{\OE}_{1,j}^{(a_{1}-1)}\overline{\OE}_{1,j}+\text{(a lower term).}
\endaligned
$$
Here the lower term is $\up^{a_1}{\OE}_{j-1,j}^{(a_{j-1}-a_{j-2})}\cdots{\OE}_{2,j}^{(a_{2}-a_{1})}\OE_{1,j}^{(a_1)}\overline{\sfK}_1$, which has a smaller $\bar j$-segment as $\ol{\sfK}_1=\ol{\OE}_{1,1}$ and $j>1$.
On the other hand, since each lower term $\fkb^{B,\bf0}$ in \eqref{eee} has a smaller $j^{+}$-segment, the calculation above shows that $\overline{\sfK}_1\fkb^{B,\bf0}$ has a smaller $(j^+,\bar j)$-segment (relative to the leading term). Hence,
$$\aligned
\overline{\sfK}_{1}&\e_{1}^{(a_1)}\cdots\e_{j}^{(a_{j-1})}=\q^{a_1}\fkb+ \mbox{ (lower terms)},\endaligned$$
where $\fkb={\OE}_{j-1,j}^{(a_{j-1}-a_{j-2})}\cdots{\OE}_{2,j}^{(a_{2}-a_{1})}\sfK_1^{-1}{\OE}_{1,j}^{(a_{1}-1)}\overline{\OE}_{1,j}$ is the leading term and every lower term has the form $\fkb'\fkc$ where $\fkb'$ is a $(j^+,\bar j)$-segment with a $(\bar j_\uparrow,j^+_\downarrow)$-column $\prec$ that for $\fkb$.  We now compute $\f_{i-1}\cdots\f_2\f_{1}\fkb$. By (2) above,
$$\f_{i-1}\cdots\f_2\f_{1}=\OE_{i,1}+(\text{lower terms with smaller $1^-$-segments}).$$
Then, by Lemma~\ref{comm formulas}(4),
\begin{equation}\label{monpbw3}\notag
\begin{aligned}
&\f_{i-1}\cdots\f_2\f_{1}\fkb=\q^{a_1-1}{\OE}_{j-1,j}^{(a_{j-1}-a_{j-2})}\cdots{\OE}_{2,j}^{(a_{2}-a_{1})}\sfK_1^{-1}{\OE}_{1,j}^{(a_{1}-1)}(\f_{i-1}\cdots\f_2\f_{1})\overline{\OE}_{1,j}\\
&=
\begin{cases}
\q^{a_1-1}{\OE}_{j-1,j}^{(a_{j-1}-a_{j-2})}\cdots{\OE}_{2,j}^{(a_{2}-a_{1})}\sfK_1^{-1}{\OE}_{1,j}^{(a_{1}-1)}{\sfK}_{1}\overline{\sfK}_{j}+ \mbox{ (lower terms)},             & i=j;\\
\q^{a_1-1}{\OE}_{j-1,j}^{(a_{j-1}-a_{j-2})}\cdots{\OE}_{2,j}^{(a_{2}-a_{1})}\sfK_1^{-1}{\OE}_{1,j}^{(a_{1}-1)}\overline{\OE}_{i,j}{\sfK}_{1}{\sfK}_{i}^{-1}+ \mbox{ (lower terms)}, & i<j;\\
\q^{a_1-1}{\OE}_{j-1,j}^{(a_{j-1}-a_{j-2})}\cdots{\OE}_{2,j}^{(a_{2}-a_{1})}\sfK_1^{-1}{\OE}_{1,j}^{(a_{1}-1)}{\sfK}_{1}{\sfK}_{j}\overline{\OE}_{i,j}+ \mbox{ (lower terms)}, & i>j.
\end{cases}\\
\end{aligned}
\end{equation}
The computation of $\f_{i-1}\cdots\f_2\f_{1}(\fkb'\fkc)=(\f_{i-1}\cdots\f_2\f_{1}\fkb')\fkc$ is similar. It has leading terms whose $(j^+,\bar j)$-segment is smaller than that of the leading term in $\f_{i-1}\cdots\f_2\f_{1}\fkb$.

Combining the two cases and noting \cite[(5.9)]{DW1} give
\begin{equation}\label{monpbw3}\notag
\begin{aligned}
&\f_{i-1}\cdots\f_{1}\overline{\sfK}_{1}\e_{1}^{(a_1)}\cdots\e_{j}^{(a_{j-1})}\\
&=
\begin{cases}
{\OE}_{j-1,j}^{(a_{j-1}-a_{j-2})}\cdots{\OE}_{2,j}^{(a_{2}-a_{1})}{\OE}_{1,j}^{(a_{1}-1)}\overline{\OE}_{j,j}+ \mbox{ (lower terms)},             & i=j;\\
{\OE}_{j-1,j}^{(a_{j-1}-a_{j-2})}\cdots{\OE}_{2,j}^{(a_{2}-a_{1})}{\OE}_{1,j}^{(a_{1}-1)}\overline{\OE}_{i,j}{\sfK}_{i}^{-1}+ \mbox{ (lower terms)}, & i<j;\\
{\OE}_{j-1,j}^{(a_{j-1}-a_{j-2})}\cdots{\OE}_{2,j}^{(a_{2}-a_{1})}{\OE}_{1,j}^{(a_{1}-1)}{\sfK}_{j}\overline{\OE}_{i,j}+ \mbox{ (lower terms)}, & i>j,
\end{cases}\\
\end{aligned}
\end{equation}
where  every lower term has a smaller $(j^+,\bar j)$-segment.
Now (3) follows from collecting the $\sfK$'s to the left.
\end{proof}
We are now ready to prove the promised monomial basis.
\begin{proposition}\label{The monomial basis}
The set 
\begin{equation}\label{monobasis}\frak M=\{\fkm^{A,\bfj}\mid A\in M_n(\bN|\bZ_2)',\bfj\in\Z^n\}\end{equation} defined in \eqref{MonBs}
forms a basis, a monomial basis,  for  $\qUq$.
\end{proposition}
\begin{proof} For $1\leq j\leq n-1$, by Lemma \ref{leater}(2),
$$\aligned
\fkF_j^0&=\sfF_{n-1}^{(a_{n,j}^0)} \sfF_{n-2}^{(a_{n-1,j}^0+a_{n,j}^0)}\cdots \sfF_j^{(\sum_{j+1\leq s\leq n}a_{s,j}^0)}\\
&=\OE_{j+1,j}^{(a^0_{j+1,j})}\OE_{j+2,j}^{(a^0_{j+2,j})}\cdots \OE_{n,j}^{(a^0_{n,j})}+(\text{lower terms})\\
&=\OE_A^{j^-}+(\text{lower terms}),\endaligned$$
where each lower term has a smaller $j^-$-segment. If we put
$\OE_A^-=\OE_{A}^{1^-}\OE_A^{2^-}\cdots \OE_{A}^{(n-1)^-}$, then
\begin{equation}
\begin{aligned}
\fkF_{1}^0\fkF_2^0\cdots\fkF_{n-1}^0=\OE_{A}^{-}+ \mbox{ (lower terms)}.
\end{aligned}
\end{equation}
Here the matrix associated with each lower term is lower triangular for the even part and 0 for the odd part.

On the other hand, 
putting $\OE_A^{n^+,\bar n}:=\OE_{n-1,n}^{(a_{n-1,n}^0)}\cdots \OE_{2,n}^{(a_{2,n}^0)}\OE_{1,n}^{(a_{1,n}^0)}\overline{\OE}_{1,n}^{a_{1n}^1}\cdots\overline{\OE}_{n,n}^{a_{nn}^1}$ as in \eqref{EAj+}, 
by Lemma \ref{leater}(3) 
$$
\begin{aligned}
\fkF_{1,n}^1\fkF_{2,n}^1\cdots \fkF_{n,n}^1\fkE_{n-1}^0&=\alpha \OE_A^{n^+,\bar n}+ \mbox{ (lower terms $\fkb^{B,\bfj}$)},\\
\end{aligned}
$$
where $\alpha\in \q^{\mathbb Z}\sfK^{\mathbb Z^n}$ is a constant and each lower term $\fkb^{B,\bfj}$ has the form
$$
{\mathfrak b}^{B,\bfj}=\sfK^\bfj \OE_{n-1,n}^{(b_{n-1,n}^0)}\cdots \OE_{2,n}^{(b_{2,n}^0)}\OE_{1,n}^{(b_{1,n}^0)}\overline{\OE}_{1,n}^{b_{1n}^1}\cdots\overline{\OE}_{n,n}^{b_{nn}^1}
\cdot\OG.
$$
with $B\prec A_{n^+,\bar n}$, where $\OG$ denotes  the remaining part of  this  lower term and 
$A_{n^+,\bar n}$ is the matrix whose $n^+$ and $\bar n$ columns are the same as in $A$ and the rest are zeros. In other words, 
$\fkb^{A_{n^+,\bar n},\bfl}=\OE_A^{n^+,\bar n}.$

Observe that every recursively defined root vector $\OE_{i,n}$ involves the generator $\sfE_{n-1}=\OE_{n-1,n}$ and, without $\sfE_{n-1}$, ${\sfK}_{\bar n}$ cannot be generated from (QQ4) and nor ${\sfE}_{\bar n}$ from (QQ3). Since $\sfE_{n-1}$ does not appear in $\OG$ and nor in 
$$\Pi:=\prod_{j=1}^{n-1}(\fkF_{1,n-j}^1\fkF_{2,n-j}^1\cdots \fkF_{n,n-j}^1\fkE_{n-j-1}^0),$$ it follows that
when $\OG\cdot\Pi$ is written as a linear combination of PBW basis elements, the root vectors $\OE_{i,n}, \overline{\OE}_{j,n}$, $i<n, j\leq n$, will not
appear. Hence, each of them has no $(n^+,\bar n)$-segment. Thus, every term in
$$\OE_{n-1,n}^{(b_{n-1,n}^0)}\cdots \OE_{2,n}^{(b_{2,n}^0)}\OE_{1,n}^{(b_{1,n}^0)}\overline{\OE}_{1,n}^{b_{1n}^1}\cdots\overline{\OE}_{n,n}^{b_{nn}^1}\cdot \OG\cdot\Pi $$
is less than every term in $\OE_A^{n^+,\bar n}\cdot\Pi$, and so the leading term must occur in the latter. 
By induction, we conclude 
\begin{equation*}
\prod_{j=1}^n(\fkF_{1,n-j+1}^1\fkF_{2,n-j+1}^1\cdots \fkF_{n,n-j+1}^1\fkE_{n-j}^0)
=\beta_{A}\OE^{n^+,\bar n}_{A}\cdots\OE^{1^+,\bar 1}_{A}+ \mbox{ (lower terms)}, 
\end{equation*}
for some $ \beta_{A}\in \q^{\mathbb Z}\sfK^{\mathbb Z^n}$. Finally,
\begin{equation}\notag
\begin{aligned}
\fkm^{A,\bfj}&=\sfK^\bfj (\beta_{{A}} \OE^{n^+,\bar n}_{A}\cdots\OE^{1^+,\bar 1}_{A}+\mbox{ lower terms})(\OE_{A}^-+ \mbox{ lower terms})\\
&=\beta_{A} \fkb^{A,\bfj}+ \mbox{ (lower terms)}, 
\end{aligned}
\end{equation}
Thus, by Lemma \ref{The PBW basis},
the set
$
\mathfrak{M}':=\{\beta_{A}^{-1} \fkm^{A,\bfj}  \mid A\in
M_n(\bN|\bZ_2)',\,\bfj\in \mathbb Z^{n}\}
$
forms a basis for  $\qUq$. By writing $\beta_A=\up^{e_A}\sfK^{\bfj_A}$, for some $e_A\in\Z$ and 
$\bfj_A\in\Z^n$, we have
$
\mathfrak{M}'=\{\up^{-e_A} \fkm^{A,\bfj-\bfj_A}  \mid A\in
M_n(\bN|\bZ_2)',\,\bfj\in \mathbb Z^{n}\}
$. Hence, $\{\fkm^{A,\bfj-\bfj_A}  \mid A\in
M_n(\bN|\bZ_2)',\,\bfj\in \mathbb Z^{n}\}
$, forms a basis which is exactly $\mathfrak M=
\{ \fkm^{A,\bfj}  \mid A\in
M_n(\bN|\bZ_2)',\,\bfj\in \mathbb Z^{n}\}
$.
\end{proof}
Parallel to Proposition \ref{ZbasisB}, we have the following.
\begin{corollary}\label{ZbasisM}The set
$$\mathfrak M_\sZ=\left\{\sfK^\tau \left[\begin{matrix}\sfK\\\la\end{matrix}\right]\mathfrak m^{A,\bfl}\bigg| A\in M_n(\N|\Z_2)', \la\in\mathbb N^n, \tau\in(\Z_2)^n\right\}$$
forms a $\sZ$-basis for the Lusztig form $\qUqZ$.
\end{corollary}

\section{The $\qUq$-supermodule $\SW$}
We now aims at the construction of the regular representation of $\qUq$. This will be a three step construction. First, we construct a deformed polynomial superalgebra $\SW$ on which $\qUq$ acts so that we obtain a $\qUq$-supermodule. We then extend in next section this action to the $n$-fold tensor product $\SW^{\otimes n}$, which can be regarded again a deformed polynomial superalgebra, and extend further in \S7 to its formal power series algebra in which we construct subspace $\sVn$ together with some explicit action formulas on a basis by generators of $\qUq$. From \S8 onwards, we prove that $\sVn$ is isomorphic to the regular representation of $\qUq$ and, thus, give new realisations for both $\qUq$ and the queer $q$-Schur algebras.

Let $n$ be a positive integer and write $[1,n]:=\{1,2,\ldots,n\}$. As in \eqref{bar on i}, put $\ol {i}=n+i$ for all $i\in[1,n]$. We will identify the set
$\{1,2,\ldots,n,\bar1,\bar2,\ldots,\bar n\}$ with $[1,2n]$. Let $\Z_2=\{0,1\}$. We always regard
$\Z_2$ as a subset of $\N$ unless it is used to describe a superspace, where $\Z_2$ is an abelian group of order 2.

Define the parity function on $[1,2n]$
\begin{equation}
p:[1,2n]\longrightarrow \mathbb Z_2,\;p(i)=0,p(\ol{i})=1 \,\forall i\in[1,n].
\end{equation}

\begin{definition}
Let $\SW=\mathscr A_{\mathbb Q(\up)}[X_1,\ldots,X_n,X_{\bar1},\ldots,X_{\bar n}]$ be the superalgebra over $\mathbb Q(\q)$ with
$$\aligned
\text{even generators:}& \;X_1,\cdots, X_{n},\\
\text{odd generators:} &\;X_{\bar 1},\cdots, X_{\bar n},
\endaligned$$
and relations:
\begin{equation}
\begin{aligned}
(1)\quad&X_iX_j=X_jX_i,\    X_iX_{\bar j}=X_{\bar j}X_i,\label{suppol}\\
(2)\quad&X_{\bar i}X_{\bar j}=-X_{\bar j}X_{\bar i}\;(i\neq j),\;X_{\bar i}^{2} =\frac{\q-\q^{-1}}{\q+\q^{-1}}X_{i}^{2}.
\end{aligned}
\end{equation}
\end{definition}
Note that, if we specialise $\up$ to 1, then $\mathscr A_1(n)$ is isomorphic to the tensor product of the  polynomial algebra $\mathbb Q[X_1,\ldots,X_n]$ and the exterior (super)algebra $\Lambda_{\mathbb Q}(X_{\bar 1},\cdots, X_{\bar n})$. Since $X_{\ol i}^2=0$ in $\mathscr A_1(n)$, this algebra $\mathscr A_1(n)$ is known as a polynomial superalgebra. The superalgebra $\SW$ is a deformed polynomial superalgebra and may be regarded as the algebra over $\mathbb Q(\up)[X_1,\ldots,X_n]$ with generators $X_{\bar 1},\cdots, X_{\bar n}$ and relations \eqref{suppol}(2). We  will call $\SW$ the {\it queer polynomial superalgebra} in the sequel.

For $\bfa=(a_1,\cdots, a_{n},a_{\bar 1},\cdots, a_{\bar n})\in \mathbb N^{n}\times \mathbb N^{n},$ set

\begin{equation}\label{monomial}
\begin{aligned}
& X^{\bfa}=X_{1}^{a_{1}}\cdots
 X_{n}^{a_{n}}X_{\bar 1}^{a_{\bar 1}}\cdots
 X_{\bar n}^{a_{\bar n}},\\
&  X^{[\bfa]}=X_{1}^{(a_{1})}\cdots
 X_{n}^{(a_{n})}X_{\bar 1}^{a_{\bar 1}}\cdots
 X_{\bar n}^{a_{\bar n}}.
\end{aligned}
\end{equation}
These monomials have parity
\begin{equation}\label{parity1}
p(X^{\bfa})=p(X^{[\bfa]})=p(\bfa):={a_{\bar 1}+\cdots +a_{\bar n}}\ (\mod 2).
\end{equation}
We also have the usual degree function:
$$\deg(X^\bfa)=|\bfa|:=\sum_{i=1}^{n}(a_i+a_{\bar i})\quad\text{for all}\quad \bfa\in\N\times\Z_2.$$
\begin{lemma} (1) The sets
$$
\sX:=\{X^{\bfa} \mid
  \bfa\in \mathbb N^{n}\times \mathbb Z_2^{n} \}  \mbox{  and  }
\{X^{[\bfa]} \mid
  \bfa\in \mathbb N^{n}\times \mathbb Z_2^{n} \}
$$
form bases for $\SW$.
 
 (2) For $i\in\mathbb Z_2,r\in\mathbb N$, let
 $$\SW_i=\text{\rm span}\{X^\bfa\in\sX\mid p(\bfa)=i\},\quad\mathscr A_\up(n,r)=\text{\rm span}\{X^\bfa\in\sX\mid \deg(X^\bfa)=r\}.$$
Then there are two grading structures on $\SW$: 
$$\SW=\SW_0\oplus\SW_1 =\bigoplus_{r\geq0}\mathscr A_\up(n,r).$$
 \end{lemma}

We now introduce some linear maps in the (super) subalgebra $\scrL$ of the algebra $\text{End}_{\mathbb Q(\up)}(\SW)$ defined by
\begin{equation}\aligned
\scrL&=\scrL_0\oplus\scrL_1,\; \text{ where}\\
\scrL_i&=\{f\in \text{End}_{\mathbb Q(\up)}(\SW)\mid f(\SW_j)\subseteq \SW_{j+i}\;\forall j\in\mathbb Z_2\}.
\endaligned
\end{equation}

For all $i\in[1, 2n]$,
define {\it quantum differential (or $\up$-differential) operators}
$$
\partial_i:\SW\longrightarrow \SW $$
by setting, for $X^\bfa\in\sX$,
\begin{equation}\label{partial}
\partial_{i}(X^\bfa)=
\begin{cases}[a_i]X^{\bfa-\bfe_i}, &\text{ if }i\in[1,n];\\
(-1)^{\sum_{n+1\leq j<i}a_{j}}[a_i]X^{\bfa-\bfe_i}, &\text{ if }i\in[n+1,\ 2n].
\end{cases}
\end{equation}
Note that, if $a_i=0$, then $\partial_{i}(X^\bfa)=[0]X^\bfa=0$. 
Note also that $\partial_i$ satisfies the super quantum derivative rule
$$\aligned
\partial_i(ab)&=\partial_i(a)\de_i(b)+(-1)^{p(a)p(i)}\de_i^{-1}(a)\partial_i(b)\\
&=\partial_i(a)\de_i^{-1}(b)+(-1)^{p(a)p(i)}\de_i(a)\partial_i(b),
\endaligned$$
where $\de_i^{\pm1}$ is the linear isomorphism
$$\de_i^{\pm1}:\SW\longrightarrow \SW,\;\de_i^{\pm1}(X^\bfa)\longmapsto\up^{\pm a_i}X^\bfa.$$
Clearly, $\partial_1,\ldots,\partial_n\in\scrL_0$ and $\partial_{\bar 1},\ldots,\partial_{\bar n}\in\scrL_1$.

We also define, for $i\in[1,n], j\geq0$,
$$\partial_{i}^{(j)}:\SW\longrightarrow \SW \text{ by setting } \partial_{i}^{(j)}(X^\bfa)=[a_i+j]X^{\bfa-\bfe_i},$$
and define the linear maps $ \chi _1,\ldots, \chi _n\in\scrL_0$ and $ \chi _{\bar 1},\ldots, \chi _{\bar n}\in\scrL_1$ as the multiplication by $X_i$:
\begin{equation}\label{chi}
\chi_i:\SW\longrightarrow \SW, \;a\longmapsto X_ia,\text{ for all }a\in\SW.
\end{equation}
Note that, for a basis element $X^\bfa$,
$$
\chi_i(X^\bfa)=X_iX^\bfa=\begin{cases}X^{\bfa+\bfe_i}, &\text{ if }i\in[1,n];\\
(-1)^{\sum_{n+1\leq j<i}a_{j}}X^{\bfa+\bfe_i}, &\text{ if }i\in[n+1,2n].
\end{cases}$$
We also 
need the following {\it signed identity maps}:
$$s_{\ol i}:\SW\longrightarrow\SW,\;\;X^\bfa\longmapsto (-1)^{a_{\ol i}}X^\bfa,\;\;i\in[1,n].$$


The following subspace decompositions will be useful below for checking certain commutation formulas and relations: for $i\in[1,n]$,
\begin{equation}\label{Abij}
\SW=\begin{cases}\scrA_{i,0}\oplus\scrA_{i,1},&\\
\scrA_{\bar i,0}\oplus\scrA_{\bar i,1},&\\
\end{cases}\quad
\text{where}\quad
\aligned
\scrA_{i,0}&=\text{span}\{X^\bfa\in\sX\mid a_{i}=0\},\\
\scrA_{i,1}&=\text{span}\{X^\bfa\in\sX\mid a_{i}\geq1\},\\
\scrA_{\bar i,j}&=\text{span}\{X^\bfa\in\sX\mid a_{\ol i}=j\}\, (j\in\mathbb Z_2).\\
\endaligned
\end{equation}
For a linear map $f$ on $\SW$, $f|_{\scrA_{i,j}}$ denotes its restriction to $\scrA_{i,j}$. 
We may also define the projection map onto $\scrA_{\bar i,j}$ via the decomposition $\SW=\scrA_{\bar i,0}\oplus\scrA_{\bar i,1}$.
\begin{lemma}\label{opecom1}The following commutation relations hold in $\mathscr L$.
\begin{enumerate}
\item  For all $i,j\in[1,2n],i\neq j$,
$\partial_{i}\partial_{j}=(-1)^{p(i)p(j)}\partial_{j}\partial_{i}$, $ \chi _{i} \chi _{j}=(-1)^{p(i)p(j)}\chi _{j} \chi _{i}$ and, for all $i\in[1,n]$, $\partial_{\ol i}^2=0$, $\chi_{\ol i}^2=\frac{\up-\up^{-1}}{\up+\up^{-1}}\chi_i^2$.

\item For all $i,j\in[1,2n]$, $\partial_i  \de_j=\begin{cases}\de_j  \partial_i,&\text{ if }i\neq j;\\ \up\de_i  \partial_i,&\text{ if }i=j.\end{cases}$

\item  For all $i,j\in[1,2n]$ with $j\neq i$ or $j\neq\bar i$ when $i\in[1,n]$, $\chi_j  \partial_i=(-1)^{p(i)p(j)}\partial_i  \chi_j$, and, for all $i\in[1,n]$,
\begin{enumerate}
\item $\partial_i  \chi_i=\chi_i\parb_{i}^{(1)}$; $\chi_{\ol i}  \partial_{\ol i}$ (resp., $\partial_{\ol i}   \chi_{\ol i}$) is the projection map onto $\scrA_{\ol i,1}$ (resp., $\scrA_{\ol i,0}$), and
$\chi_{\ol i}  \partial_{\ol i}+\partial_{\ol i}   \chi_{\ol i}=1$.
\item $\partial_i  \chi_{\bar i}|_{\scrA_{\ol{i},0}}=\chi_{\bar i}  \partial_i|_{\scrA_{\ol{i},0}}$ and $\partial_i  \chi_{\bar i}|_{\scrA_{\ol{i},1}}=\chi_{\bar i}  \partial_i^{(2)}|_{\scrA_{\ol{i},1}}$.
\end{enumerate}

\item For all $i,j\in[1,2n]$ with $j\neq i$ or $j\neq\bar i$ when $i\in[1,n]$, we have
 $\chi_j  \de_i=\de_i  \chi_j$, and, for $i\in[1,n]$,
\begin{enumerate}
\item $\chi_i  \de_i=\up^{-1}\de_i  \chi_i$; $\chi_{\ol i}  \de_{\ol i}|_{\scrA_{\ol{i},0}}=\up^{-1} \de_{\ol i}  \chi_{\ol i}|_{\scrA_{\ol{i},0}}$, $\chi_{\ol i} \de_{\ol i}|_{\scrA_{\ol{i},1}}=\up \de_{\ol i}  \chi_{\ol i}|_{\scrA_{\ol{i},1}}$;
\item $\chi_{\ol i}  \de_{i}|_{\scrA_{\ol{i},0}}=\de_{i}  \chi_{\ol i}|_{\scrA_{\ol{i},0}}$, $\chi_{\ol i}  \de_{i}|_{\scrA_{\ol{i},1}}=\up^{-2} \de_{i}  \chi_{\ol i}|_{\scrA_{\ol{i},1}}$. 
\end{enumerate}
\item For $i\in[1,n]$,\vspace{-1ex}
\begin{equation}\label{opecom2}
\begin{aligned}
(\mathrm a)&\quad\sbi\chibi=-\chibi\sbi,\;\sbi\parb_{{\bar i}}=\parb_{{\bar i}},\\
(\mathrm b)&\quad \chi_i\parb_{i}^{(2)}+ \chi_i\parb_{i}=({\upv+\upv^{-1}})\chi_i\parb_{i}^{(1)},\\
(\mathrm c)&\quad \chi_i\parb_{i}^{(2)}= \chi_i\parb_{i}+\delta_{i}+\delta^{-1}_{i},\\
(\mathrm d)&\quad \chi_{i}\parb_{i}^{(1)} =\q\chi_{i}\parb_{i} + \delta^{-1}_{i},\\
(\mathrm e)&\quad\parbi\dei^{\pm1}\chibi+\chibi\parbi\dei^{\pm1}=\dei^{\pm1},\\
(\mathrm f)&\quad\chii\parb_{i}^{(2)} =\upv^2\chii\pari+({\upv+\upv^{-1}})\dei^{-1},\\
(\mathrm g)&\quad \parb_{i}^{(1)}\parb_{i}^{(1)}+ \parb_{i}\parb_{i}=(\upv+\upv^{-1})\parb_{i}^{(1)}\parb_{i}.\\
\end{aligned}
\end{equation}
\end{enumerate}
\end{lemma}
\begin{proof}
All relations can be checked easily by definition through applying to a basis vector $X^\bfa$. We omit the proof. Note that relations (b)/(g), (c), (d) (f) in (5) follow respectively from the identities: for all $a\geq1$,
\begin{equation}
\begin{aligned}\notag
({\rm b}')\quad&[a+2]+[a]=(\upv+\upv^{-1})[a+1];\\
({\rm c}')\quad&[a+2]-[a]=\upv^{a+1}+\upv^{-(a+1)};\\
({\rm d}')\quad&[a+1]-\upv[a]=\upv^{-a},\\
({\rm f}')\quad&[a+2]-\upv^2[a]=(\upv+\upv^{-1})\upv^{-a}.\\
\end{aligned}
\end{equation}
\end{proof}

We now use these operators to define a $\qUq$-module structure on $\SW$. 

For $1\leq h\leq n-1, i\in[1,n]$, let
\begin{equation}\label{EFKs}
\aligned
\sK_i&=\de_i \de_{\ol i},\\
\sE_h&=\chi_h \partial_{h+1}  \de_{\ol{h+1}}+ \chi_{\ol h} \partial_{\ol{h+1}}  \de_{h+1}^{-1}s_{\ol h},\\
\sF_h&=\chi_{h+1} \partial_{h}  \de_{\ol{h}}^{-1}+\chi_{\ol
{h+1}} \partial_{\ol{h}}  \de_{h},\\
\endaligned
\qquad %
\aligned
\sK_{\ol i}&=\chi_i \partial_{\ol i} \de^{-1}_i+\chi_{\ol i} \partial_i \de_{\ol i},\\
\sE_{\ol h}&=\chi_{\ol h} \partial_{{h+1}}  \de_{\ol{h+1}}s_{\ol h}+\chi_h \partial_{\ol{h+1}}  \de_{{h+1}}^{-1},\\
\sF_{\ol h}&=\chi_{\ol{h+1}} \partial_{h}  \de_{{\ol h}}^{-1}+\chi_{h+1} \partial_{\ol{h}}  \de_{h}.\\
\endaligned
\end{equation}
Clearly, $\sK_i,\sE_h,\sF_h\in\scrL_0$ and $\sK_{\ol i},\sE_{\ol h},\sF_{\ol h}\in\scrL_1$.
Let 
$$\{\bfe_i\mid{1\leq i\leq 2n}\}=\{\bfe_1,\ldots,\bfe_n,\bfe_{\ol 1},\ldots,\bfe_{\ol n}\}$$ ($\bfe_{\ol i}=\bfe_{n+i}$) be the standard  basis vectors of $\mathbb Z^{2n}.$

 \begin{theorem}\label{supalg} (1) The map 
$\rho:\qUq\to\text{End}_{\mathbb Q(\up)}(\SW)$ sending
$\k_{i}, \e_{j},\f_{j},\k_{\bi}, \e_{\bj},\linebreak\f_{\bj}$ to $\sK_{i}, \sE_{j},\sF_{j},\sK_{\bi}, \sE_{\bj},\sF_{\bj}$,
respectively, defines an algebra homomorphism. 
In other words, $\SW$ becomes a $\qUq$-module with the following $\qUq$-action formulas:
\begin{equation}\label{U-actions}
\begin{aligned}
\k_i .X^{\bfa}&=\v^{a_i+a_{\bar i}}    X^{\bfa},\\
\e_\i .X^{\bfa}&=\v^{a_{\overline{\i+1}}}[a_{\i+1}] X^{\bfa+\bfe_\i-\bfe_{\i+1}}
                +\v^{-a_{\i+1}}[ a_{\overline{\i+1}}]  X^{\bfa+\bfe_{\bar\i}-\bfe_{\overline{\i+1}}},\\
 \f_\i .X^{\bfa}&=\v^{-a_{\bar{\i}}}[ a_{\i}]   X^{\bfa-\bfe_\i+\bfe_{\i+1}}
                 +\v^{a_{\i}}[a_{\bar \i}]    X^{\bfa-\bfe_{\bar \i}+\bfe_{\overline{\i+1}}},\\
                 \k_{\bar i} .X^{\bfa}&=(-1)^{\sum_{1\leq j < i }a_{\bar j}} \bigg( \v^{-a_{i}}[ a_{\bar i}]
                          X^{\bfa+\bfe_i-\bfe_{\bar i}}+ \v^{a_{\bar i}}[a_{i}]
                          X^{\bfa+\bfe_{\bar i}-\bfe_{i}}\bigg),\\
\e_{\bar \i} .X^{\bfa}&=(-1)^{\sum_{1\leq j\leq \i}a_{\bar j}}\bigg(\v^{a_{\overline{\i+1}}} [ a_{\i+1}]
                       X^{\bfa+\bfe_{\bar \i}-\bfe_{\i+1}}+  \v^{-a_{\i+1}}[ a_{\overline{\i+1}}]
                       X^{\bfa+\bfe_\i-\bfe_{\overline{\i+1}}}\bigg),\\
\f_{\bar \i} .X^{\bfa}&=(-1)^{\sum_{1\leq j\leq \i}a_{\bar j}} \v^{-a_{\bar \i}}[a_{\i}]
                       X^{\bfa-\bfe_\i+\bfe_{\overline{\i+1}}}+(-1)^{\sum_{1\leq j< \i}a_{\bar j}}\v^{a_\i} [ a_{\bar\i} ]
                        X^{\bfa-\bfe_{\bar \i}+\bfe_{\i+1}},\\
\end{aligned}
\end{equation}
for all $\bfa\in \mathbb N^{n}\times \mathbb Z_2^{n},$ $1\leq i, h\leq n, h\neq n$.

(2) It is a weight $\qUq$-supermodule and all its homogeneous components  $\SWr$, $r\in \mathbb N$, are $\qUq$-subsupermodules.
\end{theorem}
\begin{proof} Assertion (2) follows easily from (1). We now prove (1) by verifying that all the relations 
are satisfied for these operators.

The relations in (QQ1) can all be trivially checked except the last one which we prove now.
If $1\leq i\neq j\leq n$ then, by Lemma \ref{opecom1}(1)--(4), $ab=(-1)^{p(a)p(b)}ba$ for all
$a\in\{\partial_i,\chi_i,\delta_i,\partial_{\bar i},\chi_{\bar i},\delta_{\bar i}\}$, $b\in\{\partial_j,\chi_j,\delta_j,\partial_{\bar j},\chi_{\bar j},\delta_{\bar j}\}$. So
\begin{equation}
\begin{aligned}\notag
\kbi\kbj&=(\chii\parbi\dei^{-1}+\chibi\pari\debi)(\chij\parbj\dej^{-1}+\chibj\parj\debj)\\
&=-(\chij\parbj\dej^{-1}+\chibj\parj\debj)(\chii\parbi\dei^{-1}+\chibi\pari\debi)=-\kbj\kbi.
\end{aligned}
\end{equation}
If $1\leq i=j\leq n$, by Lemma \ref{opecom1}(2),(3),(1), we see $(\chii\parbi\dei^{-1})^2=0$, and so
\begin{equation}
\begin{aligned}\notag
\kbi^2&=(\chii\parbi\dei^{-1}+\chibi\pari\debi)^2=\chii\parbi\dei^{-1}\cdot\chibi\pari\debi
+\chibi\pari\debi\cdot\chii\parbi\dei^{-1}+\chibi\pari\debi\cdot\chibi\pari\debi.
\end{aligned}
\end{equation}
Restricting $\kbi^2$ to the subspace $\scrA_{\bar i,0}$ gives $\chibi\pari\debi\chii\parbi\dei^{-1}|_{\scrA_{\bar i,0}}=0$ and, by Lemma \ref{opecom1}(3), $\parbi\chibi|_{\scrA_{\bar i,0}}=1$. Thus, 
\begin{equation}
\begin{aligned}\notag
\kbi^2|_{\scrA_{\bar i,0}}&=\chii\dei^{-1}(\parbi\chibi)\pari
+\chibi\pari(\up\chibi\debi)\pari=\chii\dei^{-1}\pari
+\upv\frac{\upv-\upv^{-1}}{\upv+\upv^{-1}}\chii^2\pari^2=\frac{\cki^{2}-\cki^{-2}}{\upv^2-\upv^{-2}}\bigg|_{\scrA_{\bar i,0}},
\end{aligned}
\end{equation}
where the last equality is seen by applying the left side to $X^\bfa$, since
\begin{equation}\label{gauss}
\upv^{-a_i+1}[a_i]+\upv\frac{\upv-\upv^{-1}}{\upv+\upv^{-1}}[a_i-1][a_i]=\frac{\upv^{2a_i}-\upv^{-2a_i}}{\upv^2-\upv^{-2}}.
\end{equation}
Similarly, restricting $\kbi^2$ to the subspace $\scrA_{\bar i,1}$ gives $\chii\parbi\dei^{-1}\chibi\pari\debi|_{\scrA_{\bar i,1}}=0$,
$\chibi\pari\debi\cdot\chii\parbi\dei^{-1}|_{\scrA_{\bar i,1}}=\pari\chii\dei^{-1}|_{\scrA_{\bar i,1}}$
and $\chibi\pari\debi\cdot\chibi\pari\debi|_{\scrA_{\bar i,1}}=\upv\frac{\upv-\upv^{-1}}{\upv+\upv^{-1}}\pari \chii^2\pari|_{\scrA_{\bar i,1}}$. Hence,
\begin{equation}
\kbi^2|_{\scrA_{\bar i,1}}=\pari\chii\dei^{-1}+\upv\frac{\upv-\upv^{-1}}{\upv+\upv^{-1}}\pari \chii^2\pari=\frac{\cki^{2}-\cki^{-2}}{\upv^2-\upv^{-2}}\bigg|_{\scrA_{\bar i,1}}
\end{equation}
by \eqref{gauss} with $a_i$ replaced by $a_i+1$, proving (QQ1).

The relations in (QQ2) follows from the following commuting relations:
    $$\de_i\de_{\ol i}(\chi_k\partial_l)=\up^{\de_{i,j}-\de_{i,j+1}}(\chi_k\partial_l)\de_i\de_{\ol i},\quad
    \de_i\de_{\ol i}(\chi_{l}\partial_{k})=\up^{-(\de_{i,j}-\de_{i,j+1})}(\chi_{l}\partial_{k})\de_i\de_{\ol i},
    $$
where $(k,l)\in\{(j,j+1),(\ol j,j+1),(j,\ol{j+1}),(\ol j,\ol{j+1})\}$.

To check the rest of the relations in Definition \ref{presentqUq}, we apply the following two general rules. First, we use Lemma \ref{opecom1}(3) to break the proof of  a relation $\sR=0$ on $\SW$ by
proving that $\sR\partial_{\ol i}\chi_{\ol i}=0$ and
$\sR\chi_{\ol i}\partial_{\ol i}=0$, or equivalently, $\sR|_{\scrA_{{\ol i}, j}}=0$ for $j=0,1$.
Second, to simplify a product of several operators, we always move the $\de_i,s_i$ to the right and move $\chi_i$ to the left.


We now verify (QQ3). The first two commuting relations follows directly from the commuting relations in Lemma \ref{opecom1}(1)--(4) for $i\neq j,j+1$. We now verify the relation
\begin{equation}\label{QQ3}
(\kbi\cei-\q\cei\kbi)\cki=\ebi\;(1\leq i<n)
\end{equation}
 in (QQ3); the proof for the other three are similar.
By definition,
\begin{equation}\notag
  \begin{aligned}
 & \kbi\cei-\q\cei\kbi\\
=\;&\chii\parbi\dei^{-1}\cdot\chii\parb_{i+1}\delta_{\ov{i+1}}\\
&+\chibi\pari\debi\cdot\chii\parb_{i+1}\delta_{\ov{i+1}}\\
&+\chii(\parbi\dei^{-1}\cdot\chibi)\parb_{\ov{i+1}}\delta^{-1}_{i+1}\sbi\\
&+\chibi\pari\debi \cdot\chibi\parb_{\ov{i+1}}\delta^{-1}_{i+1}\sbi\\
 &-\upv
  \chii\parb_{i+1}\delta_{\ov{i+1}}\cdot\chii\parbi\dei^{-1}\\
  &-\up\chibi\parb_{\ov{i+1}}\delta^{-1}_{i+1}\cdot\sbi\chii\parbi\dei^{-1}\\
  &-\up  \chii\parb_{i+1}\delta_{\ov{i+1}}\cdot\chibi\pari\debi\\
  &-\up\chibi\parb_{\ov{i+1}}\delta^{-1}_{i+1}\cdot\sbi\chibi\pari\debi\\
\end{aligned}
\qquad
  \begin{aligned}
  &\quad\\
=\;&\upv^{-1}\chii\chii\parbi\parb_{i+1}\dei^{-1}\delta_{\ov{i+1}}\\
&+\chibi\chii\parb_{i}^{(1)}\parb_{i+1}\debi\delta_{\ov{i+1}}\\
&+\chii(\dei^{-1}-\chibi\parbi\dei^{-1})\parb_{\ov{i+1}}\delta^{-1}_{i+1}\sbi\;(\text{by }(\ref{opecom2}\mathrm e))\\
 &+\chibi\pari\debi
  \chibi\parb_{\ov{i+1}}\delta^{-1}_{i+1}\sbi\\
  & -\upv  \chii\chii\parb_{i+1}\parbi\delta_{\ov{i+1}}\dei^{-1} \\
  &-\upv\chibi\chii\parb_{\ov{i+1}}\parbi\delta^{-1}_{i+1}\dei^{-1} \\
  &-\upv
  \chii\chibi\parb_{i+1}\pari\delta_{\ov{i+1}}\debi\\
  &-\upv\chibi\chibi\parb_{\ov{i+1}}\pari\sbi\delta^{-1}_{i+1}\debi.\\
\end{aligned}
\end{equation}
Here each term on the right is obtained from the corresponding term on the left by applying some relations in Lemma \ref{opecom1}.

When restricting to $\scrA_{\ol i,0}$, terms with a factor $\parb_{\ol i}$ is zero. Observe also $\de_{\ol i}|_{\scrA_{\ol i,0}}=1=s_{\ol i}|_{\scrA_{\ol i,0}}$. Thus,
$$\begin{aligned}\notag
(\kbi\cei-\q\cei\kbi)\cki|_{\scrA_{\ol i,0}}
  = &\chibi\chii\parb_{i}^{(1)}\parb_{i+1}\delta_{\ov{i+1}}\delta_{i}+\chii\parb_{\ov{i+1}}\delta^{-1}_{i+1}+\chibi\pari(\debi  \chibi)\parb_{\ov{i+1}}\delta^{-1}_{i+1}\delta_{i}\\ & -\upv  \chii\chibi\parb_{i+1}\pari\delta_{\ov{i+1}}\delta_{i}-\upv\chibi\chibi\parb_{\ov{i+1}}\pari\delta^{-1}_{i+1}\delta_{i}\\
\end{aligned}$$
By Lemma \ref{opecom1}(4a)(3a), the third and fifth terms are cancelled. Hence, by (\ref{opecom2}{\rm d}),
$$\aligned
(\kbi\cei-\q\cei\kbi)\cki|_{\scrA_{\ol i,0}}&=\chibi(\chii\parb_{i}^{(1)} - \upv  \chii\pari )
  \parb_{i+1}\delta_{\ov{i+1}}\delta_{i}+\chii\parb_{\ov{i+1}}\delta^{-1}_{i+1}\\
  = &\chibi  \parb_{i+1}\delta_{\ov{i+1}}+\chii\parb_{\ov{i+1}}\delta^{-1}_{i+1}
  =\ebi|_{\scrA_{\ol i,0}}.
 \endaligned$$

Now restricting to $\scrA_{\ol i,1}$,  by noting $\de_{\ol i}|_{\scrA_{\ol i,1}}=\up 1$, $s_{\ol i}|_{\scrA_{\ol i,0}}=-1$, yields
\begin{equation}\label{QQ3a}
\aligned
&(\kbi\cei-\q\cei\kbi)\cki|_{\scrA_{\ol i,1}}\\
=\;&\chii\chii\parbi\parb_{i+1}\delta_{\ov{i+1}}+\upv^2\chibi\chii\pari^{(1)}\parb_{i+1}\delta_{\ov{i+1}}\de_i+(\upv\chii\chibi\parbi\parb_{\ov{i+1}}\delta^{-1}_{i+1}-\upv\chii\parb_{\ov{i+1}}\delta^{-1}_{i+1})\\
 &-\upv\chibi \pari\debi\chibi\parb_{\ov{i+1}}\delta^{-1}_{i+1}\dei\\
  & -\upv^2  \chii\chii\parb_{i+1}\parbi\delta_{\ov{i+1}} -\upv^2\chii\chibi\parb_{\ov{i+1}}\parb_{\ol{i}}\delta^{-1}_{i+1} -\upv^3  \chii\chibi\parb_{i+1}\pari\delta_{\ov{i+1}}\dei\\
  &+\upv^3\chibi\chibi\parb_{\ov{i+1}}\pari\delta^{-1}_{i+1}\dei,\\
  \endaligned
  \end{equation}
By applying $\chi_{\bar i}\partial_{\bar i}|_{\scrA_{\ol i,1}}=1$ to the third and sixth terms,
 $\pari\debi\chibi|_{\scrA_{\ol i,1}}=\chibi\parb_i^{(2)}|_{\scrA_{\ol i,1}}$ (cf. Lemma \ref{opecom1}(3b)) to the fourth term, and combining terms one and five, two and seven, we obtain  
  $$\aligned
\eqref{QQ3a}
 =\;&(1-\upv^2 )\chii\chii\parbi\parb_{i+1}\delta_{\ov{i+1}}+\upv^2\chibi(\chii\pari^{(1)}- \upv  \chii \pari )
     \parb_{i+1}\delta_{\ov{i+1}}\dei+(0)\\
  &-\upv\chibi \chibi\pari^{(2)}\parb_{\ov{i+1}}\delta^{-1}_{i+1}\dei+\upv^2\chii\parb_{\ov{i+1}}\delta^{-1}_{i+1}+\upv^3\chibi\chibi\parb_{\ov{i+1}}\pari\delta^{-1}_{i+1}\dei\\
 =\;&-(\upv^2+1 )\chibi (\chibi\parbi)\parb_{i+1}\delta_{\ov{i+1}}+\upv^2\chibi \parb_{i+1}\delta_{\ov{i+1}}\;(\text{by }(\ref{opecom2}{\rm d}))\\
  &-\upv\chibi \chibi\pari^{(2)}\parb_{\ov{i+1}}\delta^{-1}_{i+1}\dei+\upv^2\chii\parb_{\ov{i+1}}\delta^{-1}_{i+1}+\upv^3\chibi\chibi\parb_{\ov{i+1}}\pari\delta^{-1}_{i+1}\dei\\
     =\; &-\chibi  \parb_{i+1}\delta_{\ov{i+1}}-\upv\chibi \chibi\pari^{(2)}\parb_{\ov{i+1}}\delta^{-1}_{i+1}\dei+\upv^3\chibi\chibi\pari\parb_{\ov{i+1}}\delta^{-1}_{i+1}\dei+\upv^2\chii\parb_{\ov{i+1}}\delta^{-1}_{i+1}\\
 =\;  & -\chibi  \parb_{i+1}\delta_{\ov{i+1}} +\chii\parb_{\ov{i+1}}\delta^{-1}_{i+1}\;(\text{by }(\ref{opecom2}{\rm f}) \text{ via }\chibi^2=\frac{\up-\up^{-1}}{\up+\up^{-1}}\chii^2)
  =\ebi|_{\scrA_{\ol i,1}},
  \endaligned$$
proving \eqref{QQ3}.

For (QQ4), the commuting relations when $i\neq j$ are clear. We only check the relation
\begin{equation}\label{QQ4}
\cei\fbi-\fbi\cei=\ckii^{-1}\kbi-\kbii\cki^{-1};
\end{equation} 
the proof of the other relations for $i=j$ in (QQ4) is similar.

By definition $\cei=\chii\parb_{i+1}\delta_{\ov{i+1}}+\chibi\parb_{\ov{i+1}}\delta^{-1}_{i+1}\sbi$,
$\fbi=\chi_{\ov{i+1}}\parb_{i}\delta_{\ov{i}}^{-1} +\chi_{{i+1}}\parb_{\ov{i}}\delta_{i}$, the left hand side of \eqref{QQ4} becomes
\begin{equation}
\begin{aligned}\label{qqq4a}
&\cei\fbi-\fbi\cei\\
=\;& \chii\parb_{i+1}\delta_{\ov{i+1}} \chi_{\ov{i+1}}\parb_{i}\delta_{\ov{i}}^{-1} \\
& + \chibi(\parb_{\ov{i+1}}\delta^{-1}_{i+1}     \chi_{\ov{i+1}})\sbi \parb_{i}\delta_{\ov{i}}^{-1}\\
   &+ \chii\parb_{i+1}\delta_{\ov{i+1}}\chi_{{i+1}}\parb_{\ov{i}}\delta_{i}\\
   &+ \chibi\parb_{\ov{i+1}}\delta^{-1}_{i+1}\sbi    \chi_{{i+1}}\parb_{\ov{i}}\delta_{i}\\
&-\chi_{\ov{i+1}}\parb_{i}\delta_{\ov{i}}^{-1}  \chii\parb_{i+1}\delta_{\ov{i+1}} \\
&-\chi_{{i+1}}\parb_{\ov{i}}\delta_{i}   \chii\parb_{i+1}\delta_{\ov{i+1}}\\
&-\chi_{\ov{i+1}}\parb_{i}\delta_{\ov{i}}^{-1} \chibi\parb_{\ov{i+1}}\delta^{-1}_{i+1}\sbi\\
&-\chi_{{i+1}}(\parb_{\ov{i}}\delta_{i}\chibi)\parb_{\ov{i+1}}\delta^{-1}_{i+1}\sbi\\
\end{aligned}\qquad
\begin{aligned}
\quad\\
=\;& \chii\parb_{i+1}\delta_{\ov{i+1}}\chi_{\ov{i+1}}\parb_{i}\delta_{\ov{i}}^{-1}\; \qquad\qquad\quad(L_1)\\
   &     + \chibi(-\chi_{\ov{i+1}}\parb_{\ov{i+1}}+1)\delta^{-1}_{i+1}\parb_{i}\sbi\delta_{\ov{i}}^{-1}\;(L_2),\text{by }(\ref{opecom2}{\rm e})\\
&+ \chii\chi_{{i+1}}\parb_{i+1}^{(1)}\parb_{\ov{i}}\delta_{\ov{i+1}}\delta_{i}\; \qquad\qquad\;\;(L_3)\\
&+\upv^{-1} \chibi\chi_{{i+1}}\parb_{\ov{i+1}}\parb_{\ov{i}}\delta^{-1}_{i+1}\delta_{i}\;\qquad\quad (L_4)\\
&-\chi_{\ov{i+1}}\chii \parb_{i}^{(1)}   \parb_{i+1} \delta_{\ov{i}}^{-1}\delta_{\ov{i+1}}\;\qquad \quad(L_5^-)\\
& -\upv\chi_{{i+1}} \chii\parb_{\ov{i}}\parb_{i+1}\delta_{i}\delta_{\ov{i+1}}\; \qquad\qquad(L_6^-)\\
&-\chi_{\ov{i+1}}\parb_{i}\delta_{\ov{i}}^{-1} \chibi\parb_{\ov{i+1}}\delta^{-1}_{i+1}\sbi \; \qquad\quad(L_7^-)\\
&-\chi_{{i+1}}(-\chibi\parb_{\ov{i}}+1)\delta_{i}\parb_{\ov{i+1}}\delta^{-1}_{i+1}\sbi. \;\; (L_8^-)\\
\end{aligned}
\end{equation}
Here we labeled  the eight term on the right by $L_j$ or $L_k^-$ for $1\leq j\leq 4, 5\leq k\leq 8$. On the other hand, the right hand side of \eqref{QQ4} has four terms, labelled in order by $R_1,R_2,R_3^-,R_4^-$ as follows:
\begin{equation}
\begin{aligned}\label{qqq4b}
&\;\ckii^{-1}\kbi-\kbii\cki^{-1}\\
=\;& \delta_{{i+1}}^{-1} \delta_{\ov{i+1}}^{-1} (\chii\parbi\dei^{-1}+\chibi\pari\debi)-
(\chi_{i+1}\parb_{\ov{i+1}}\delta^{-1}_{i+1}+\chi_{\ov{i+1}}\parb_{{i+1}}\delta_{\ov{i+1}})\delta_{{i}}^{-1}\delta_{\ov{i}}^{-1}\\
=\;& \chii\parbi\dei^{-1}\delta_{{i+1}}^{-1} \delta_{\ov{i+1}}^{-1} +\chibi\pari\debi\delta_{{i+1}}^{-1} \delta_{\ov{i+1}}^{-1}-\chi_{i+1}\parb_{\ov{i+1}}\delta^{-1}_{i+1}\delta_{{i}}^{-1}\delta_{\ov{i}}^{-1}
   -\chi_{\ov{i+1}}\parb_{{i+1}}\delta_{\ov{i+1}}\delta_{{i}}^{-1}\delta_{\ov{i}}^{-1}.\\
   &\qquad (R_1)\qquad\qquad\qquad(R_2)\qquad\qquad\qquad(R_3^-)\qquad\qquad\qquad(R_4^-)\\
\end{aligned}
\end{equation}

We first consider the case by restricting the operators in \eqref{qqq4a} and \eqref{qqq4b} to ${\scrA_{\ol{i+1},0}}$. Thus, only five terms, $L_1,L_2,L_3,L_5^-,L_6^-$, in \eqref{qqq4a} and three terms $R_1,R_2,R_4^-$, in \eqref{qqq4b} survive, as the terms with a factor $\parb_{\ol{i+1}}$ are 0.  Regrouping and noting $\chi_{\ol{i+1}}\parb_{\ol{i+1}}|_{\scrA_{\ol{i+1},0}}=0$, $\de_{\ol{i+1}}|_{\scrA_{\ol{i+1},0}}=1$ and $\delta_{\ov{i+1}} \chi_{\ov{i+1}}=\up \chi_{\ov{i+1}}\delta_{\ov{i+1}}$ yield
$$\aligned
&(\cei\fbi-\fbi\cei)|_{\scrA_{\ol{i+1},0}}-(\ckii^{-1}\kbi-\kbii\cki^{-1})|_{\scrA_{\ol{i+1},0}}\\
=\;& L_1-L_5^-+R_4^-+L_3-L_6^-+L_2-R_1-R_2\\
=\;& \chii\parb_{i+1}\delta_{\ov{i+1}}
   \chi_{\ov{i+1}}\parb_{i}\delta_{\ov{i}}^{-1} -\chi_{\ov{i+1}}\chii \parb_{i}^{(1)}   \parb_{i+1} \delta_{\ov{i}}^{-1}\delta_{\ov{i+1}}+\chi_{\ov{i+1}}\parb_{{i+1}}\delta_{\ov{i+1}}\delta_{{i}}^{-1}\delta_{\ov{i}}^{-1}\\
   &+ \chii\chi_{{i+1}}\parb_{i+1}^{(1)}\parb_{\ov{i}}\delta_{\ov{i+1}}\delta_{i}-\upv\chi_{{i+1}} \chii\parb_{\ov{i}}\parb_{i+1}\delta_{i}\delta_{\ov{i+1}}\\
   &+\chibi\delta^{-1}_{i+1}\parb_{i}\sbi\delta_{\ov{i}}^{-1}-\chii\parbi\dei^{-1}\delta_{{i+1}}^{-1} \delta_{\ov{i+1}}^{-1} -\chibi\pari\debi\delta_{{i+1}}^{-1} \delta_{\ov{i+1}}^{-1}\\
=\;&\chi_{\ov{i+1}}(\up\chi_i\parb_i-\chii \parb_{i}^{(1)}+\de_i^{-1})\parb_{i+1}\debi^{-1}+ \chii(\chi_{{i+1}}\parb_{i+1}^{(1)}-\up\chi_{{i+1}}\parb_{i+1})\parb_{\ov{i}}\delta_{i}\\
&+\chibi\delta^{-1}_{i+1}\parb_{i}\sbi\delta_{\ov{i}}^{-1}-\chii\parbi\dei^{-1}\delta_{{i+1}}^{-1}-\chibi\pari\debi\delta_{{i+1}}^{-1}\\
=\;&(\chii\parb_{\ov{i}}\delta_{i}+\chibi\parb_{i}\sbi\delta_{\ov{i}}^{-1}-\chii\parbi\dei^{-1}-\chibi\pari\debi)\delta_{{i+1}}^{-1}|_{\scrA_{\ol{i+1},0}}\;\; (\text{by }(\ref{opecom2}{\rm d})).
\endaligned
$$
Let $\psi$ be the four term element in parentheses. Clearly, $\psi|_{\scrA_{\ol{i+1},0}\cap\scrA_{\ol{i},0}}=0$.
Applying $\psi$ to $X^\bfa\in \scrA_{\ol{i+1},0}\cap\scrA_{\ol{i},1}$ yields $\alpha X^{\bfa+\bfe_i-\bfe_{\ol i}}$ with
$ \alpha=\upv^{a_i}-\upv^{-1}\frac{\upv-\upv^{-1}}{\upv+\upv^{-1}}[a_i]-\upv^{-a_i}
  -\upv\frac{\upv-\upv^{-1}}{\upv+\upv^{-1}}[a_i]=0.$ Hence, $\text{\eqref{QQ4}}|_{\scrA_{\ol{i+1},0}}=0$.

We now prove $\text{\eqref{QQ4}}|_{\scrA_{\ol{i+1},1}}=0$.  By Lemma \ref{opecom1}(4a)(3b),
 $\chii\parb_{i+1}\delta_{\ov{i+1}}\chi_{\ov{i+1}}\parb_{i}\delta_{\ov{i}}^{-1}|_{\scrA_{\ol{i+1},1}}=\chii\chi_{\ov{i+1}}\parb_{i+1}^{(2)}\parb_{i}\delta_{\ov{i}}^{-1}|_{\scrA_{\ol{i+1},1}}$. Also, $\chi_{\ol{i+1}}\parb_{\ol{i+1}}|_{\scrA_{\ol{i+1},1}}=1$ (applied to $L_2,L_7$  in \eqref{qqq4a} to get $L_2=0$), $\de_{\ol{i+1},1}|_{\scrA_{\ol{i+1},1}}=\up1$ and $\chi_{\ol{i+1}}\chi_{\ol{i}}=-\chi_{\ol{i}}\chi_{\ol{i+1}}$.
Thus, \eqref{qqq4a}$|_{\scrA_{\ol{i+1},1}}$ becomes
$$\aligned
(\cei\fbi-\fbi\cei)|_{\scrA_{\ol{i+1},1}}=& \chii\chi_{\ov{i+1}}\parb_{i+1}^{(2)}\parb_{i}\delta_{\ov{i}}^{-1} + \upv\chii\chi_{{i+1}}\parb_{i+1}^{(1)}\parb_{\ov{i}}\delta_{i}+(\upv^{-1} \chibi\chi_{{i+1}}\parb_{\ov{i+1}}\parb_{\ov{i}}\delta^{-1}_{i+1}\delta_{i})\\
&-\upv\chi_{\ov{i+1}}\chii \parb_{i}^{(1)}   \parb_{i+1} \delta_{\ov{i}}^{-1}-\upv^2\chi_{{i+1}} \chii\parb_{\ov{i}}\parb_{i+1}\delta_{i}+\parb_{i}\delta_{\ov{i}}^{-1} \chibi\delta^{-1}_{i+1}\sbi\\
&+(\chi_{{i+1}}\chibi\parb_{\ov{i}}\parb_{\ov{i+1}}\delta_{i}\delta^{-1}_{i+1}\sbi-\chi_{{i+1}}\parb_{\ov{i+1}}\delta_{i}\delta^{-1}_{i+1}\sbi)
\endaligned$$
Since, for the three terms in parentheses, by $\chibi\parb_{\ov{i}}|_{\scrA_{\bar i,0}}=0$ and 
$\chibi\parb_{\ov{i}}|_{\scrA_{\bar i,1}}=1$,
  \begin{equation}
\begin{aligned}\notag
(-\upv^{-1}\chi_{{i+1}}\chibi\parb_{\ov{i}}&\parb_{\ov{i+1}}\delta^{-1}_{i+1}\delta_{i}+\chi_{{i+1}}\chibi\parb_{\ov{i}}\parb_{\ov{i+1}}\delta_{i}\delta^{-1}_{i+1}\sbi-\chi_{{i+1}}\parb_{\ov{i+1}}\delta_{i}\delta^{-1}_{i+1}\sbi)|_{\scrA_{\ol{i+1},1}}\\
&=- \chi_{{i+1}}\parb_{\ov{i+1}}\delta_{i}\delta^{-1}_{i+1}\delta_{\ov{i}}^{-1}|_{\scrA_{\ol{i+1},1}},
\end{aligned}
\end{equation}
it follows that
$$\aligned
(\cei\fbi-\fbi\cei)|_{\scrA_{\ol{i+1},1}}=& \chii\chi_{\ov{i+1}}\parb_{i+1}^{(2)}\parb_{i}\delta_{\ov{i}}^{-1} + \upv\chii\chi_{{i+1}}\parb_{i+1}^{(1)}\parb_{\ov{i}}\delta_{i}-\upv\chi_{\ov{i+1}}\chii \parb_{i}^{(1)}   \parb_{i+1} \delta_{\ov{i}}^{-1} \\
&-\upv^2\chi_{{i+1}} \chii\parb_{\ov{i}}\parb_{i+1}\delta_{i}+\parb_{i}\delta_{\ov{i}}^{-1} \chibi\delta^{-1}_{i+1}\sbi-\chi_{{i+1}}\parb_{\ov{i+1}}\delta_{i}\delta^{-1}_{i+1}\delta_{\ov{i}}^{-1}.\\
\endaligned
$$
We label the six terms by $l_1,l_2,l_3^-,l_4^-,l_5,l_6^-$.
On the other hand,
$$\aligned
&\;(\ckii^{-1}\kbi-\kbii\cki^{-1})|_{\scrA_{\ol{i+1},1}}\\
=\;  & \upv^{-1} \chii\parbi\dei^{-1}\delta_{{i+1}}^{-1}
+\upv^{-1} \chibi\pari\debi\delta_{{i+1}}^{-1}  -\chi_{i+1}\parb_{\ov{i+1}}\delta^{-1}_{i+1}\delta_{{i}}^{-1}\delta_{\ov{i}}^{-1}
   -\upv\chi_{\ov{i+1}}\parb_{{i+1}}\delta_{{i}}^{-1}\delta_{\ov{i}}^{-1},\\
\endaligned
$$
whose terms are labelled by $r_1,r_2,r_3^-,r_4^-$.
Hence, the difference of $(\cei\fbi-\fbi\cei)|_{\scrA_{\ol{i+1},1}}$ and $(\ckii^{-1}\kbi-\kbii\cki^{-1})|_{\scrA_{\ol{i+1},1}}$ has 10 terms which are regrouped as follows:
$$(\cei\fbi-\fbi\cei)|_{\scrA_{\ol{i+1},1}}-(\ckii^{-1}\kbi-\kbii\cki^{-1})|_{\scrA_{\ol{i+1},1}}=\Sigma'|_{\scrA_{\ol{i+1},1}}+\Sigma''|_{\scrA_{\ol{i+1},1}},$$
where
\begin{equation}
\begin{aligned}\notag
\Sigma'=\;& \upv\chii\chi_{{i+1}}\parb_{i+1}^{(1)}\parb_{\ov{i}}\delta_{i}
 -\upv^2\chi_{{i+1}} \chii\parb_{\ov{i}}\parb_{i+1}\delta_{i}+\parb_{i}\delta_{\ov{i}}^{-1} \chibi\sbi\delta^{-1}_{i+1}\\
 &-\upv^{-1} \chii\parbi\dei^{-1}\delta_{{i+1}}^{-1}-\upv^{-1} \chibi\pari\debi\delta_{{i+1}}^{-1}(=l_2-l_4^-+l_5-r_1-r_2),\\
\Sigma''=\;&\chii\chi_{\ov{i+1}}\parb_{i+1}^{(2)}\parb_{i}\delta_{\ov{i}}^{-1}
-\upv\chi_{\ov{i+1}}\chii \parb_{i}^{(1)}   \parb_{i+1} \delta_{\ov{i}}^{-1}+\chi_{i+1}\parb_{\ov{i+1}}\delta^{-1}_{i+1}\delta_{{i}}^{-1}\delta_{\ov{i}}^{-1}  \\
&-\chi_{{i+1}}\parb_{\ov{i+1}}\delta_{i}\delta^{-1}_{i+1}\delta_{\ov{i}}^{-1}
  +\upv\chi_{\ov{i+1}}\parb_{{i+1}}\delta_{{i}}^{-1}\delta_{\ov{i}}^{-1}(=l_1-l_3^-+r_3^--l_6^-+r_4^-).
\end{aligned}
\end{equation}
We now prove both $\Sigma'|_{\scrA_{\ol{i+1},1}}=0$ and $\Sigma''|_{\scrA_{\ol{i+1},1}}=0$. This can be seen by further restricting them to $\scrA_{\ol{i+1},1}\cap \scrA_{\ol{i},0}$ and $\scrA_{\ol{i+1},1}\cap \scrA_{\ol{i},1}$. More precisely, by noting $\parbi|_{\scrA_{\ol{i},0}}=0$, we have
\begin{equation}
\begin{aligned}\notag
\Sigma'|_{\scrA_{\ol{i+1},1}\cap \scrA_{\ol{i},0}}
&=\parb_{i}\delta_{\ov{i}}^{-1} \chibi\delta^{-1}_{i+1}-\upv^{-1} \chibi\pari\delta_{{i+1}}^{-1}
=\parb_{i}(\up^{-1}\chibi\delta_{\ov{i}}^{-1} )\delta^{-1}_{i+1}-\upv^{-1} \chibi\pari\delta_{{i+1}}^{-1}
=0.\\
\Sigma'|_{\scrA_{\ol{i+1},1}\cap \scrA_{\ol{i},1}}
&= \upv\chii(\chi_{{i+1}}\parb_{i+1}^{(1)}-\up\chi_{{i+1}}\parb_{i+1})\parb_{\ov{i}}\delta_{i}-\parb_{i} \chibi\delta^{-1}_{i+1}-\upv^{-1} \chii\parbi\dei^{-1}\delta_{{i+1}}^{-1}-\chibi\pari\delta_{{i+1}}^{-1}  \\
&= \upv\chii\delta_{{i+1}}^{-1}\parb_{\ov{i}}\delta_{i}(\text{by }(\ref{opecom2}{\rm d}))-\parb_{i} \chibi\delta^{-1}_{i+1} -\upv^{-1} \chii\parbi\dei^{-1}\delta_{{i+1}}^{-1}-\chibi\pari\delta_{{i+1}}^{-1}  \\
&=\Sigma \delta_{{i+1}}^{-1},
\end{aligned}
\end{equation}
where $\Sigma=\upv\chii\parb_{\ov{i}}\delta_{i}-\parb_{i} \chibi -\upv^{-1} \chii\parbi\dei^{-1}-\chibi\pari$.
Now, for $X^\bfa\in \scrA_{\ol{i+1},1}$, we have
\begin{equation}
\begin{aligned}\notag
\Sigma(X^\bfa)=&\bigg( \upv^{a_i+1}-\frac{\upv-\upv^{-1}}{\upv+\upv^{-1}}[a_i+2]
-\upv^{-1-a_i}-\frac{\upv-\upv^{-1}}{\upv+\upv^{-1}}[a_i]\bigg)X^{\bfa+\bfe_i-\bfe_{\ol{i}}}=0,
\end{aligned}
\end{equation}
proving $\Sigma'|_{\scrA_{\ol{i+1},1}}=0$.
To prove $\Sigma''|_{\scrA_{\ol{i+1},1}}=0$, note that $\Sigma''=\Sigma'''\debi^{-1}$, where

\begin{equation}
\begin{aligned}\notag
\Sigma'''=\;& \chii\chi_{\ov{i+1}}\parb_{i+1}^{(2)}\parb_{i}     +\upv\chi_{\ov{i+1}}\parb_{{i+1}}\delta_{{i}}^{-1}-\upv\chi_{\ov{i+1}}\chii \parb_{i}^{(1)}   \parb_{i+1}  \\
&-\chi_{{i+1}}\parb_{\ov{i+1}}\delta_{i}\delta^{-1}_{i+1}+\chi_{i+1}\parb_{\ov{i+1}}\delta^{-1}_{i+1}\delta_{{i}}^{-1}.
\end{aligned}
\end{equation}
However, for $X^\bfa\in \scrA_{\ol{i+1},1}$, since
$$\beta:=\frac{\upv-\upv^{-1}}{\upv+\upv^{-1}}([a_i][a_{i+1}+2]+\upv^{-a_i+1}[a_{i+1}]-\upv[a_{i+1}][a_i+1])+\upv^{-a_{i+1}}(\upv^{-a_i}-\upv^{a_i})=0,$$
it follows that $\Sigma'''(X^\bfa)=\beta X^{\bfa+\bfe_{i+1}-\bfe_{\ol{i+1}}}=0$,
proving $\Sigma''|_{\scrA_{\ol{i+1},1}}=0$ and, hence, the relation \eqref{QQ4} is proven.

The proof of (QQ5) and (QQ6) is similar. Like the proof for (QQ4), it is necessary to break the proof of a relation $\mathcal R=0$ into the cases $\mathcal R|_{\scrA_{\ol{i+1},0}}=0$ and $\mathcal R|_{\scrA_{\ol{i+1},1}}=0$ and apply Lemma \ref{opecom1}. However, unlike the (QQ4) case, no further breaking is required in the proof of (QQ5) and (QQ6). 
\end{proof}

We remark that one may also use \eqref{EFKs} to directly check all relations (QQ1)--(QQ6) are satisfied on basis elements in $\sX$. However, the computation would be even longer. In fact, this verification was carried out by a \textsc{Matlab} program.

\setcounter{equation}{0}
\section{The $n$-fold tensor product $\sTn$ of  $\SW$}

In this section, we investigate the $\qUq$-supermodule structure on the $n$-fold tensor product of the queer polynomial superalgebra:
\begin{equation}\label{sTn}
\sTn= \SW^{\otimes n}.
\end{equation}
By setting $X_{i,j}:=\underbrace{1\otimes\cdots\otimes1}_{j-1}\otimes X_i\otimes1\cdots\otimes1$, the tensor product $\sTn$ may be regarded as the queer polynomial superalgebra $\mathbb Q(\up)[X_{i,j}]_{1\leq i\leq 2n\atop1\leq j\leq n}$ in $X_{i,j}$
(subject to the relations similar to \eqref{suppol}, according to the parity $p(i)+p(j)=p(i)\in\mathbb Z_2$), or as the algebra over the polynomial algebra $\mathbb Q(\up)[X_{i,j}]_{1\leq i,j\leq n}$
with generators $X_{\bar i,j}$ ($1\leq i,j\leq n$) and relations
\begin{equation}\label{sTn}
X_{\bar i,j}X_{\bar k,l}=-X_{\bar k,l}X_{\bar i,j}, \qquad
X_{\bar i,j}^2=\frac{\q-\q^{-1}}{\q+\q^{-1}}X_{i,j}^2,
\end{equation}
for all $1\leq i,j,k,l\leq n,(i,j)\neq (k,l).$
Since we are only interested in its $\qUq$-supermodule structure, we will simply regard $\sTn$ as a tensor superspace.

We first describe a basis for $\sTn$ in terms of matrices. 

 For $A\in M_n(\N|\N)$ as above, let
\begin{equation}\notag
\begin{aligned}%
\noma{ A}
&=X_{1}^{a_{11}^0}\cdots X_{n}^{a_{n1}^0} X_{\bar 1}^{a_{11}^1}\cdots X_{\bar n}^{a_{n1}^1}\otimes \cdots \otimes
  X_{1}^{a_{1n}^0}\cdots X_{n}^{a_{nn}^0} X_{\bar 1}^{a_{1n}^1}\cdots X_{\bar n}^{a_{nn}^1}\\
  &=X^{\mathbf c_1}\otimes X^{\mathbf c_2}\ldots\otimes X^{\mathbf c_n}, \\
\nomab{A}&=X_{1}^{[a_{11}^0]}\cdots X_{n}^{[a_{n1}^0]} X_{\bar 1}^{[a_{11}^1]}\cdots X_{\bar n}^{[a_{n1}^1]}\otimes \cdots \otimes
  X_{1}^{[a_{1n}^0]}\cdots X_{n}^{[a_{nn}^0]} X_{\bar 1}^{[a_{1n}^1]}\cdots X_{\bar n}^{[a_{nn}^1]}\\
  &=X^{[\mathbf c_1]}\otimes X^{[\mathbf c_2]}\ldots\otimes X^{[\mathbf c_n]}, \\
\end{aligned}
\end{equation}
where $\mathbf c_i=\mathbf c_i(\sqA)$ denotes the $i$-th column of $\sqA$.
Then we obtain bases
for $\sTn$
\begin{equation}\label{X-bases}
\{X^{A}\mid A\in M_n(\N|\bZ_2)\}\text{ and }
\{\nomab{A}\mid A\in M_n(\N|\bZ_2) \}.
\end{equation}
The super structure on $\SW$ extends to $\sTn$ with the following parity
\begin{equation}\label{parity2}
p(X^A)=p(A)=|A^1|:=\sum_{1\leq i,j\leq n}a^1_{i,j}.
\end{equation}

The queer polynomial superalgebra $\sTn$ has a decomposition
into its homogeneous components:
\begin{equation}\label{sTnr}
\sTn=\displaystyle\oplus_{r\geq 0}\sTnr,
\end{equation}
where $\sTnr = \span \{ \nomab{A}  \mid A\in  M_n(\bN|\bZ_2)_r\}$. 

Recall the comultiplication defined in \eqref{comult}. Let $\Delta^{(1)}=\Delta$, and for $m\geq 2$, let
$$\Delta^{(m)}=(\Delta\otimes\underbrace{1\otimes\cdots\otimes1}_{m-1})\circ\cdots\circ(\Delta\otimes 1)\circ\Delta.$$
 Then we have, for $n\geq2$, $\Delta^{(n-1)}:\qUq\to\qUq^{\otimes n}$ and 
\begin{equation}
\begin{aligned}\label{comult2}
\Delta^{(n-1)}(\k_i)&= \k_i\otimes\cdots\otimes \k_i,\\
\Delta^{(n-1)}(\e_\i)& =\sum_{1\leq j\leq n} \underbrace{1\otimes 1\cdots \otimes 1}_{j-1} \otimes \e_\i\otimes
\widetilde{\k}_{\i}^{-1}\otimes\cdots\otimes
\widetilde{\k}_{\i}^{-1},  \\
\Delta^{(n-1)}(\f_\i) &= \sum_{1\leq j\leq n} \underbrace{\widetilde{\k}_{\i}\otimes\cdots\otimes
\widetilde{\k}_{\i}}_{j-1}  \otimes \f_\i\otimes  1\otimes\cdots \otimes 1
,  \\
\Delta^{(n-1)}(\k_{\ol1}) &= \sum_{1\leq j\leq n}\underbrace{ \k_{1}^{-1}\cdots\otimes
\k_{1}^{-1} }_{j-1} \otimes \k_{\bar 1}\otimes  \k_{1}\otimes\cdots \otimes \k_{1}
.
\end{aligned}
\end{equation}

The following numbers associated with a given matrix $A=(A^0|A^1)\in M_n(\N|\bZ_2)$ with $A^0=(a_{i,j}^0)$ and $A^1=(a_{i,j}^1)$ and $ \i\in [1,n), j\in [2,n]$ will be used in the action formulas of $\sfK_i,\sfE_h, \sfF_h, \sfK_{\bar1}$ on $\sTn$ in Lemma \ref{genmul} and Theorem \ref{mulfor}.
\begin{equation}\label{sigma pm}
\begin{aligned}
&\sigma_{\sce}^+(\i,j,A)={a_{\i+1,j}^1+{\sum_{j< t\leq n}(-a_{\i,t}^0-a_{\i,t}^1+a_{\i+1,t}^0+a_{\i+1,t}^1)} },\\
&\sigma_{\sce}^-(\i,j,A)={-a_{\i+1,j}^0 +{\sum_{j< t\leq n}(-a_{\i,t}^0-a_{\i,t}^1+a_{\i+1,t}^0+a_{\i+1,t}^1)} },\\
&\sigma_{\scf}^+(\i,j,A)={-a_{\i,j}^1+{\sum_{1\leq t< j}(a_{\i,t}^0+ a_{\i,t}^1 -a_{\i+1,t}^0-a_{\i+1,t}^1)} },\\
&\sigma_{\scf}^-(\i,j,A)={a_{\i,j}^0+{\sum_{1\leq t< j}(a_{\i,t}^0+ a_{\i,t}^1 -a_{\i+1,t}^0-a_{\i+1,t}^1)}},\\
&\sigma_\sck^+(\bar1,j,A)={-a_{1j}^0-\sum_{1\leq t< j}(a_{1,t}^0+a_{1,t}^1)+ \sum_{j< t\leq n}(a_{1,t}^0+a_{1,t}^1)},\\
&\sigma_\sck^-(\bar1,j,A)=
          {a_{1j}^1-\sum_{1\leq t< j}(a_{1,t}^0+a_{1,t}^1)+ \sum_{j< t\leq n}(a_{1,t}^0+a_{1,t}^1)}\\
          & \fks_1(A^1)=0,\quad \fks_j(A^1)={\sum_{1\leq t< j, 1\leq s \leq n} a_{s,t}^1}\;(j\geq 2).
\end{aligned}
\end{equation}
Here the subscript $\sce$ indicates the $E$-action formula, etc.


We now extend the $\qUq$-action formulas in $\SW$ given in Theorem \ref{supalg} to the tensor product $\sTn$ via \eqref{comult2}.  By Remark \ref{gen rel}(2), it suffices to just consider the actions of the generators $\sfK_i,\sfE_h,\sfF_h,\sfK_{\bar 1}$.
\begin{lemma}\label{genmul} Maintain the notation above.
The  $\qUq$-supermodule $\sTn $ is spanned by  the basis
$$\{ \nomab{A} \mid A=(A^0|A^1)\in  M_n(\bN|\bZ_2)\}$$
with the following actions by generators:
\begin{enumerate} 
\item \quad\;\;$\k_i.\nomab{A} =\q^{\sum_{1\leq j\leq n}(a_{i,j}^0+a_{i,j}^1) }\nomab{A}$;
\item \quad\\\vspace{-5ex}
$$\aligned
\e_\i.\nomab{A}
=&\sum_{1\leq j\leq n;\, a_{\i+1,j}^0\neq 0} \q^{\sigma_\sce^+(\i,j,A)}
    [a_{\i ,j}^0+1]\nomab{A^0+E_{\i,j}-E_{\i+1,j}|A^1}\\
  &+\sum_{1\leq j\leq n;\, a_{\i+1,j}^1\neq 0} \q^{\sigma_\sce^-(\i,j,A)}
    [ a_{\i, j}^1+1]\nomab{A^0|A^1+E_{\i,j}-E_{\i+1,j}}.
\endaligned$$
\item \quad\\\vspace{-6ex}
$$
\aligned
\f_\i.\nomab{A}=&\sum_{1\leq j\leq n;\, a_{\i,j}^0\neq 0} \q^{\sigma_\scf^+(\i,j,A)}
    [a_{\i+1,j}^0+1]\nomab{A^0-E_{\i,j}+E_{\i+1,j}|A^1}\\
  &+\sum_{1\leq j\leq n;\, a_{\i,j}^1\neq 0} \q^{\sigma_\scf^-(\i,j,A)}
    [ a_{\i+1,j}^1+1]\nomab{A^0|A^1-E_{\i,j}+E_{\i+1,j}}.
\endaligned$$

\item\quad\\\vspace{-6ex}
$$\aligned
\quad\k_{\bar 1}.\nomab{A}
=&\sum_{1\leq j\leq n;\, a_{1,j}^1\neq 0}
                  (-1)^{\fks_j(A^1)} \q^{\sigma_\sck^+(\bar1,j,A) } [ a_{1,j}^0+1] \nomab{A^0+E_{1,j}|A^1-E_{1,j}}\\
 &+ \sum_{1\leq j\leq n;\, a_{1,j}^0\neq 0}(-1)^{\fks_j(A^1)}
     \q^{\sigma_\sck^-(\bar1,j,A) } [a_{1,j}^1+1]\nomab{A^0-E_{1,j}|A^1+E_{1,j}}.
\endaligned$$
\end{enumerate}

In particular, for every $r\geq 0$, $\sTnr$ is a $\qUq$-subsupermodule.

\end{lemma}

\begin{proof} Observe from \eqref{U-actions} that, if $a_i=0$, then the term involving $[a_i]=0$ is 0.  With this in mind, the divided power version of \eqref{U-actions} (cf. \eqref{monomial}) becomes
\begin{equation}\label{U-[actions]}
\begin{aligned}
\k_i X^{[\bfa]}&=\v^{a_i+a_{\bar i}}    X^{[\bfa]},\\
\e_\i X^{[\bfa]}&=\delta_{1,a_{h+1}}^\leq\v^{a_{\overline{\i+1}}}[a_{\i}+1] X^{[\bfa+\bfe_\i-\bfe_{\i+1}]}
                +\delta_{1,a_{\ol{h+1}}}\v^{-a_{\i+1}}[ a_{\overline{\i}}+1]  X^{[\bfa+\bfe_{\bar\i}-\bfe_{\overline{\i+1}}]},\\
 \f_\i X^{[\bfa]}&=\delta_{1,a_{h}}^\leq\v^{-a_{\bar{\i}}}[ a_{\i+1}+1]   X^{[\bfa-\bfe_\i+\bfe_{\i+1}]}
                 +\delta_{1,a_{\ol{h}}}\v^{a_{\i}}[a_{\ol{\i+1}}+1]    X^{[\bfa-\bfe_{\bar \i}+\bfe_{\overline{\i+1}}]},\\
    \k_{\bar i} X^{[\bfa]}&=(-1)^{\sum_{1\leq j < i }a_{\bar j}} \bigg(\delta_{1,a_{\ol{i}}} \v^{-a_{i}}[ a_{i}+1]
                          X^{[\bfa+\bfe_i-\bfe_{\bar i}]}+\delta_{1,a_{i}}^\leq \v^{a_{\bar i}}[a_{\bar i}+1]
                          X^{[\bfa-\bfe_{i}+\bfe_{\bar i}]}\bigg),\\
\end{aligned}
\end{equation}
where
 $\delta_{i,j}^\leq=\begin{cases}1,&\text{ if }i\leq j;\\0,&\text{ otherwise.}\end{cases}$
The $\delta$'s indicate when a term is 0, but make the formulas  a bit unpleasant to read.\footnote{If we make the convention that $X^\bfb=0$ whenever a component of $\bfb$ is negative, the $\delta$'s can be removed.}

Recall also the sign rule: for supermodules $V_1, V_2$ of a superalgebra $\mathcal A$, \begin{equation}\label{signte}
\begin{aligned}
(g_1\otimes g_2). (v_1\otimes v_2)=(-1)^{\bar{g}_2\bar{v}_1}g_1v_1\otimes g_2v_2\quad\text{for all }g_1, g_2\in\mathcal A, v_1,v_2\in V_i.
\end{aligned}
\end{equation}
Via \eqref{comult2}, the even generators $\k_i,\e_i,\f_i$ acts on $X^{[A]}$, for any $A=(A^0|A^1)\in  M_n(\N|\bZ_2)$ with columns ${\bfaa_{i}}=\bfaa_i(\sqA)$ and row $\mathbf r_i=\mathbf r_i(A)$ as follows:
\begin{equation} \notag
\begin{aligned}
\k_i.\nomab{A}&=(\k_i\otimes\cdots\otimes \k_i)(\nomab{\bfaa_{1}}\otimes \nomab{\bfaa_{2}} \cdots \otimes
 \nomab{\bfaa_{n}})
 =\q^{\sum_{1\leq j\leq n}(a_{i,j}^0+a_{i,j}^1) }\nomab{A},\\
\e_\i.\nomab{A}&=\sum_{1\leq j\leq n} (1\cdots \otimes 1 \otimes \underset{(j)}{\e_\i}\otimes
                 \widetilde{\k}_{\i}^{-1}\cdots\otimes \widetilde{\k}_{\i}^{-1})(\nomab{\bfaa_{1}}\otimes \nomab{\bfaa_{2}} \otimes\cdots \otimes
 \nomab{\bfaa_{n}})\\
 &=\sum_{1\leq j\leq n} \q^{\sum_{j< t\leq n}(-a_{\i,t}^0-a_{\i,t}^1+a_{\i+1,t}^0+a_{\i+1,t}^1)}
    (\cdots \otimes \nomab{\bfaa_{j-1}}\otimes\e_\i \nomab{\bfaa_{j}} \otimes  \nomab{\bfaa_{j+1}}\otimes \cdots)\\
         &=\sum_{1\leq j\leq n} \q^{ -\bfr_{h,(j,n]}+\bfr_{h+1,(j,n]}}
   \cdot \nomab{\bfaa_{1}}\cdots \nomab{\bfaa_{j-1}}\otimes\e_\i \nomab{\bfaa_{j}}\otimes  \nomab{\bfaa_{j+1}} \cdots \nomab{\bfaa_{n}},\\
  \end{aligned}
 \end{equation}
and
\begin{equation}\notag
\begin{aligned}
\f_\i.\nomab{A}&=\sum_{1\leq j\leq n} (\widetilde{\k}_{\i} \cdots\otimes \widetilde{\k}_{\i}   \otimes \underset{(j)}{\f_\i}\otimes  1\cdots \otimes 1)
          (\nomab{\bfaa_{1}}\otimes \nomab{\bfaa_{2}} \cdots \otimes  \nomab{\bfaa_{n}})\\
&=\sum_{1\leq j\leq n} \q^{\sum_{1\leq t< j}(a_{\i,t}^0+ a_{\i,t}^1 -a_{\i+1,t}^0-a_{\i+1,t}^1)}
(\cdots \otimes \nomab{\bfaa_{j-1}}\otimes\f_\i \nomab{\bfaa_{j}} \otimes  \nomab{\bfaa_{j+1}}\otimes \cdots)\\
         &=\sum_{1\leq j\leq n} \q^{\bfr_{h,[1,j)}-\bfr_{h+1,[1,j)}}
    \cdot\nomab{\bfaa_{1}}\cdots \nomab{\bfaa_{j-1}}\otimes \f_\i \nomab{\bfaa_{j}}\otimes \nomab{\bfaa_{j+1}} \cdots \nomab{\bfaa_{n}},\\  \end{aligned}
\end{equation}
where \begin{equation}\label{row partial sum}
\mathbf r_{i,(j,n]}=\sum_{j+1\leq t\leq n}(a_{i,t}^0+a_{i,t}^1),\qquad
\mathbf r_{i,[1,j)}=\sum_{1\leq t\leq j-1}(a_{i,t}^0+a_{i,t}^1).\end{equation}
But, by \eqref{U-[actions]} with the convention in the previous footnote, 
$$\aligned
\e_\i \nomab{\bfaa_{j}} &=\q^{a_{\i+1,j}^1 }[a_{\i,j}^0+1] \nomab{\bfaa_{j}+\bfe_\i-\bfe_{\i+1}}
    +\q^{-a_{\i+1,j}^0}[ a_{\i,j}^1+1]  \nomab{\bfaa_{j}+\bfe_{n+\i}-\bfe_{n+\i+1}}\\
\f_\i \nomab{\bfaa_{j}}&=\q^{-a_{\i,j}^1 }[a_{\i+1, j}^0+1] \nomab{\bfaa_{j}-\bfe_\i+\bfe_{\i+1}}
    +\q^{a_{\i,j}^0}[ a_{\i+1, j}^1+1]  \nomab{\bfaa_{j}-\bfe_{n+\i}+\bfe_{n+\i+1}}.
\endaligned
$$
Substituting gives
\begin{equation} \notag
\begin{aligned}
\e_\i.\nomab{A}&=\sum_{1\leq j\leq n} \q^{a_{\i+1,j}^1 -\bfr_{h,(j,n]}+\bfr_{h+1,(j,n]}}
    [a_{\i,j}^0+1]\nomab{A^0+E_{\i,j}-E_{\i+1,j}|A^1}\\
  &\quad\,+\sum_{1\leq j\leq n} \q^{-a_{\i+1,j}^0  -\bfr_{h,(j,n]}+\bfr_{h+1,(j,n]}}
    [ a_{\i,j}^1+1]\nomab{A^0|A^1+E_{\i,j}-E_{\i+1,j}}\\
  &=\sum_{1\leq j\leq n} \q^{\sigma_\sce^+(\i,j,A)}
    [a_{\i,j}^0+1]\nomab{A^0+E_{\i,j}-E_{\i+1,j}|A^1}\\
    &\quad\,+\sum_{1\leq j\leq n} \q^{\sigma_\sce^-(\i,j,A)}
    [ a_{\i,j}^1+1]\nomab{A^0|A^1+E_{\i,j}-E_{\i+1,j}},
\end{aligned}
\end{equation}
as desired. The $F_h$ case is similar.

For the odd generator $\k_{\bar1}$,  the sign rule (\ref{signte}) applies.
Thus,
\begin{equation} \notag
\begin{aligned}
\k_{\bar 1}.\nomab{A}&=\sum_{1\leq j\leq n} (\k_{1}^{-1}\cdots\otimes \k_{1}^{-1}  \otimes \underset{(j)}{\k_{\bar 1}}\otimes  \k_{1}\cdots \otimes \k_{1})
          (\nomab{\bfaa_{1}}\otimes \nomab{\bfaa_{2}} \cdots \otimes  \nomab{\bfaa_{n}})\\
 &=\sum_{1\leq j\leq n} (-1)^{\fks_j(A^1)}
     \q^{-\mathbf r_{1,[1,j)}+\mathbf r_{1,(j,n]}}\cdot X^{[\mathbf c_1]}\cdots \nomab{\bfaa_{j-1}}\otimes\k_{\bar 1} \nomab{\bfaa_{j}} \otimes  \nomab{\bfaa_{j+1}} \cdots X^{[\mathbf c_n]}\\
  \end{aligned}
\end{equation}
Now, $\k_{\bar 1} \nomab{\bfaa_{j}} =  \q^{-a_{1,j}^0}[ a_{1,j}^0+1]
                          \nomab{\bfaa_{j}+\bfe_1-\bfe_{n+1}}
             + \q^{a_{1,j}^1}[a_{1,j}^1+1]
                          \nomab{\bfaa_{j}+\bfe_{n+1}-\bfe_{1}}$. Substituting gives the required formula.
\end{proof}

For an $n\times n$ matrix $B=(b_{i,j})$, let
$$\ro(B)=(\sum_{i=1}^nb_{1,i},\sum_{i=1}^nb_{2,i},\ldots,\sum_{i=1}^nb_{n,i}),\;\;\co(B)=(\sum_{i=1}^nb_{i,1},\sum_{i=1}^nb_{i,2},\ldots,\sum_{i=1}^nb_{i,n})$$
and, for $A=(A^0|A^1)\in M_n(\bN|\bZ_2)$,
let 
\begin{equation}\label{roco}
\ro(A)=\ro(A^0)+\ro(A^1),\;\;\co(A)=\co(A^0)+\co(A^1).
\end{equation} 


\begin{corollary}\label{poly1} For every $A\in M_n(\N|\Z_2)$, $\wt(X^{[A]})=\ro(A)$.
Hence, $\sTn$ and $\sTnr$ are polynomial weight supermodules with $\wt(\sTn)=\N^n$ and $\wt(\sTnr)=(\N^n)_r:=\{(a_i)\in\N^n\mid\sum_ia_i=r\}$, respectively.
\end{corollary}
\begin{proof}
By Lemma \ref{genmul}(1), we see that the weight of $X^{[A]}$ is $\ro(A)$. Our assertion follows from Lemma \ref{genmul}.
\end{proof}

\section{Some action formulas in the completion $\wsTn$ of $\sTn$}
In this section, we use the bases for $\sTn$ given in \eqref{X-bases} to introduce its completion
\begin{equation}\label{direct product}
\wsTn:=\prod_{A\in M_n(\N|\bZ_2)}\mathbb Q(\up)X^A=\prod_{r\geq 0}\sTnr.
\end{equation}
If we identify $\sTn$ with the queer polynomial superalgebra ${\mathbb Q(\up)}[X_{i,j}]_{1\leq i\leq 2n\atop1\leq j\leq n}$, then
$\wsTn$ may be regarded as the algebra, defined over the formal power series algebra ${\mathbb Q(\up)}[[X_{i,j}]]_{1\leq i, j\leq n}$, with generators $X_{\bar i,j}$ ($1\leq i,j\leq n$) and relations \eqref{sTn}. 

Note that the $\qUq$-module structure on $\sTn$ extends to $\wsTn$. However, the superspace structure on $\sTn$ cannot be extended to $\wsTn$. We now construct a (supers) subpace $\sVn$ of $\wsTn$ and extend the action of $\qUq$ on $\sTn$ to $\sVn$. We will see in the next section that $\sVn$ is in fact a $\qUq$-supermodule.

Recall the matrix sets $M_n(\N|\N)'$ and 
$M_n(\N|\Z_2)'$ defined in \eqref{MnNZ'}.
We may identify
$M_n(\N|\N)'\times \N^n$ with $M_n(\N|\N)$ by sending $(A,\la)$,
for  $\lambda\in\mathbb N^n,  A\in M_n(\N|\N)',$ to
$A+\la$, where
\begin{equation}\label{A+la}
  A+\lambda:=(A^0+\diag(\lambda)| A^1).
  \end{equation}

According to the definitions of $\sigma_*^\pm$ in \eqref{sigma pm}, we see easily
$$\sigma_{*}^\pm(\i,j,A+\lambda)=\sigma_{*}^\pm(\i,j,A)+\sigma_{*}^\pm(\i,j,(0|0)+\lambda).$$
We list the following relations for later use in the proof of Theorem \ref{mulfor}.
\begin{lemma}\label{effpro}
For $A\in  M_n(\bN|\N)', \lambda\in\mathbb N^n, \i,j\in[1,n],$   then\vspace{1ex}

\begin{itemize}
\item[(1)] $\sigma_\sce^+(\i,j,A+\lambda)=
\begin{cases}
\sigma_\sce^+(\i,j,A){-\lambda_{\i}}+\lambda_{\i+1}, & j\leq \i-1;\\
\sigma_\sce^+(\i,j,A){+\lambda_{\i+1}}, & j=h;\\
\sigma_\sce^+(\i,j,A), & j\geq  \i+1.
\end{cases}$\vspace{1ex}
\item[(2)]$\sigma_\sce^-(\i,j,A+\lambda)=
\begin{cases}
\sigma_\sce^-(\i,j,A){-\lambda_{\i}}+\lambda_{\i+1}, & j\leq \i-1;\\
\sigma_\sce^-(\i,j,A){+\lambda_{\i+1}}, & j=\i;\\
\sigma_\sce^-(\i,j,A){-\lambda_{\i+1}}, & j= \i+1 ;\\
\sigma_\sce^-(\i,j,A), & j\geq \i+2.
\end{cases}$\vspace{1ex}
\item[(3)]$\sigma_\scf^+(\i,j,A+\lambda)=
\begin{cases}
\sigma_\scf^+(\i,j,A), & j\leq \i;\\
\sigma_\scf^+(\i,j,A)+{\lambda_{\i}}, & j= \i+1;\\
\sigma_\scf^+(\i,j,A)+{\lambda_{\i}}-{\lambda_{\i+1}}, & j\geq \i+2.
\end{cases}$\vspace{1ex}
\item[(4)] $\sigma_\scf^-(\i,j,A+\lambda)=
\begin{cases}
\sigma_\scf^-(\i,j,A), & j\leq \i-1;\\
\sigma_\scf^-(\i,j,A)+{\lambda_{\i}}, & j=\i, \i+1 ;\\
\sigma_\scf^-(\i,j,A)+{\lambda_{\i}}-\lambda_{\i+1}, & j\geq\i+2.
\end{cases}$\vspace{1ex}
\item[(5)] $\sigma_\sck^+(\bar1,j,A+\lambda)=\sigma_\sck^+(\bar1,j,A){-\lambda_{1}}.$\vspace{1ex}
\item[(6)] $\sigma_\sck^-(\bar1,j,A+\lambda)=
\begin{cases}
\sigma_\sck^+(\bar1,j,A),             & j=1;\\
\sigma_\sck^+(\bar1,j,A){-\lambda_{1}}, & j\geq 2.
\end{cases}$
\end{itemize}
\end{lemma}

For  $\bfj\in \mathbb Z^n,  A\in M_n(\N|\N)',$  define elements
\begin{equation}\label{A(j)}
 A(\bfj)
 =\sum_{\lambda\in \mathbb N^n}
 \q^{\lambda\cdot \bfj}\nomab{A+\lambda}\in \wsTn,
\end{equation}
where $\lambda\cdot\bfj=\la_1j_1+\cdots+\la_nj_n$.

For convenience, we set $ A(\bfj)=0$ as long as there exists a negative entry in $A$.

We will mainly interested in those $A(\bfj)$ with $A\in M_n(\N|\Z_2)'$.
 By assigning $A(\bfj)$ to the parity $p(A)=|A^1|\in\Z_2$ defined in \eqref{parity2}, we obtains a superspace $\sVn$ spanned by $A(\bfj)$ for all $\bfj\in \mathbb Z^n,  A\in M_n(\N|\Z_2)'$.


\begin{lemma}\label{basrel}
Let $A=(A^0|A^1)\in  M_n(\bN|\bZ_2)', \bfj=(j_1,\ldots,j_n)\in \Z^n,i,h\in [1,n] $. Then\\
\begin{itemize}
\item[(a)] $\displaystyle\sum_{\lambda\in \mathbb N^n} \q^{\lambda\cdot \bfj\pm\lambda_i\pm\lambda_\i}
     \nomab{A^0+\lambda|A^1}=  A(\bfj\pm\bfe_i\pm\bfe_\i);$\\
\item[(b)] $\displaystyle \sum_{\lambda\in \mathbb N^n}
    \q^{\lambda\cdot \bfj }
     \nomab{A^0+\lambda-E_{i,i}|A^1}=  \ups^{j_i}  A (\bfj);$\\
\item[(c)] $\displaystyle\sum_{\lambda\in \mathbb N^n}
  \q^{\lambda\cdot \bfj +\lambda_{i}}
              [\lambda_\i+1] \nomab{A^0+E_{\i,\i}+\lambda|A^1} 
=  \frac{\q^{-j_{\i}}}{\ups-\ups^{-1}}\big( A(\bfj+\bfe_{\i} +\bfe_{i})
     -A( \bfj +\bfe_{i}-\bfe_{\i})\big),$
     where $i\neq h$;\\
 \item[(d)] $\displaystyle\sum_{\lambda\in \mathbb N^n}
   \q^{\lambda\cdot \bfj-\lambda_1 }
                            [\lambda_1+1] \nomab{A^0+E_{1,1}+\lambda|A^1}  =    \frac{\q^{-j_1+1}}{\ups-\ups^{-1}}\big(A( \bfj )
     -A( \bfj -2\bfe_{1} )\big).$
\end{itemize}
\end{lemma}
\begin{proof}
The assertion (a) is clear. For (b),
\begin{equation}\notag
\sum_{\lambda\in \mathbb N^n}
    \q^{\lambda\cdot \bfj }
     \nomab{A^0+\lambda-E_{i,i}|A^1}
=\sum_{\lambda\in \mathbb N^n;\lambda_i\geq 1}
                \q^{(\lambda-\bfe_i)\cdot \bfj +j_i}
    \nomab{A^0+(\lambda-\bfe_{i})|A^1}
=
    \q^{j_i}(A^0|A^1) (\bfj).
\end{equation}
To prove (c), we first note that, for $\mu=\lambda+\bfe_{\i}$,
\begin{equation}
\begin{aligned}
 \ups^{\lambda\cdot \bfj+\lambda_i } &[\lambda_{\i}+1]= \ups^{\lambda\cdot \bfj  +\lambda_i } \frac{\ups^{\lambda_{\i}+1}
     -\ups^{-\lambda_{\i}-1}}{\ups-\ups^{-1}}= \frac{\ups^{ \lambda\cdot \bfj+\lambda_i +\lambda_{\i}+1}
     -\ups^{\lambda\cdot  \bfj +\lambda_i -\lambda_{\i}-1}}{\ups-\ups^{-1}}\\
&= \frac{\ups^{ \lambda\cdot  (\bfj+\bfe_{i} +\bfe_{\i})+1}
     -\ups^{\lambda\cdot ( \bfj +\bfe_{i}-\bfe_{\i})-1}}{\ups-\ups^{-1}}= \frac{\ups^{ \mu\cdot  (\bfj+\bfe_{i} +\bfe_{\i})-j_{\i}}
     -\ups^{\mu\cdot ( \bfj +\bfe_{i}-\bfe_{\i})-j_{\i}}}{\ups-\ups^{-1}},
\end{aligned}
\end{equation}
which is zero if $\mu_h=0$ as $i\neq h$. Thus, the left hand side of (c) becomes
\begin{equation}\notag
\begin{aligned}
\sum_{\mu\in \mathbb N^n}
                 \frac{\ups^{ \mu\cdot  (\bfj+\bfe_{i} +\bfe_{\i})-j_{\i}}
     -\ups^{\mu\cdot ( \bfj +\bfe_{i}-\bfe_{\i})-j_{\i}}}{\ups-\ups^{-1}}
                    \nomab{A^0+\mu|A^1}
= \q^{-j_{\i}}  \frac{A(\bfj+\bfe_{i} +\bfe_{\i})
     -A( \bfj +\bfe_{i}-\bfe_{\i})}  {\ups-\ups^{-1}}
                   .
\end{aligned}
\end{equation}
Finally, for (d), we have 
\begin{equation*}
\begin{aligned}
\ups^{\lambda\cdot \bfj-\lambda_i } [\lambda_{i}+1] \frac{\ups^{ \lambda\cdot \bfj +1}
     -\ups^{\lambda\cdot ( \bfj -2\bfe_{i} )-1}}{\ups-\ups^{-1}}
= \frac{\ups^{ \mu\cdot  \bfj -j_i+1}
     -\ups^{\mu\cdot ( \bfj -2\bfe_{i} )-j_i+1}}{\ups-\ups^{-1}}\quad(\mu=\lambda+\bfe_{i}),\\
\end{aligned}
\end{equation*}
which is 0 is $\mu_i=0$. Taking $i=1$, the left hand side of (d) becomes
\begin{equation}\notag
\begin{aligned}
\sum_{\mu\in \mathbb N}\frac{\ups^{ \mu\cdot  \bfj -j_1+1}
     -\ups^{\mu\cdot ( \bfj -2\bfe_{1} )-j_1+1}}{\ups-\ups^{-1}}
     \nomab{A^0+\mu|A^1}
=          \q^{-j_1+1}\frac{A( \bfj )
     -A( \bfj -2\bfe_{1} )}{\ups-\ups^{-1}}              .
\end{aligned}\vspace{-2ex}
\end{equation}
\end{proof}


%
%
Let $\beta_h=\bfe_{\i} +\bfe_{\i+1}$ and $\alpha_h=\bfe_{\i}-\bfe_{\i+1}$ for all $1\leq h<n$ and 
recall the notation $\delta_{i,j}^\leq$ used in \eqref{U-[actions]}.

\begin{theorem}\label{mulfor}For  $\bfj=(j_i)\in \mathbb Z^n, A=(A^0|A^1)=((a_{s,t}^0)|(a_{s,t}^1))\in M_n(\N|\Z_2)'$, and $1\leq i,h\leq n, h\neq n$, we have in the $\qUq$-module $\wsTn$ the following action formulas:

\begin{itemize}
\item[(1)]$
\k_i. A(\bfj)
= \q^{\sum_{1\leq j\leq n}(a_{i,j}^0+a_{i,j}^1)} A(\bfj+\bfe_i).
$
\item[]
$$\begin{aligned}
(2)\;&\e_\i. A(\bfj)= 
 \sum_{1\leq j\leq \i-1, a_{h+1,j}^0\geq1}
              \q^{\sigma_\sce^+(\i,j,A)}
    [a_{h, j}^0+1](A^0+E_{\i,j}-E_{\i+1,j}|A^1)(\bfj-\alpha_h)\\
&+\delta^\leq_{1, a^0_{h+1,h}} \q^{\sigma_\sce^+(\i,\i,A)-j_{\i}}  \frac{({A^0-E_{\i+1,\i}|A^1})(\bfj+\beta_h)
     -({A^0-E_{\i+1,\i}|A^1})( \bfj -\alpha_h)}  {\ups-\ups^{-1}}\\
&+
           \q^{\sigma_\sce^+(\i,\i+1,A)+j_{\i+1} }
    [a_{\i,\i+1}^0+1]({A^0+E_{\i,\i+1}|A^1})(\bfj)\\
    &+\sum_{\i+2\leq j\leq n,  a_{h+1,j}^0\geq1}
              \q^{\sigma_\sce^+(\i,j,A)}
    [a_{h,j}^0+1](A^0+E_{\i,j}-E_{\i+1,j}|A^1)(\bfj) \\
\end{aligned}$$
\item[]
$$\begin{aligned}
\quad&+\sum_{1\leq j\leq \i-1, a_{h+1,j}^1=1}
     \q^{\sigma_\sce^-(\i,j, A)}
    [ a_{h,j}^1+1](A^0|A^1+E_{\i,j}-E_{\i+1,j})(\bfj-\alpha_h)\\
&+\delta_{1,a_{h+1,h}^1}  \q^{\sigma_\sce^-(\i,\i, A)}
    [ a_{h,h}^1+1](A^0|A^1+E_{\i,\i}-E_{\i+1,\i})(\bfj+\bfe_{\i+1})\\
&+ \delta_{1,a_{h+1,h+1}^1}  \q^{\sigma_\sce^-(\i,\i+1, A) }
    [ a_{h,h+1}^1+1](A^0|A^1+E_{\i,\i+1}-E_{\i+1,\i+1})(\bfj-\bfe_{\i+1})\\
&+\sum_{\i+2\leq j\leq n, a_{h+1,j}^1=1}   \q^{\sigma_\sce^-(\i,j, A) }
    [ a_{h, j}^1+1](A^0|A^1+E_{\i,j}-E_{\i+1,j})(\bfj).\\
\end{aligned}$$
\item[]
$$\begin{aligned}
(3)\;\f_\i. &A(\bfj)=\sum_{1\leq j\leq \i-1,a^0_{h,j}\geq1}
     \q^{\sigma_\scf^+(\i,j,A)   }
    [a_{\i+1,j}^0+1](A^0-E_{\i,j}+E_{\i+1,j}|A^1)(\bfj)\\
&+   \q^{\sigma_\scf^+(\i\,\i,A)+j_h}
    [a_{\i+1,\i}^0+1]({A^0+E_{\i+1,\i}|A^1})(\bfj)\\
    &+\delta^\leq_{1,a^0_{h,h+1}} \q^{\sigma_\scf^+(\i,\i+1,A)-j_{h+1}} \frac{{  ({A^0-E_{\i,\i+1}|A^1}) (\bfj+\beta_h)}
     -{ ({A^0-E_{\i,\i+1}|A^1}) ( \bfj +\alpha_h)}}{\ups-\ups^{-1}}\\
&+\sum_{\i+2\leq j\leq n,a^0_{h,j}\geq1}
     \q^{\sigma_\scf^+(\i,j,A)     }
    [a_{\i+1,j}^0+1](A^0-E_{\i,j}+E_{\i+1,j}|A^1)(\bfj+\alpha_h)\\
\end{aligned}$$
\item[]\vspace{-2ex}
$$\begin{aligned}
&+\sum_{1\leq j\leq \i-1,a^1_{h,j}=1}
      \q^{\sigma_\scf^-(\i,j,A)  }
    [ a_{\i+1,j}^1+1](A^0|A^1-E_{\i,j}+E_{\i+1,j})(\bfj)\\
  &+ \delta_{1,,a^1_{h,h}}     \q^{\sigma_\scf^-(\i,\i,A)  }
    [ a_{\i+1,\i}^1+1](A^0|A^1-E_{\i,\i}+E_{\i+1,\i})(\bfj+\bfe_{\i})\\
    &+   \delta_{1,,a^1_{h,h+1}}    \q^{\sigma_\scf^-(\i,\i+1,A)   }
    [ a_{\i+1,\i+1}^1+1](A^0|A^1-E_{\i,\i+1}+E_{\i+1,\i+1})(\bfj+\bfe_{\i})\\
    &+ \sum_{\i+2\leq j\leq n,a^1_{h,j}=1}
      \q^{\sigma_\scf^-(\i,j,A)  }
    [ a_{\i+1,j}^1+1](A^0|A^1-E_{\i,j}+E_{\i+1,j})(\bfj+\alpha_h).\\
\end{aligned}$$
\item[]
$$\begin{aligned}
(4)\; \k_{\bar 1}.& A(\bfj)
=\delta_{1,a^1_{1,1}}         
                  \q^{\sigma_\sck^+(\bar1,1,A)-j_1+1}\frac{(A^0|A^1-E_{1,1})( \bfj )
     -(A^0|A^1-E_{1,1})( \bfj -2\bfe_{1} )}{\ups-\ups^{-1}}\\
&+\sum_{2\leq j\leq n,a^1_{1,j}=1}
              (-1)^{\fks_j(A^1)}
                                 \q^{\sigma_\sck^+(\bar1,j,A)}
                [ a_{1,j}^0+1] (A^0+E_{1,j}|A^1-E_{1,j})(\bfj-\bfe_1)\\
 &+  \q^{\sigma_\sck^-(\bar1,1,A)+j_1}
    [a_{1,1}^1+1](A^0|A^1+E_{1,1})( \bfj  )\\
&+ \sum_{2\leq j\leq n,a^0_{1,j}\geq1}
    (-1)^{\fks_j(A^1)}
     \q^{\sigma_\sck^-(\bar1,j,A)}
 [a_{1,j}^1+1] (A^0-E_{1,j}|A^1+E_{1,j})(\bfj-\bfe_1).\\
\end{aligned}$$
\end{itemize}
\end{theorem}
%
%
\begin{proof}

By Lemma \ref{genmul}(1), (1) is seen easily:
\begin{equation}\notag
\begin{aligned}
\k_i. A(\bfj)
&= \sum_{\lambda\in \mathbb N^n} \ups^{\lambda\cdot \bfj}\k_i.\nomab{A+\lambda}\\
&= \sum_{\lambda\in \mathbb N^n} \q^{\lambda\cdot \bfj}
      \q^{\lambda_i+\sum_{1\leq j\leq n}(a_{i,j}^0+a_{i,j}^1)}\nomab{A+\lambda}\\
 &= \q^{\sum_{1\leq j\leq n}(a_{i,j}^0+a_{i,j}^1)} A(\bfj+\bfe_i)
 \quad (\mbox{ by Lemma \ref{basrel} (a)}).
\end{aligned}
\end{equation}
By Lemma \ref{genmul}(2), we have
\begin{equation}\notag
\begin{aligned}
\e_\i. A(\bfj)
=& \sum_{\lambda\in \mathbb N^n} \q^{\lambda\cdot \bfj}\e_\i.  \nomab{A+\lambda}\\
=&\sum_{\lambda\in \mathbb N^n}\, \sum_{1\leq j\leq n,b^0_{h+1,j}>0}
    \q^{\lambda\cdot \bfj }
           \q^{\sigma_\sce^+(\i,j, A+\lambda)  }
  {[a_{\i,j}^0+1]}\nomab{A^0+\lambda+E_{\i,j}-E_{\i+1,j}|A^1}\\
  &+ \sum_{\lambda\in \mathbb N^n}\, \sum_{1\leq j\leq n,a^1_{h+1,j}>0}
    \q^{\lambda\cdot \bfj } \q^{\sigma_\sce^-(\i,j, A+\lambda)  }
    { [ a_{\i,j}^1+1] }\nomab{A^0+\lambda|A^1+E_{\i,j}-E_{\i+1,j}}\\
    =&\Sigma^++\Sigma^-,
\end{aligned}
\end{equation}
where $b^0_{i,j}$ is the $(i,j)^0$-entry of $A+\la$, $\Sigma^+$ denotes the first double summation and $\Sigma^-$ denotes the second. 
By Lemma  \ref{effpro}(1),  
\begin{equation}\notag
\begin{aligned}
\Sigma^+=&\sum_{\lambda\in \mathbb N^n}\, \sum_{\i+2\leq j\leq n,a^0_{h+1,j}>0}
    \q^{\lambda\cdot \bfj }
           \q^{\sigma_\sce^+(\i,j,A) }
    [a_{\i ,j}^0+1]\nomab{A^0+E_{\i,j}-E_{\i+1,j}+\lambda|A^1}\\
&+\sum_{\lambda\in \mathbb N^n}\, \sum_{1\leq j\leq \i-1,a^0_{h+1,j}>0}
    \q^{\lambda\cdot \bfj }
           \q^{\sigma_\sce^+(\i,j,A)-\lambda_{h}+\lambda_{h+1}}
    [a_{\i, j}^0+1]\nomab{A^0+E_{\i,j}-E_{\i+1,j}+\lambda|A^1}\\
&+\sum_{\lambda\in \mathbb N^n,\la_{h+1}>0}
    \q^{\lambda\cdot \bfj }
           \q^{\sigma_\sce^+(\i,\i+1,A) }
    [a_{\i,\i+1}^0+1]\nomab{A^0+E_{\i,\i+1}-E_{\i+1,\i+1}+\lambda|A^1}\\
&+\delta^\leq_{1, a^0_{h+1,h}}\sum_{\lambda\in \mathbb N^n}
   \q^{\lambda\cdot \bfj }
           \q^{\sigma_\sce^+(\i,\i,A)+\lambda_{h+1}}
    [\lambda_h+1] \nomab{A^0+E_{\i,\i}-E_{\i+1,\i}+\lambda|A^1}. \\
\end{aligned}
\end{equation}
The two double summations can be easily swapped. In the third summation, we use
$\up^{\la\cdot\bfj}=\up^{j_{h+1}}\up^{(\la-\bfe_{h+1})\cdot\bfj}$. 
Then applying Lemma \ref{basrel}(c) with $i=h+1$ to the last summation yields
 \begin{equation}\notag
\begin{aligned}   
\Sigma^+&= \sum_{\i+2\leq j\leq n,a^0_{h+1,j}>0}
              \q^{\sigma_\sce^+(\i,j,A)}
    [a_{\i, j}^0+1](A^0+E_{\i,j}-E_{\i+1,j}|A^1)(\bfj) \\
&+ \sum_{1\leq j\leq \i-1,a^0_{h+1,j}>0}
              \q^{\sigma_\sce^+(\i,j,A)}
    [a_{\i ,j}^0+1](A^0+E_{\i,j}-E_{\i+1,j}|A^1)(\bfj-\alpha_h)\\
&+
           \q^{\sigma_\sce^+(\i,\i+1,A)+j_{\i+1} }
    [a_{\i,\i+1}^0+1]({A^0+E_{\i,\i+1}|A^1})(\bfj)\\
&+\delta^\leq_{1, a^0_{h+1,h}}  \q^{\sigma_\sce^+(\i,\i,A)-j_{\i}}  \frac{({A^0-E_{\i+1,\i}|A^1})(\bfj+\beta_h)
     -({A^0-E_{\i+1,\i}|A^1})( \bfj -\alpha_h)}  {\ups-\ups^{-1}},\\
\end{aligned}
\end{equation}
where $\beta_h=\bfe_{\i} +\bfe_{\i+1}$ and $\alpha_h=\bfe_{\i}-\bfe_{\i+1}$. Similarly,
\begin{equation}\notag
\begin{aligned}
\Sigma^-=&\sum_{\lambda\in \mathbb N^n}\, \sum_{1\leq j\leq \i-1,a^1_{h+1,j}>0}
    \q^{\lambda\cdot \bfj } \q^{\sigma_\sce^-(\i,j, A) -\lambda_{\i}+\lambda_{\i+1}}
    [ a_{\i ,j}^1+1]\nomab{A^0+\lambda|A^1+E_{\i,j}-E_{\i+1,j}}\\
&+\sum_{\lambda\in \mathbb N^n}\, \sum_{\i+2\leq j\leq n,a^1_{h+1,j}>0}
    \q^{\lambda\cdot \bfj } \q^{\sigma_\sce^-(\i,j, A)  }
    [ a_{\i, j}^1+1]\nomab{A^0+\lambda|A^1+E_{\i,j}-E_{\i+1,j}}\\
&+\delta_{1,a^1_{h+1,h}}\sum_{\lambda\in \mathbb N^n}
    \q^{\lambda\cdot \bfj } \q^{\sigma_\sce^-(\i,\i, A)+\lambda_{\i+1}}
    [ a_{\i,\i}^1+1]\nomab{A^0+\lambda|A^1+E_{\i,\i}-E_{\i+1,\i}}\\
&+\delta_{1,a^1_{h+1,h+1}}\sum_{\lambda\in \mathbb N^n}
 \q^{\lambda\cdot \bfj } \q^{\sigma_\sce^-(\i,\i+1, A) -\lambda_{\i+1}  }
    [ a_{\i,\i+1}^1+1]\nomab{A^0+\lambda|A^1+E_{\i,\i+1}-E_{\i+1,\i+1}}\\
 \end{aligned}$$
$$\begin{aligned}
\quad
=&\sum_{1\leq j\leq \i-1,a^1_{h+1,j}=1}
     \q^{\sigma_\sce^-(\i,j, A)}
    [ a_{\i, j}^1+1](A^0|A^1+E_{\i,j}-E_{\i+1,j})(\bfj-\bfe_\i+\bfe_{\i+1})\\
&+\sum_{\i+2\leq j\leq n, a^1_{h+1,j}=1}   \q^{\sigma_\sce^-(\i,j, A) }
    [ a_{\i, j}^1+1](A^0|A^1+E_{\i,j}-E_{\i+1,j})(\bfj)\\
&+\delta_{1,a^1_{h+1,h}}  \q^{\sigma_\sce^-(\i,\i, A)}
    [ a_{\i,\i}^1+1](A^0|A^1+E_{\i,\i}-E_{\i+1,\i})(\bfj+\bfe_{\i+1})\\
&+\delta_{1,a^1_{h+1,h+1}}   \q^{\sigma_\sce^-(\i,\i+1, A) }
    [ a_{\i,\i+1}^1+1](A^0|A^1+E_{\i,\i+1}-E_{\i+1,\i+1})(\bfj-\bfe_{\i+1}).
\end{aligned}
\end{equation}
This completes the proof of (2). The proof for (3) is similar.

Finally, by Lemma \ref{genmul}(4), we have
\begin{equation}\notag
\begin{aligned}
& \k_{\bar 1}. A(\bfj)
= \sum_{\lambda\in \mathbb N^n} \q^{\lambda\cdot \bfj} \k_{\bar 1}.\nomab{A+\lambda}=\Sigma^0+\Sigma^1,\\
\end{aligned}
\end{equation}
where
$$\aligned
\Sigma^0=&\sum_{\lambda\in \mathbb N^n}\, \sum_{1\leq j\leq n,a^1_{1,j}=1}
    (-1)^{\fks_j(A^1)}\q^{\lambda\cdot \bfj }
                    \q^{\sigma_\sck^+(\bar1,j,A+\lambda)}
                { [ a_{1,j}^0+1] }\nomab{A^0+E_{1,j}+\lambda|A^1-E_{1,j}},\\
\Sigma^1 =& \sum_{\lambda\in \mathbb Z^n}\, \sum_{1\leq j\leq n,b^0_{1,j}\geq1}
   (-1)^{\fks_j(A^1)} \q^{\lambda\cdot \bfj }
    \q^{\sigma_\sck^-(\bar1,j,A+\lambda)}
    { [a_{1,j}^1+1] }\nomab{A^0-E_{1,j}+\lambda|A^1+E_{1,j}}.
\endaligned$$(Here, $b^0_{1,j}$ is the $(1,j)^0$-entry of $A+\la$.)
Thus, by Lemmas  \ref{effpro}(5) and \ref{basrel}(a),(d),
\begin{equation}\notag
\begin{aligned}
\Sigma^0=&\sum_{\lambda\in \mathbb N^n}\, \sum_{2\leq j\leq n,a^1_{1,j}=1}
   (-1)^{\fks_j(A^1)} \q^{\lambda\cdot \bfj }
                   \q^{\sigma_\sck^+(\bar1,j,A)-\lambda_1}
                [ a_{1,j}^0+1] \nomab{A^0+E_{1,j}+\lambda|A^1-E_{1,j}}\\
&+ \delta_{1,a^1_{1,1}}\sum_{\lambda\in \mathbb N^n}
   (-1)^{\fks_1(A^1)}\q^{\lambda\cdot \bfj }
                   \q^{\sigma_\sck^+(\bar1,1,A)-\lambda_1}
                [\lambda_1+1] \nomab{A^0+E_{1,1}+\lambda|A^1-E_{1,1}}\\
=&\sum_{2\leq j\leq n,a_{1,j}^1=1}
              (-1)^{\fks_j(A^1)}  
                   \q^{\sigma_\sck^+(\bar1,j,A)}
                [ a_{1,j}^0+1] (A^0+E_{1,j}|A^1-E_{1,j})(\bfj-\bfe_1)\\
&+     \delta_{1,a^1_{1,1}} \q^{\sigma_\sck^+(\bar1,1,A)-j_1+1}\frac{(A^0|A^1-E_{1,1})( \bfj )
     -(A^0|A^1-E_{1,1})( \bfj -2\bfe_{1} )}{\ups-\ups^{-1}},
\end{aligned}
\end{equation}
and, similarly by Lemma  \ref{effpro}(6), 
\begin{equation}\notag
\begin{aligned}
\Sigma^1=&\sum_{\lambda\in \mathbb N^n}\, \sum_{2\leq j\leq n,a_{1,j}^0\geq1}
   (-1)^{\fks_j(A^1)} \q^{\lambda\cdot \bfj }
     \q^{\sigma_\sck^-(\bar1,j,A)-\lambda_1}
 [a_{1,j}^1+1] \nomab{A^0-E_{1,j}+\lambda|A^1+E_{1,j}}\\
 &+ \delta^\leq_{1,\la_1} (-1)^{\fks_j(A^1)} \sum_{\lambda\in \mathbb N^n}
    \q^{\lambda\cdot \bfj }
     \q^{\sigma_\sck^-(\bar1,1,A)}
    [a_{1,1}^1+1] \nomab{A^0-E_{1,1}+\lambda|A^1+E_{1,1}}\\
=& \sum_{2\leq j\leq n}
    (-1)^{\fks_j(A^1)}
     \q^{\sigma_\sck^-(\bar1,j,A)}
 [a_{1,j}^1+1] (A^0-E_{1,j}|A^1+E_{1,j})(\bfj-\bfe_1)\\
 &+ \q^{\sigma_\sck^-(\bar1,1,A)+j_1}
    [a_{1,1}^1+1](A^0|A^1+E_{1,1})( \bfj  ),
\end{aligned}
\end{equation}
as desired.
\end{proof}

\begin{remarks}\label{odd part}
(1) The multiplication formulas in Theorem \ref{mulfor} can be easily divided into two halves: the even half (i.e., $\Sigma^+$ or $\Sigma^0$ in the proof) and the odd half (i.e., $\Sigma^-$ or $\Sigma^1$ in the proof). The even half in (2) or (3) is similar to (but not\footnote{Note the comultiplication for $\qUq$ used here.} exactly the same as) the corresponding formulas for quantum $\mathfrak{gl}_n$ in \cite[Lemma 5.3]{BLM} (see also \cite[Theorem 14.8]{ddp}). However, the odd half involves matrices of the form $(A^0|A^1+E_{h,j}-E_{h+1,j})$,
$(A^0|A^1-E_{h,j}+E_{h+1,j})$ or $(A^0-E_{1,j}|A^1+E_{1,j})$, $(A^0|A^1+E_{1,1})$, etc., which are not necessarily in $M_n(\N|\Z_2)$ when $a^1_{h,j}=1$. In other words, we do not know from these formulas if the elements $\sfE_h.A(\bfj),
\sfF_h.A(\bfj)$ and $\sfK_{\bar 1}.A(\bfj)$ are belong to $\sVn$. We resolve the issues in the next section.

(2) For the matrices occurring in the even half $\Sigma^+$ of $\sfE_h.A(\bfj)$, it is obtained by moving 1 from the $(h+1,j)^0$-entry (if $a_{h+1,j}^0>0$) to the entry above for all columns $j\neq h,h+1$, while  deducting 1 from the $(h+1,h)^0$-entry whenever $a_{h+1,h}^0>0$ for column $h$,
and always adding 1 to the $(h,h+1)^0$-entry for column $h+1$. Thus, when $a_{h+1,h}^0>0$, it is effectively moving 1 from $(h+1,h)^0$-entry to the $(h,h+1)^0$-entry. For the matrices occurring in the odd half $\Sigma^-$, it is always obtained by moving 1 from  $(h+1,j)^1$-entry (if $a_{h+1,j}^1>0$) to the entry above for every column $j$.

There are similar descriptions for matrices occurring in $\sfF_h.A(\bfj)$ and $\sfK_{\bar 1}.A(\bfj)$.
\end{remarks}

\section{The $\qUq$-supermodule $\sVn$}
We now prove that the superspace $\sVn$ spanned by all $A(\bfj)$ for $A\in  M_n(\bN|\bZ_2)', \bfj\in \Z^n$ is a $\qUq$-supermodule. This requires to show that the action given in Theorem \ref{mulfor} stabilises $\sVn$. As noted in Remark \ref{odd part}, it suffices to show that the odd parts in these action formulas belong to $\sVn$. The following lemma confirms this.

\begin{lemma}\label{close} If $A=(A^0|A^1)\in  M_n(\bN|\bZ_2)'$ and $\bfj=(j_i)\in \Z^n$, then $(A^0|A^1+E_{i,j})(\bfj)\in\sVn$.
More precisely, for $a_{i,j}^1=1$, we have
\begin{equation}\notag
\begin{aligned}
(A^0|A^1+E_{i,j})(\bfj)=
\begin{cases}
\displaystyle \frac{\q-\q^{-1}}{\q+\q^{-1}}{a_{i,j}^0+2\brack 2} (A^0+2E_{i,j}|A^1-E_{i,j})(\bfj),\;&\text{ if } i\neq j;\\
 f_\bfj \Big(\up^{-1}A'(\bfj+2\bfe_{i})
+\ups A'(\bfj-2\bfe_{i})
 -(\q+\q^{-1})A'(\bfj)\Big),\;&\text{ if }i=j \\
\end{cases}
\end{aligned}
\end{equation}
where $f_\bfj=
\displaystyle\frac{\ups^{-2j_i-1}}{(\q-\q^{-1})(\q+\q^{-1})^2} $ and $A'=({A^0|A^1-E_{i,i}})$.
\end{lemma}
\noindent
\begin{proof} We first observe that, by the relation in $X_{\bar i}^{2} =\frac{\q-\q^{-1}}{\q+\q^{-1}}X_{i}^{2}$ in \eqref{suppol},
$$\nomab{A^0|A^1+E_{i,j}}=\displaystyle \frac{\q-\q^{-1}}{\q+\q^{-1}}{a_{i,j}^0+2\brack 2}\nomab{A^0+2E_{i,j}|A^1-E_{i,j}}.$$

Let $B=(B^0|B^1)=(A^0|A^1+E_{i,j}^1)$. Then $B^0=A^0$, $b_{i,j}^1=2$  and
$$B(\bfj)=\sum_{\lambda\in \mathbb N^n} \q^{\lambda\cdot \bfj}\nomab{B+\lambda}=\sum_{\lambda\in \mathbb N^n}  \frac{\q-\q^{-1}}{\q+\q^{-1}} {a_{i,j}^0+2\brack 2} \q^{\lambda\cdot \bfj}\nomab{B^0+2E_{i,j}+\lambda|B^1-2E_{i,j}}.$$
If $ i\neq j$ then,
\begin{equation}
\begin{aligned}
B(\bfj)= \frac{\q-\q^{-1}}{\q+\q^{-1}}{a_{i,j}^0+2\brack 2} (A^0+2E_{i,j}|A^1-E_{i,j})(\bfj).
\end{aligned}
\end{equation}
To see the  $i=j$ case,
we first observe that 
\begin{equation}\notag
\begin{aligned}
\ups^{\lambda\cdot \bfj } &[\lambda_{i}+1][\lambda_{i}+2]
=\ups^{\lambda\cdot \bfj } \frac{\ups^{2\lambda_i+3}
     +\ups^{-2\lambda_i-3}-\ups-\ups^{-1}}{(\ups-\ups^{-1})^2}.\\
&= \ups^{-2j_i}\frac{\ups^{ (\lambda+2\bfe_{i})\cdot  (\bfj+2\bfe_{i})-1}
+\ups^{ (\lambda+2\bfe_{i})\cdot  (\bfj-2\bfe_{i})+1}-\ups^{ (\lambda+2\bfe_{i})\cdot  \bfj+1}
-\ups^{ (\lambda+2\bfe_{i})\cdot  \bfj-1}}{(\ups-\ups^{-1})^2}\\
&= \ups^{-2j_i}\frac{\ups^{ \mu\cdot  (\bfj+2\bfe_{i})-1}
+\ups^{ \mu\cdot  (\bfj-2\bfe_{i})+1}-(\ups+\up^{-1})
\ups^{ \mu\cdot  \bfj}}{(\ups-\ups^{-1})^2},\quad\text{where }\mu=\la+2\bfe_i.
\end{aligned}
\end{equation}
Then, for $A'=({A^0|A^1-E_{ii}})$,
\begin{equation}\notag
\begin{aligned}
B(\bfj)&=\frac{\q-\q^{-1}}{\q+\q^{-1}} \sum_{\lambda\in \mathbb N^n} \q^{\lambda\cdot \bfj}
{\lambda_i+2\brack 2}\nomab{A^0+2E_{ii}+\lambda|A^1-E_{ii}}\\
&=\frac{\q-\q^{-1}}{(\q+\q^{-1})^2} \sum_{\lambda\in \mathbb N^n} \q^{\lambda\cdot \bfj}
[\lambda_i+2][\lambda_i+1]\nomab{A^0+2E_{ii}+\lambda|A^1-E_{ii}}\\
&=\frac{\ups^{-2j_i}(\q-\q^{-1})}{(\q+\q^{-1})^2} \sum_{\mu\in \mathbb N^n,\mu_i\geq 2}\frac{\ups^{ \mu\cdot  (\bfj+2\bfe_{i})-1}
+\ups^{ \mu\cdot  (\bfj-2\bfe_{i})+1}-(\q+\q^{-1})\ups^{ \mu\cdot  \bfj}
}{(\ups-\ups^{-1})^2}
\nomab{A'+\mu}\\
&=\frac{\ups^{-2j_i}}{(\q+\q^{-1})^2} \sum_{\mu\in \mathbb N^n}\frac{\ups^{ \mu\cdot  (\bfj+2\bfe_{i})-1}
+\ups^{ \mu\cdot  (\bfj-2\bfe_{i})+1}-(\q+\q^{-1})\ups^{ \mu\cdot  \bfj}
}{(\ups-\ups^{-1})}
\nomab{A'+\mu}\\
&-\frac{\ups^{-2j_i}}{(\q+\q^{-1})^2} \sum_{\mu\in \mathbb N^n,\mu_i\leq 1}\frac{\ups^{ \mu\cdot  (\bfj+2\bfe_{i})-1}
+\ups^{ \mu\cdot  (\bfj-2\bfe_{i})+1}-(\q+\q^{-1})\ups^{ \mu\cdot  \bfj}
}{(\ups-\ups^{-1})}
\nomab{A'+\mu}\\
&=\frac{\ups^{-2j_i}}{(\q+\q^{-1})^2{(\ups-\ups^{-1})}} \bigg(\ups^{ -1} A'(\bfj+2\bfe_{i})
+\ups A'(\bfj-2\bfe_{i}) -(\q+\q^{-1})A'(\bfj)
\bigg),
\end{aligned}
\end{equation}
 since, for any $\mu\in \mathbb N^n$ with $\mu_i\leq 1$,
${\ups^{ \mu\cdot  (\bfj+2\bfe_{i})-1}
+\ups^{ \mu\cdot  (\bfj-2\bfe_{i})+1}-(\q+\q^{-1})\ups^{ \mu\cdot  \bfj}
}=0.$ 
%
\end{proof}

\begin{theorem}\label{basis A(j)} The superspace $\sVn$ is a $\qUq$-submodule of $\wsTn$ under the actions given in Theorem \ref{mulfor} and the set 
$$\fkL_\scrV=\{A(\bfj)\mid A\in M_n(\bN|\bZ_2)',\bfj\in\Z^n\}$$ forms a basis. Moreover, it is a $\qUq$-supermodule.
\end{theorem}
\begin{proof} 
The first and last assertions are seen easily from the action formulas in Theorem \ref{mulfor}
and the lemma above. It remains to prove that $\mathcal L_\scrV$ is linearly independent. 

Suppose $\sum_{i\in I,\bfj\in J}f_{i,\bfj}A_i(\bfj)=0$ for some $f_{i,\bfj}\in\mathbb Q(\up)$. Here $\{A_i\mid i\in I\}$ is a finite subset of $M_n(\bN|\bZ_2)'$ and $J$ is a finite subset of $\Z^{m+n}$. 
Since, as elements of the direct product \eqref{direct product}, the components between elements $A_i(\bfj)$ ($i\in I$) do not overlap each other, it follows that  
 $\sum_{\bfj\in J}f_{i,\bfj}A_i(\bfj)=0$ for every $i\in I$. For simplicity, we
drop subscripts $i$ and assume assume $\sum_{\bfj\in J}f_{\bfj}A(\bfj)=0$. In other words,
$$0=\sum_{\bfj\in J}f_{\bfj}A(\bfj)=\sum_\la\bigg(\sum_{\bfj\in J}f_\bfj\up^{\la\cdot\bfj}\bigg)X^{[A+\la]}.$$
Hence, $\sum_{\bfj\in J}f_\bfj\up^{\la\cdot\bfj}=0$ for every $\la\in\N^{n}$. 

We claim that there exists $\mu\in\N^{n}$ such that, for $\bfj,\bfj'\in J$, $\mu\cdot\bfj\neq\mu\cdot\bfj'$ whenever $\bfj\neq\bfj'$. Indeed, for distinct $\bfj,\bfj'\in J$, consider the polynomial
   $p_{\bfj,\bfj'}(x)=(j_1-j_1')x^{n}+(j_2-j_2')x^{n-1}+\cdots+(j_n-j_n')x$. Let 
   $$R=\{z\in\mathbb C\mid p_{\bfj,\bfj'}(z)=0 \text{ for some }\bfj,\bfj'\in J \}.$$
Since $J$ is finite, it follows that $R$ is finite. Thus, there exists $a\in\mathbb N$ and $a\not\in R$. So
$p_{\bfj,\bfj'}(a)\neq0$ for all distinct $\bfj,\bfj'\in J$. Hence, putting $\mu=(a^{n},a^{n-1},\ldots,a)$, we obtain $\mu\cdot\bfj\neq\mu\cdot\bfj'$ for all distinct $\bfj,\bfj'\in J$, proving the claim.


Now, for the $\mu$ given in the claim, choose $\la=d\mu$ for $d=0,1,\ldots,|J|-1$. Then
$\sum_{\bfj\in J}f_\bfj(\up^{\mu\cdot\bfj})^d=\sum_{\bfj\in J}f_\bfj\up^{d\mu\cdot\bfj}=0.$
Since the $|J|\times|J|$-matrix $((\up^{\mu\cdot\bfj})^d)$ is a Vandermonde  determinant which is nonzero by the selection of $\mu$, it follows that 
$f_\bfj=0$, for all $\bfj\in J$. In other words, all $f_{i,\bfj}=0$. Hence, $\mathcal L_\scrV$ is linearly independent.
\end{proof}

\begin{remark} By applying Lemma \ref{close},
the odd parts of the action formulas in Theorem \ref{mulfor} can be further refined as a linear combination of the basis $\fkL_\scrV$ if some $a_{h,j}^1$ or $a_{h+1,j}^1$ are positive (i.e., equal to 1). Thus, we obtain a matrix representation of the generators with respect to the basis $\fkL_\scrV$. However, this makes the formulas more complicated. What we will do next is to find a leading term for actions by divided powers of generators under the order relation $\preceq$ defined in \eqref{prec}. 
\end{remark}

\section{Leading terms in the action formulas}

We now have a close look at the action formulas in Theorem \ref{mulfor} and reveal a certain triangular property of the actions. The ``lower terms'' below means a linear combination of
$B(\bfj')$ with $\vec B$ strictly less than $\vec A$ of the leading term $A(\bfj)$. We deal with the even case first. Recall the conventions made in Convention~\ref{conv}. 

\begin{lemma}\label{order for E} Let $A=(A^0|A^1)\in M_n(\bN|\bZ_2)',$ $\bfj\in\Z^n$ and $1\leq h<n$. 
\begin{itemize}
\item[(1)] If  $\vec A$ starts at $a_{h,k}^0$ for some $k<h$, or 
after $a_{h+1,h}^0$, which is called the $k=h$ case below,
then, for any $m\in\N$, there exists $a\in\Z$ such that
$$\sfF_h^{(m)}A({\bfj})=\up^a(\sfF_h^{(m)}*A)({\bfj})+(\text{lower terms or 0 for $k=h$})$$
where, assuming $1\leq m\leq a_{h,k}^0$ for the $k< h$ case, the matrix
$$\sfF_h^{(m)}*A:=
\begin{cases}
(A^0-mE_{h,k}+mE_{h+1,k}|A^1),&\text{if }k< h;\\
(A^0+mE_{h+1,h}|A^1),&\text{if }k=h,\\
\end{cases}$$ 
starts at the $(h+1,k)^0$-entry $a_{h+1,k}^0+m$ for $1\leq k\leq h$ and, 
for every  lower term $B(\bfj')$ in the $k<h$ case,  $\sfF_h^{(m)}*A\succ B$ at the $(h+1,k)^0$-entry.
\item[(2)]
If  $\vec A$ starts at $a_{h+1,k}^0$ for some $k>h+1$ or after $a_{h,h+1}^0$, 
then,  for any $m\in\N$, there exists $a\in\Z$ such that
$$\sfE_h^{(m)}A({\bfj})=\up^a(\sfE_h^{(m)}*A)({\bfj})+(\text{lower terms})$$
where, assuming $1\leq m\leq a_{h+1,k}^0$ for $k>h+1$, the matrix
$$\sfE_h^{(m)}*A:=\begin{cases}(A^0+mE_{h,h+1}|A^1),\text{ if }k=h+1;\\
(A^0+mE_{h,k}-mE_{h+1,k}|A^1),\text{ if }k> h+1,
\end{cases}$$
starts at the $(h,k)^0$-entry $a_{h,k}^0+m$ and, for a lower term $B(\bfj')$,  $\sfE_h^{(m)}*A\succ B$ at the $(h,k)^0$-entry. 
\end{itemize}
\end{lemma}
Note that the hypothesis on $\vec A$ in (1) means that $A^1=0$ and $A^0$ is lower triangular with either a leading $(h,k)^0$-entry if $k<h$, or zero columns $1^-,2^-,\ldots,h^-$ if $k=h$. In (2), $A$ has the form
\begin{equation*}
A=\left(\left.
\begin{smallmatrix}
0&&\cdots&\cdots&\cdots&\cdots&0&\cdots&\cdots&0\\
&&\cdots&\cdots&\cdots&\cdots&&\cdots&\\
&\cdots&0&&\cdots&a_{h,k-1}^0&0&0&\cdots&0\\
&&&0&\cdots&a_{h+1,k-1}^0&a_{h+1,k}^0&0&\cdots&0\\
&&&&\vdots&\vdots&\vdots&&\vdots&
\end{smallmatrix}\;\;\right|\;\;
\begin{smallmatrix}
\cdots&a_{1,k-1}^1&0&\cdots&0\\
\cdots&\vdots&\vdots&\cdots&\vdots\\
\cdots&\vdots&\vdots&\cdots&\vdots\\
\cdots&a_{n,k-1}^1&0&\cdots&0
\end{smallmatrix}
\right),
\end{equation*} if $h+1< k$, 
or columns $(k+1)^+,\ldots, n^+$ and columns $\overline {k+1},\ldots,\overline{n}$ are all zeros if $k=h+1$,
\begin{proof}
(1) By the assumption on $A$, if $A$ starts at $a_{h,k}^0$ for some $k\leq h-1$, then
 $\text{supp}(h):=\{j\mid a^0_{h,j}\neq0, 1\leq j\leq 2n\}\subseteq[k,h-1]$ and $a_{h+1,k}^0=0$. If $A$ starts after $a_{h+1,h}^0$, then supp$(h)=\emptyset$ and $a_{h+1,h}^0=0$. Thus, by Theorem \ref{mulfor},
$$\aligned \sfF_h.A(\bfj)
=&\delta_{k,h-1}^\leq\sum_{k\leq j\leq \i-1,a^0_{h,j}\geq1}
     \q^{\sigma_\scf^+(\i,j,A)   }
    [a_{\i+1,j}^0+1](A^0-E_{\i,j}+E_{\i+1,j}|A^1)(\bfj)\\
    &+   \q^{\sigma_\scf^+(\i\,\i,A)+j_h}
    [a_{\i+1,\i}^0+1]({A^0+E_{\i+1,\i}|A^1})(\bfj)\endaligned$$
(the summation is 0 if supp$(h)=\emptyset$, i.e. in the $k=h$ case). 
If $\vec A$ starts at $a_{h,k}^0$ for some $k\leq h-1$, it is clear that $(F_h*A)(\bfj)$ is the leading term and, for a lower term $B(\bfj')$, columns $1^-, 2^-,\ldots,k^-$ of $B$ and $A$ are the same and $F_h*A\succ B$ at the $(h+1,k)^0$-entry.  Since $a_{h+1,k}^0=0$, inductively, we obtain
$$\sfF_h^{m}A({\bfj})=\up^a[m]^!(\sfF_h^{(m)}*A)({\bfj})+(\text{lower terms or 0 for $k=h$}),$$
proving (1). 

We now prove (2). Since $a_{h+1,j}^0=0$ for $k<j\leq n$, we have, by Theorem \ref{mulfor}(2),
$$\aligned
\sfE_h.A(\bfj)=\Sigma_{[1,h]}^+&+   \q^{\sigma_\sce^+(\i,\i+1,A)+j_{\i+1} }
    [a_{\i,\i+1}^0+1]({A^0+E_{\i,\i+1}|A^1})(\bfj)\\
    &+\delta^\leq_{h+2,k}\sum_{\i+2\leq j\leq k,  a_{h+1,j}^0\geq1}
              \q^{\sigma_\sce^+(\i,j,A)}
    [a_{h,j}^0+1](A^0+E_{\i,j}-E_{\i+1,j}|A^1)(\bfj)\\
    &+\Sigma^-_{[\bar1,\overline{k-1}]},
\endaligned $$
Here the omitted terms in $\Sigma_{[1,h]}^+$ and $\Sigma_{[1,k-1]}^-$ involve column indices $j$ with $j\in[1,h]$ or $j\in[\bar 1,\overline{k-1}]$ and so are all lower terms in comparing with $\sfE_h*A$ at the $(h,k)^0$-entry.
Hence the leading term is $(\sfE_h*A)(\bfj)$. Since $a_{h,k}^0=0$ for $h+1< k\leq n$,
the rest of the proof is similar to the argument above. 
\end{proof}

We now deal with the odd case.
\begin{lemma}\label{odd case}Let $A=(A^0|A^1)\in M_n(\bN|\bZ_2)',$ $\bfj\in\Z^n$ and $1\leq k\leq n$. Assume the first $n-k$ sections of $\vec A$ are zeros (i.e., all $\bar j$-columns and $j^+$-columns of $\vec A$ for $j=k+1, k+2,\ldots,n$  are zeros ). 
\begin{itemize}
\item[(1)] If  $a_{1,k}^1=0$,  then
$$\sfK_{\bar 1}.A(\bfj)=\up^a(\sfK_{\bar 1}*A)(\bfj')+(\text{lower terms}),$$
for some $\bfj'\in\Z^n,a\in\Z$, where
$$\sfK_{\bar 1}*A:=\begin{cases}(A^0|A^1+E_{1,1}),&\text{if }k=1;\\
(A-E_{1,k}|A^1+E_{1,k}),&\text{if }2\leq k\leq n, a_{1,k}^0\geq 1.
\end{cases}$$
More precisely,  if $B(\bfj')$ is a lower term, then $K_{\bar 1}*A\succ B$ at the $(1,k)^1$-entry.

\item[(2)] If $a_{h,k}^1=1$, $a_{h+1,k}^1=0$,
then 
$$\sfF_h.A({\bfj})=\up^b(\sfF_h*A)(\bfj')+(\text{lower terms})$$ for some $\bfj'\in\Z^n$ and $b\in\Z$, where $\sfF_h*A:=(A^0|A^1-E_{h,k}+E_{h+1,k})$. More precisely for every matrix $B$ in a lower term,
columns $k$ of $B^1$ and $A^1$ are the same  and
$\sfF_h*A\succ B$ at the $(h+1,k)^1$-entry.
\end{itemize}
\end{lemma}
\begin{proof}Assertion (1) follows from Theorem \ref{mulfor}(4). 
Since $a_{1,j}^1=0$ for all $j\geq k$ and $a_{1,j}^0=0$ for all $j>k$, the formula shows that all terms occur in $\Sigma^0$ ($=$ the sum over the first row of $A^1$ with row indices $j\in[1,k-1]$) are lower, since these matrices involved have the same column $\bar k$ as $A$ and $\sfK_{\bar 1}*A \succ A$ at the $(1,k)^1$-entry.
Similarly, every matrix $B\neq \sfK_{\bar 1}*A$ occurring in $\Sigma^1$ ($=$ the sum over the first row of $A^0$ with column indices $j\in[1,k]$) has 
the same $\bar k$-column as that of $\vec A$. Hence, $\sfK_{\bar 1}*A \succ B$ at the $(1,k)^1$-entry. (Note that the $(1,k)^1$-entry is not necessarily the leading entry of $\sfK_{\bar 1}*A $ and there is no lower terms if $k=1$ and $\bfj'=\bfj$ in this case.)

For (2), since $a^1_{h,k}=1$ is the last nonzero entry in row $h$ of $A^1$ and, in row $h$ of $A^0$, nonzero entries occurs before column $k+1$, $(A^0|A^1-E_{h,k}+E_{h+1,k})^{\vec\;}$ is clearly the leading term in $\sfF_h.A({\bfj})$, by Theorem \ref{mulfor}(3).
For the matrices $B$ in a lower term, column $k$ is $B^1$ is the same as that of $A^1$. Hence, $\sfF_h*A\succ B$ at the $(h+1,k)^1$-entry. 
\end{proof}

For any $A=(A^0|A^1)\in M_n(\bN|\bZ_2)'$, $\bfj=(j_i)\in\Z^n$, and $1\leq i, j\leq n$, let  $\fkF_{i,j}^1=\fkF_{i,j}^1(A)$ be the odd monomials  as defined in \eqref{Fij}.

\begin{corollary}\label{FijA}
Let $A=(A^0|A^1),A_\bullet=(A_\bullet^0|A_\bullet^1)\in M_n(\bN|\bZ_2)'$ with $A^i=(a_{k,l}^i)$ and $A_\bullet^i=(\ul{a}_{k,l}^i)$ such  that $\ul{a}_{1,j}^0\geq a^1_{1,j}+\cdots+a_{n,j}^1$ for $j>1$, and assume that $\vec A_\bullet$ starts after $\ul{a}^1_{1,j}$ for $j\geq 1$.
Then, for $\fkF_{i,j}^1=\fkF_{i,j}^1(A)$,
$$\fkF^1_{1,j}\fkF^1_{2,j}\cdots \fkF^1_{n,j} A_\bullet(\bfj)=\up^aB(\bfj')+\text{lower terms},$$
for some $a\in\mathbb Z$, $B=(B^0|B^1)\in M_n(\bN|\bZ_2)'$, where $B^0$ is obtained from $A_\bullet^0$ with $\ul{a}_{1,j}^0$ replaced by $\ul{a}_{1,j}^0-(a^1_{1,j}+\cdots+a_{n,j}^1)$, and $B^1=A_\bullet^1+a_{1,j}^1E_{1,j}+\cdots+a_{n,j}^1E_{n,j}$. Moreover, for every lower term $C(\bfj'')$, $C\prec B$ at the leading entry of $B$.
\end{corollary}
\begin{proof}By the hypothesis, columns $j,j+1,\ldots,n$ (resp., $j+1,\ldots,n$) of $A_\bullet^1$ (resp., the upper triangular part of $A_\bullet^0$) are zeros.
 If ${a_{n,j}^1}=1$, then $\ul{a}_{1,j}^0\geq1$ and, by Lemma \ref{odd case}(1),
$\sfK_{\bar 1}^{a_{n,j}^1}*A_\bullet$ is the matrix obtained by moving $1$ from the $(1,j)^0$-entry to $(1,j)^1$-entry if $j>1$ or by adding 1 to the $(1,1)^1$-entry if $j=1$. Then, by repeatedly applying Lemma \ref{odd case}(2), $\fkF^1_{n,j}*A_\bullet=(\sfF_{n-1}\cdots \sfF_1)*(\sfK_{\bar 1}^{a_{n,j}^1}*A_\bullet)$ is the matrix obtained by moving $1$ from $(1,j)^1$-entry to $(n,j)^1$-entry. So, effectively, $\fkF^1_{n,j}*A_\bullet$ is the matrix obtained by moving 1 from the $(1,j)^0$-entry to the $(n,j)^1$-entry.
Clearly, the lower terms, which have the same column $\bar j(=n+j)$ as $A_\bullet$ under the action of $\sfF_{n-1}\cdots \sfF_1$, remains lower at the $(n,j)^1$-entry.  Likewise, if $a_{n-1,j}^1=1$, then
$\fkF_{n-1,j}*(\fkF^1_{n,j} *A_\bullet)$ is the matrix obtained by moving $1$ from the $(1,j)^0$-entry to the $(n-1,j)^1$-entry and lower terms remain lower at the $(n-1,j)^1$-entry. Our assertion follows now from an induction.
\end{proof}

\begin{remark}\label{order preserving}
(1) For the use of next section, we observe from Lemmas \ref{order for E} and \ref{odd case} and Corollary \ref{FijA} that, if the monomial $\fkm^{A,{\bf 0}}$ acts on $O({\bf0})$, then the action of $\fkm^{A,{\bf 0}}$ is effectively a sequence of actions by operators $\fka_{k,l}^i$, $i\in\mathbb Z_2$ (for $n=3$, see Example~\ref{3by3case} below for a definition):
$$\begin{aligned}
\fkm^{A,{\bf 0}}&= \underline{\fka_{1,n}^1\cdots \fka_{n,n}^1 \fka_{1,n}^0\cdots\fka_{n-1,n}^0}\cdot
     \underline{\fka_{1,n-1}^1\cdots \fka_{n,n-1}^1 \fka_{1,n-1}^0\cdots\fka_{n-2,n-1}^0}\cdot \\
    &\cdots\underline{\fka_{1,2}^1\cdots \fka_{n,2}^1 \fka_{1,2}^0}\cdot\underline{\fka_{n,1}^1\cdots \fka_{1,1}^1} \cdot\underline{\fka_{n,1}^0\cdots \fka_{2,1}^0}\cdot\underline{ \fka_{n,2}^0\cdots \fka_{3,2}^0}\cdots \underline{\fka_{n,n-1}^0},
\end{aligned}$$
where each operator $\fka_{k,l}^i$ hitting on the leading term $B(\bfj)$ of the previous operation produces a new leading term whose associated matrix is obtained by adding a number $b_{k,l}^i\geq a_{k,l}^i$ into the $(k,l)^i$-entry of $B$, which is the leading entry if $i=0$ and equal to $a_{k,l}^i$ if $i=1$,  and a lower term produced has either zero or smaller $(k,l)^i$-entry. (Note that $b_{k,l}^0$ will become $a_{k,l}^0$ after the next operation.)

(2) Since a lower term is lower than the leading term at the entry just added, it follows that the lower terms remains lower after the next operation. More precisely, if $\fka_{k',l'}^{i'}$ is the operator next to $\fka_{k,l}^i$ and
$$\fka_{k,l}^i\fka_{s,t}^j\cdots \fka_{n,n-1}^0.O({\bf0}) =\up^a A_{k,l}^i(\bfj)+(\text{lower terms}),$$
where $A_{k,l}^i$ is the matrix with entries $b_{k,l}^i,a_{s,t}^j,\cdots ,a_{n,n-1}^0$ being in position
 and every lower term is $\prec A_{i,j}^{i}$ at the $(k,l)^i$-entry, then
$$\fka_{k',l'}^{i'}\fka_{k,l}^i\cdots \fka_{n,n-1}^0.O({\bf0}) =\up^a A_{k',l'}^{i'}(\bfj)+(\text{lower terms}),$$
where $A_{k',l'}^{i'}$ is obtained from $A_{k,l}^{i}$ by adding $b_{i',j'}^{i'}$ to the $(k',l')^{i'}$-entry, and every lower term is $\prec A_{k',l'}^{i'}$ at the $(k',l')^{i'}$-entry (always a leading entry if $i'=0$).
\end{remark}

\section{A new realisation of $\qUq$}
We are now ready to prove the main results of the paper. We first prove that the $\qUq$-supermodule $\sVn$ is cyclic and isomorphic to the regular representation of $\qUq$. In this way, we obtain a new realisation for the supergroup $\qUq$ presented by a basis and explicit multiplication formulas of basis elements by generators. The monomial basis $\mathfrak M$ established in Proposition \ref{The monomial basis} plays a crucial role in the proof.

Let $O$ denote the zero matrix $(0|0)\in M_n(\bN|\bZ_2)'$. 
 


\begin{theorem}\label{genmod2}
The $\qUq$-supermodule $\sVn$ is a cyclic module generated by  $ O(\bf 0)$, and the map 
$$\rho:\qUq\longrightarrow\sVn, u\longmapsto uO(\bf 0)$$
is a $\qUq$-supermodule isomorphism. Hence, $\sVn$ is isomorphic to the regular representation of $\qUq$.
\end{theorem}
\begin{proof}
%
By Proposiiton \ref{The monomial basis}, the image of $\rho$ is spanned by the set 
$$\mathfrak X=\{\fkm^{A,\bfj}.O({\mathbf 0})\mid A\in M_n(\bN|\bZ_2)',\bfj\in\Z^n \}.$$
We need to prove that the set is linearly independent (thus $\rho$ is injective), and  that the span contains all basis elements in $\fkL_{\scrV}$ (thus, $\rho$ is onto).

Recall from \eqref{MonBs} that 
$$\fkm^{A,\mathbf 0}=\bigg(\prod_{j=1}^n(\fkF_{1,n-j+1}^1\fkF_{2,n-j+1}^1\cdots \fkF_{n,n-j+1}^1\fkE_{n-j}^0)\bigg)\cdot\fkF_{1}^0\fkF_2^0\cdots\fkF_{n-1}^0.$$
Repeatedly applying Lemma \ref{order for E}(1) and noting Remark \ref{order preserving}, we see that
$$\fkF_{1}^0(\fkF_2^0\cdots(\fkF_{n-1}^0.O({\mathbf 0})\cdot\cdot))=f_{1|0}A^-(\mathbf 0)+\text{lower terms (LT$_{1|0}$)}$$
where $f_{1|0}\in\pm\up^\Z$ and $A^-=(A^{0,-}|0)$ with $A^{0,-}$ being the lower triangular part of $A^0$. Here every lower term $B(\bfj')$ in (LT$_{1|0}$) satisfies $B\prec A^-$ at the leading entry of $\vec A^-$.
See the example below for a more detailed building of $A^-$.

Now, by applying Corollary \ref{FijA}, we have
$$\fkF_{1,1}^1\fkF_{2,1}^1\cdots \fkF_{n,1}^1.A^-(\bfl)=f_{1|1}A_{1|1}^-(\bfj^{(1)})+\text{lower terms (LT$_{1|1}$)},$$
where $f_{1|1}\in\pm\up^\Z$ and $A^-_{1|1}=(A^{0-}|\bfc^1_1,0\ldots 0)$. (Recall that $\bfc_j^1$ is the $j$th column of $A^1$.) By Corollary \ref{FijA}, every lower term $B(\bfj')$ in (LT$_{1|1}$) satisfies $B\prec A_{1|1}^-$ at an entry in column $\bar 1$ of  $ A_{1|1}^-$.

Applying  Lemma \ref{order for E}(2) yields 
$$\fkE_1^0.A_{1|1}^-(\bfj^{(1)})=E_1^{a_{1,2}^0+|\bfc_2^1|}.A_{1|1}^-(\bfj^{(1)})=f_{2|1}A_{2|1}^-(\bfj^{(1)})+\text{lower terms (LT$_{2|1}$)},$$
where  $f_{2|1}\in\pm\up^\Z$,
$A^-_{2|1}=(A^{0-}+(a_{1,2}^0+|\bfc_2^1|)E_{12}|\bfc^1_1,0\ldots 0)$, and the lower term (LT$_{2|1}$) is 0.

Now, applying the block $\fkF_{1,2}^1\fkF_{2,2}^1\cdots \fkF_{n,2}^1$ to $A_{2|1}^-(\bfj^{(1)})$ produces by Corollary \ref{FijA} a leading term with matrix
$$A^-_{2|2}=(A^{0-}+a_{1,2}^0E_{12}|\bfc^1_1,\bfc^1_2,0\ldots 0),$$
 such that, for some $f_{2|2}\in\pm \up^\Z$ and $\bfj^{(2)}\in\Z^n$,
$$\fkF_{1,2}^1\fkF_{2,2}^1\cdots \fkF_{n,2}^1.A_{2|1}^-(\bfj^{(1)})=f_{2|2}A_{2|2}^-(\bfj^{(2)})+\text{lower terms (LT$_{2|2}$)}.$$
Here every lower term $B(\bfj')$ in (LT$_{2|2}$) satisfies $B\prec A_{2|2}^-$ at an entry in column $\bar2$ of $A_{2|2}^-$. 

Continuing this process in $(n-1)$ pairs of steps, we finally reach to the last pair of actions:
$$\aligned
\fkE_{n-1}^0.A_{n-1|n-1}(\bfj^{(n-1)})&=f_{n|n-1}A_{n|n-1}^-(\bfj^{(n-1)})+\text{lower terms (LT$_{n|n-1}$)}\\
\fkF_{1,n}^1\fkF_{2,n}^1\cdots \fkF_{n,n}^1.A_{n|n-1}^-(\bfj^{(n-1)})&=f_{n|n}A_{n|n}^-(\bfj^{(n)})+\text{lower terms (LT$_{n|n}$)},
\endaligned$$
where $A_{n|n-1}^-=(A^0|\bfc^1_1,\ldots,\bfc_{n-1}^1, 0)$, $A_{n|n}^-=A$, $f_{n,n-1},f_{n|n}\in\pm\up^\Z$, and $\bfj^{(n-1)},\bfj^{(n)}\in\Z^n$.

 Now consider the actions on lower terms occurring in the step (LT$_{i|i-1}$), the new terms produced are less than the leading term  $ A^-_{i|i}$ at an entry in the column $\bar i$ of $A^-_{i|i}$ or at an entry inherited from (LT$_{i|i-1}$). Similarly, the actions on lower terms in step (LT$_{i|i}$), the new terms produced are less than the leading term $A^-_{i+1|i}$ at the leading entry of $A^-_{i+1|i}$ or at an entry inherited from (LT$_{i|i}$). This is true for all $i=1,2,\ldots,n$ (cf. Remark \ref{order preserving}).
 Hence, putting $f_A=\prod_{k=1}^n(f_{k|k-1}f_{k|k}),\;\bfj_A=\bfj^{(n)}$, we obtain
$$
\fkm^{A,{\bf0}}.{O(\bf 0)}= f_AA(\bfj_A)+ \mbox{ lower terms}.
$$
By Theorem \ref{mulfor}(1), for any $\bfj\in\Z^n$, there exists $f_{A,\bfj}\in \pm \up^{\mathbb Z}$ such that
$$
\fkm^{A,\bfj}.O({\bf 0})=f_{A,\bfj}A({\bfj+\bfj_A})+\mbox{ lower terms}.
$$
Now Theorem \ref{basis A(j)} implies that the set $\mathfrak X$ is linearly independent, forcing that $\rho$ is injective.
On the other hand, since
$$(f_{A,\bfj}^{-1}\fkm^{A,\bfj-\bfj_A}).O({\bf 0})=A({\bfj})+\mbox{ lower terms},$$ 
it follows that every $A(\bfj)$ is in the image of $\rho$ and so $\rho$ is onto.
\end{proof}
\begin{example}\label{3by3case}
We use the matrix as given \eqref{3by3} to illustrate the proof as follows. In the notation of Remark \ref{order preserving}, we set here 
\begin{equation}\label{operators}
\aligned 
\fka_{3,1}^0&=\f_2^{(a_{31}^0)},\; \fka_{2,1}^0=\f_1^{(a_{21}^0+a_{31}^0)},\; \fka_{3,2}^0=\f_2^{(a_{32}^0)}\\
\fka_{1,1}^1&=\k_{\bar 1}^{a_{11}^1},\; \fka_{2,1}^1=
\f_{ 1}^{a_{21}^1}\k_{\bar 1}^{a_{21}^1},\; \fka_{3,1}^1=
\f_{2}^{a_{31}^1}\f_{ 1}^{a_{31}^1}\k_{\bar 1}^{a_{31}^1},\\
\fka_{1,2}^0&=\e_1^{(a_{12}^0+a_{12}^1+a_{22}^1+a_{32}^1)},\cdots\cdots.\endaligned
\end{equation}
The six step actions proceed as follows:
\begin{equation}\notag
\f_2^{(a_{31}^0)} \f_1^{(a_{21}^0+a_{31}^0)} \f_2^{(a_{32}^0)}. {O(\bf 0)}=
 \up^aA^-(\bfl)+ \mbox{ lower terms}.
\end{equation}
\begin{equation}\notag
\begin{aligned}
\k_{\bar 1}^{a_{11}^1}
\f_{ 1}^{a_{21}^1}\k_{\bar 1}^{a_{21}^1}
\f_{2}^{a_{31}^1}\f_{ 1}^{a_{31}^1}\k_{\bar 1}^{a_{31}^1}.
A^-(\bf 0)
=& f_{1|1}A^-_{1|1}(\bfj^{(1)})  + \mbox{ lower terms}.
\end{aligned}
\end{equation}
\begin{equation}\notag
\begin{aligned}
&\e_1^{(a_{12}^0+a_{12}^1+a_{22}^1+a_{32}^1)}.
A^-_{1|1}(\bfj^{(1)})
=&f_{2|1}A_{2|1}^-(\bfj^{(1)}).
\end{aligned}
\end{equation}
\begin{equation}\notag
\begin{aligned}
\k_{\bar 1}^{a_{12}^1}
\f_{ 1}^{a_{22}^1}\k_{\bar 1}^{a_{22}^1}
\f_{2}^{a_{32}^1}\f_{ 1}^{a_{32}^1}\k_{\bar 1}^{a_{32}^1}.
A_{2|1}^-(\bfj^{(1)})
=&f_{2|2 }A_{2|2}^-(\bfj^{(2)})  + \mbox{ lower terms}.
\end{aligned}
\end{equation}

\begin{equation}\notag
\begin{aligned}
&\e_1^{(a_{13}^0+a_{13}^1+a_{23}^1+a_{33}^1)}
\e_2^{(a_{13}^0+a_{23}^0+a_{13}^1+a_{23}^1+a_{33}^1)}.
A_{2|2}^-(\bfj^{(2)})
=&f_{3|2 }A_{3|2}^-(\bfj^{(2)}) + \mbox{ lower terms}.
\end{aligned}
\end{equation}

\begin{equation}\notag
\begin{aligned}
&\k_{\bar 1}^{a_{13}^1}
\f_{ 1}^{a_{23}^1}\k_{\bar 1}^{a_{23}^1}
\f_{2}^{a_{33}^1}\f_{ 1}^{a_{33}^1}\k_{\bar 1}^{a_{33}^1}.
A_{3|2}^-(\bfj^{(2)})
=& f_{3|3 }A_{3|3}^-(\bfj^{(3)}) + \mbox{ lower terms}.
\end{aligned}
\end{equation}

\begin{multicols}{2}
Here
\begin{equation}\notag
A^-=\left(\left. \begin{array}{ccc}
         0        &      0   & 0 \\
         a_{21}^0 &      0 & 0 \\
         a_{31}^0 & a_{32}^0 & 0
       \end{array}\right|\right.
\left.\begin{array}{ccc}
         0 & 0 & 0 \\
         0 & 0 & 0 \\
         0 & 0 & 0
       \end{array} \right),
\end{equation}
which is formed first by putting $a_{3,2}^0$ at the $(3,2)^0$-entry, then putting $a_{2,1}^0+a_{3,1}^0$ at the $(2,1)^0$-entry, and finally moving $a_{3,1}^0$ from $(2,1)^0$-entry to $(3,1)^0$-entry.
\end{multicols}
\begin{multicols}{2}
\begin{equation}\notag
A_{1|1}^-=\left(\left. \begin{array}{ccc}
         0        &      0   & 0 \\
         a_{21}^0 &      0 & 0 \\
         a_{31}^0 & a_{32}^0 & 0
       \end{array}\right|\right.
\left.\begin{array}{ccc}
         a_{11}^1 & 0 & 0 \\
         a_{21}^1 & 0 & 0 \\
         a_{31}^1 & 0 & 0
       \end{array} \right),
\end{equation}
which is formed first by moving $a_{3,1}^1$ down to the bottom of the 1st column of $(A^-)^1$, then $a_{2,1}^1$ to the $(2,1)^1$-entry, and finally moving $a_{1,1}^1$ to the $(1,1)^1$-entry.  \end{multicols}
The remaining matrices can be built similarly: 
\begin{equation}\notag
\begin{aligned}
A_{2|1}^-=\left(\left. \begin{array}{ccc}
         0        &      a_{12}^0+a_{12}^1+a_{22}^1+a_{32}^1   & 0 \\
         a_{21}^0 &      0 & 0 \\
         a_{31}^0 & a_{32}^0 & 0
       \end{array}\right|\right.
\left.\begin{array}{ccc}
         a_{11}^1 & 0 & 0 \\
         a_{21}^1 & 0 & 0 \\
         a_{31}^1 & 0 & 0
       \end{array} \right),
\end{aligned}
\end{equation}

\begin{equation}\notag
\begin{aligned}
A_{2|2}^-=\left(\left. \begin{array}{ccc}
         0        &      a_{12}^0   & 0 \\
         a_{21}^0 &      0 & 0 \\
         a_{31}^0 & a_{32}^0 & 0
       \end{array}\right|\right.
\left.\begin{array}{ccc}
         a_{11}^1 & a_{12}^1 & 0 \\
         a_{21}^1 & a_{22}^1 & 0 \\
         a_{31}^1 & a_{32}^1 & 0
       \end{array} \right),
\end{aligned}
\end{equation}

\begin{equation}\notag
\begin{aligned}
A_{3|2}^-=\left(\left. \begin{array}{ccc}
         0        &      a_{12}^0   & a_{13}^0+a_{13}^1+a_{23}^1+a_{33}^1 \\
         a_{21}^0 &      0 & a_{23}^0 \\
         a_{31}^0 & a_{32}^0 & 0
       \end{array}\right|\right.
\left.\begin{array}{ccc}
         a_{11}^1 & a_{12}^1 & 0 \\
         a_{21}^1 & a_{22}^1 & 0 \\
         a_{31}^1 & a_{32}^1 & 0
       \end{array} \right),
\end{aligned}\;\;A^-_{3|3}=A.
\end{equation}
\end{example}

Let $A(\bfj)^*\in \qUq$ be the preimage of the basis element $A(\bfj)\in\sVn$ such that
$$A(\bfj)^*.O(\bfl)=A(\bfj).$$ This gives rise to a third basis for $\qUq$.

\begin{theorem}\label{mthm}
The queer supergroup $\qUq$ contains the basis $$\mathfrak L=\{A(\bfj)^*\mid A\in M_n(\N|\Z_2),\bfj\in\Z^n\}$$ such that
\begin{equation}\notag
\e_i=E_{i,i+1}({\bf 0})^*,\  \f_i=E_{i+1,i}({\bf 0})^*,\
\k_i=O({\bfe_i})^*,\
\k_{\bar 1}=(0|E_{1,1})({\bf 0})^*,
\end{equation}
and the action formulas given in Theorem \ref{mulfor}(1)--(4) become the multiplication formulas of the basis elements $A(\bfj):=A(\bfj)^*$ by generators.
\end{theorem}
\begin{proof} By the action formulas in Theorem~\ref{mulfor}(1)--(4), it is straightforward to verify the following:
$$\e_i.O(\bfl)=E_{i,i+1}({\bf 0}),\;\f_i.O(\bfl)=E_{i+1,i}({\bf 0}),\;\k_i.O(\bfl)=O({\bf e}_i),\;
\k_{\bar i}.O(\bfl)=(0|E_{1,1})({\bf 0}).$$
Now apply Theorem~\ref{genmod2} together with \cite[Lemma 5.1]{DZ} to give the desired new realisation for $\qUq$.
\end{proof}
\section{The regular representation of the  queer $q$-Schur superalgebra}

Recall from \eqref{sTn} and Lemma \ref{genmul} that the $\qUq$-supermodule decomposition
$$\sTn= \SW^{\otimes n} =\displaystyle\oplus_{r\geq 0}\sTnr,$$
where $\{ \nomab{A} \mid A\in  M_n(\bN|\bZ_2)_r\}$ forms a basis for $\sTnr$. We now prove that each component $\sTnr$ is in fact the regular representation of the queer $q$-Schur algebra
$\mathcal Q_{\q}(n,r)$ introduced in \cite{DW1, DW2}. 

Recall from \eqref{roco} the row/column sum vectors $\ro(A),\co(A)$ associated with a matrix $A\in  M_n(\bN|\bZ_2)_r$ and the weight modules in Corollary \ref{poly1}. 
Let 
$$\La(n,r):=(\N^n)_r=\{(a_1,\ldots,a_n)\in\N^n\mid \sum_{i=1}^n a_i=r\}.$$
 For $\la\in\La(n,r)$, let
 $$\sTnr^\la = \span \{ \nomab{A}\mid A\in  M_n(\bN|\bZ_2), \co(A)=\lambda\}.$$
 Then
 $$
\sTnr
 =\bigoplus_{\lambda\in\Lambda(n,r)}\sTnr^\la.
 $$
 For $A\in M_n(\bN|\bZ_2)$, let $A'$ be the matrix obtained from $A$ by replacing the diagonal of $A^0$ with zeros.
\begin{lemma}\label{qqsdec}
  For every $\lambda\in\Lambda(n,r)$, $\sTnr^\la$ is a $\qUq$-subsupermodule with bases $\{ \nomab{A}\mid A\in  M_n(\bN|\bZ_2), \co(A)=\lambda\}$ and
 $$\{ \fkm^{A',{\bf0}}.\nomab{\diag(\lambda)} \mid  A\in M_n(\bN|\bZ_2),\ \co(A)= \lambda \}.$$
 Moreover, the $\qUq$-supermodules 
 $$\sTnr^\la= \qUq.\nomab{\diag(\la)}\quad \text{and}\quad
 \sTnr=\qUq. {\bf 1_r}
 $$
are both cyclic, where ${\bf 1}_r=\sum_{\lambda\in\Lambda(n,r)} \nomab{\diag(\la)}\in \sTnr.$
 \end{lemma}
 \begin{proof}By a close look at the action formulas in Lemma \ref{genmul}, the actions on $X^{[A]}$ by the generators does not change $\co(A)$. Hence, $\sTnr^\la$ is a $\qUq$-subsupermodule with the defining basis $\{ \nomab{A}\mid A\in  M_n(\bN|\bZ_2), \co(A)=\lambda\}$. In particular, $\qUq.\nomab{\diag(\la)}\subseteq\sTnr^\la$. We now prove that the converse inclusion holds.
 
 We first introduce an order relation $\preceq_\co$ on $M_n(\bN|\bZ_2)_r$ by setting
 $$
 A\preceq_\co B\iff A\preceq B \mbox{ and } \co(A)=\co(B).
 $$

We claim that, for $\la=\co(A)$,
\begin{equation}\label{SchurTri}
\fkm^{A',{\bf0}}.\nomab{\diag(\la)}= g_A \nomab{A}+ (\mbox{lower terms}_{\preceq_\co}),\quad\text{for some }g_A\in \pm\up^\Z,
\end{equation}
which gives the second basis assertion. 

To see this, we apply an argument similar to the proof of Theorem \ref{genmod2}, but note the following differences: the initial matrix in the proof there is $O=(0|0)$ and the leading term in $\fkm^{A',{\bf0}}.O({\bf 0})$ is $A'(\bfj_{A'})$. One recovers $A'$ from $O$. Here the initial matrix is $\diag(\la):=(\diag(\la)|0)$ and the leading term in $\fkm^{A',{\bf0}}.\nomab{\diag(\la)}$ ($\la=\co(A)$) is $\nomab{A}$. So one recovers $A$ from $\diag(\la)$. Note also that the sequence of actions (see Remark \ref{order preserving}) in computing $\fkm^{A',{\bf0}}.\nomab{\diag(\la)}$ move every $b_{k.l}^i\geq a_{k,l}^i$, $k\neq l$ (equality if $i=1$), from the diagonal of $\diag(\la)$ to the $(k,l)^i$-entry.

First we compute
$\fkF_{1}^0\fkF_2^0\cdots\fkF_{n-1}^0.X^{[\diag(\la)]}$.
By Lemma \ref{genmul}(3), 
$$\fkF_{n-1}^0.X^{[\diag(\la)]}=\sfF_{n-1}^{(a_{n,n-1}^0)}.X^{[\diag(\la)]}=
\begin{cases}\up^aX^{[a_{n,n-1}E_{n,n-1}+\la-a_{n,n-1}^0\bfe_{n-1}]},&\text{if }a_{n,n-1}^0\neq0;\\
X^{[\diag(\la)]},&\text{otherwise.}\end{cases}$$
(This action moves $a_{n,n-1}^0$, if $\neq0$, one step down from the diagonal position.) 
Now the two factors in $\fkF_{n-2}^0$ will move $a^0_{n-1,n-2}+a^0_{n,n-2}$ a step down from the diagonal and then move $a^0_{n,n-2}$ down by one step.
Inductively, we see that
$$\fkF_{1}^0\fkF_2^0\cdots\fkF_{n-1}^0.X^{[\diag(\la)]}=g_{1|0}X^{[A^-]}+(\text{lower terms}_{\preceq_\co}),$$
where $A^-$ is the matrix with all lower triangular entries of $A^0$ being moved in position from the diagonal. 

Now, with a similar notation used in the proof of Theorem \ref{genmod2}, we have
$$\fkF_{1,1}^1\fkF_{2,1}^1\cdots \fkF_{n,1}^1.X^{[A^-]}=g_{1|1}X^{[A_{1|1}^-]}+(\text{lower terms}_{\preceq_\co}),$$
where $A_{1|1}^-$ has all entries in column 1 of $A^1$ moved from the diagonal of $A^-_{1|0}$ in position. 

Similarly, we have
    $$\aligned
    \mathfrak E_1^0.X^{[A_{1|1}^-]}&=g_{2|1}X^{[A_{2|1}^-]}+(\text{lower terms}_{\preceq_\co}).\\
    \fkF_{1,2}^1\fkF_{2,2}^1\cdots \fkF_{n,2}^1.X^{[A^-_{2|1}]}&=g_{2|2}X^{[A_{2|2}^-]}+(\text{lower terms}_{\preceq_\co}),\endaligned$$
where $g_{2,1},g_{2|2}\in\pm\up^\Z$, $A_{2|1}^-$ is obtained from $A_{1|1}^-$ by moving $a_{1,2}^0$ one step upwards from the diagonal, and $A_{2|2}^-$ is obtained from $A_{2|1}^-$ by moving $a_{1,2}^1,a_{2,2}^1,\ldots, a_{n,2}^1$ in position from the diagonal.

Continuing this in $(n-1)$ pairs of steps, we finally reach to the last pair of actions:
$$\aligned
\fkE_{n-1}^0.X^{[A_{n-1|n-1}]}&=g_{n|n-1}X^{[A_{n|n-1}^-]}+(\text{lower terms}_{\preceq_\co}),\\
\fkF_{1,n}^1\fkF_{2,n}^1\cdots \fkF_{n,n}^1.X^{[A_{n|n-1}^-]}&=f_{n|n}X^{[A_{n|n}^-]}+(\text{lower terms}_{\preceq_\co}),
\endaligned$$
where $A_{n|n-1}^-=(A^0+|\bfc_n|E_{n,n}|\bfc^1_1,\ldots,\bfc_{n-1}^1, 0)$, $A_{n|n}^-=A$, $g_{n,n-1},g_{n|n}\in\pm\up^\Z$.
Now, \eqref{SchurTri} follows from a similar order preserving property as described in Remark \ref{order preserving}(2). 

Thus, by the claim,
$\qUq.\nomab{\diag(\la)}\supseteq\sTnr^\la.$
Hence,
$$ \qUq.\nomab{\diag(\la)}=\sTnr^\la.$$
Finally, for $\lambda\in\Lambda(n,r)$, since
 $$
 \prod^n_{i=1}\begin{bmatrix}\sfK_i\\ \lambda_i\end{bmatrix} . {\bf 1_r}
 =\sum_{\mu\in\Lambda(n,r)}\prod^n_{i=1}\begin{bmatrix}\sfK_i\\ \lambda_i\end{bmatrix} . \nomab{\diag(\mu)}
 =\sum_{\mu\in\Lambda(n,r)}\prod^n_{i=1}\begin{bmatrix}\mu_i\\ \lambda_i\end{bmatrix} . \nomab{\diag(\mu)}
 =\nomab{\diag(\la)},
 $$
it follows that every $\nomab{\diag(\la)}\in \qUq. {\bf 1_r}$ and so
 $$
\qUq. {\bf 1_r}
 =\bigoplus_{\lambda\in\Lambda(n,r)}\qUq.\nomab{\diag(\la)}=\bigoplus_{\lambda\in\Lambda(n,r)} \sTnr^\la=\sTnr.\vspace{-4ex}
 $$
 \end{proof}
 
 The last assertion of the following result follows from a general construction of the category
 $\mathcal O^{\geq0}_{\text{int}}$ from \cite{GJKK} (see also \cite[Def. 1.5, Rem. 1.6, Prop.1.7(3)]{GJKKK2}, since it can be seen easily that $\mathscr A_\up(n,k)$ belongs to $\mathcal O^{\geq0}_{\text{int}}$ and $\sTnr$, as a direct summand of $\oplus_{k=0}^r\mathscr A_\up(n,k)^{\otimes n}$,
 belongs to $\mathcal O^{\geq0}_{\text{int}}$. For completeness,
 we include a proof. 
 \begin{lemma}\label{wan}
 \begin{enumerate}
 \item $\sfE_i^{(m)}\sfK_{\bar i}=\sfE_i\sfK_{\bar i}\sfE_i^{(m-1)}-[m-1]\sfK_{\bar i}\sfE_i^{(m)}$.
\item 
 $\sfK_{\ol{i+1}}=\sfE_i\sfK_{\bar i}\sfF_i-\frac{\up^{-1}\widetilde \sfK_i-\up \widetilde \sfK_i^{-1}}{\up-\up^{-1}}\sfK_{\bar i}-\up \sfF_i\sfE_i\sfK_{\bar i}+\sfF_{\bar i}\sfE_i\sfK_i.$
 \item For any $A\in M_n(\N|\Z_2)$ with $\la=\ro(A)$, $\sfK_{\bar i}.X^{[A]}=0$ if $\la_{i}=0$.
 \end{enumerate}
 \end{lemma}
 
 \begin{proof}Assertions (1) and (2) follow from the proof of \cite[Lem. 1.4]{GJKKK2}.
 
  
 (3) We apply induction on $i$. The case for $i=1$ follows from Lemma \ref{genmul}. Assume it is true for $i\geq 1$. We prove that 
 $\sfK_{\ol {i+1}}.X^{[A]}=0$ whenever $\la_{i+1}=0$.
 
 Let $m=\la_i$. if $m=0$, the assertion follows from (2) above, since $\sfF_i.X^{[A]}=0=\sfE_i.X^{[A]}$;
 see Corollary \ref{poly1} and Lemma \ref{poly2}.
Assume now $m>0$. By the commutation formula
$$\sfE_i^{(a)}\sfF_i^{(b)}=\sum_{t=0}^{\text{min}(a,b)}\sfF^{(b-t)}\left[{\widetilde \sfK_t;2t-a-b\atop t}\right]\sfE^{(a-t)},$$
we can easily deduce that $\sfE_i^{(m)}\sfF_i^{(m)}.X^{[A]}=X^{[A]}$ and $\sfE_i^{(m-1)}\sfF_i^{(m)}.X^{[A]}=\sfF_i.X^{[A]}$, which imply, by (2)\&(1),
$$\aligned
\sfE_i^{(m)}\sfK_{\bar m}\sfF_i^{(m)}.X^{[A]}&=(\sfE_i\sfK_{\bar i}\sfE_i^{(m-1)}-[m-1]\sfK_{\bar i}\sfE_i^{(m)})(\sfF_i^{(m)}.X^{[A]})\\
&=\sfE_i\sfK_{\bar i}\sfE_i.X^{[A]}-[m-1]\sfK_{\bar i}.X^{[A]}\\
&=\sfK_{\ol{i+1}}.X^{[A]}.
\endaligned
$$
Since the weight of $\sfF_i^{(m)}.X^{[A]}$ has zero at the $i$th component, so
$\sfK_{\ol{m}}\sfF_i^{(m)}.X^{[A]}=0$ by induction. Hence, $\sfK_{\ol{i+1}}.X^{[A]}=0$.
 \end{proof}
 
 Let $\mathcal Q_{\q}(n,r)$ be the queer $q$-Schur superalgebra $(q=\up^2)$ introduced in \cite{DW1}.
 We now prove that the $\qUq$-supermodule $\sTnr$ is isomorphic to the regular representation of $\mathcal Q_{\q}(n,r)$.
 
\begin{theorem}\label{reaque}
Maintain the notation above and let 
$$I_r=\text{\rm ann}_{\qUq}(\sTnr)=\{u\in\qUq\mid u.\sTnr=0\}.$$ Then there is a superalgebra isomorphism
\begin{equation}\label{queiso}
\mathcal Q_{\q}(n,r)\cong\qUq/I_r.
\end{equation}
Moreover, $\sTnr$, regarded as a $\mathcal Q_{\q}(n,r)$-supermodule, is isomorphic to the regular representation $_{\mathcal Q_{\q}(n,r)}\mathcal Q_{\q}(n,r)$ of $\mathcal Q_{\q}(n,r)$.
\end{theorem}
\begin{proof}Consider the $\qUq$-module homomorphism
$$\tilde\rho_r:\qUq\longrightarrow \sTnr, u\longmapsto u.{\bf1}_r.$$
By Lemma \ref{qqsdec}, this homomorphism is surjective.
Clearly, $I_r\subseteq\text{ker}(\rho)$. Thus, this homomorphism induces an epimorphism
$$\rho_r:\qUq/I_r\longrightarrow \sTnr, u\longmapsto u.{\bf1}_r.$$
On the other hand,
by \cite[Theorem 9.2]{DW1},
$U_{\q}(\mfq)/J_r \cong \mc Q_{\q}(n,r),$
where   $J_r$  is the ideal generated by the elements:
 $$\sfK_1\cdots \sfK_n-\q^r,\;\;(\sfK_i-1)(\sfK_i-\q )\cdots(\sfK_i-\q^r),\,\,\sfK_{\bi}(\sfK_i-\q)\cdots(\sfK_i-\q^r),  \;1\leq i\leq n.$$
 Clearly, the first two generators are in $I$ and so is the third by Lemma \ref{wan}(3). So $J_r\subseteq I_r$ and $\rho_r$ induces an epimorphism.
  $$\ol{\rho}_r:\mathcal Q_\up(n,r) \longrightarrow \sTnr, u\longmapsto u.{\bf1}_r.$$
 dimensional comparison forces that $\ol\rho_r$ must be an isomorphism.
\end{proof}

Let $\pi_r:\qUq\to\mathcal Q_\up(n,r)$ be the quotient morphism. 
The proof above implies the following.
Recall the notation introduced in \eqref{Kla}.
\begin{corollary}The image under $\pi_r$ of the set 
$\bigg\{\fkm^{A,{\bf0}}\begin{bmatrix}\sfK\\\co(A)\end{bmatrix}\mid A\in M_n(\N|\Z_2)_r\bigg\}$
forms a basis for $\mc Q_{\q}(n,r)$.
\end{corollary}
\begin{proof}By the proof of Lemma \ref{qqsdec}, the set $\{\fkm^{A',{\bf0}}\left[{\sfK\atop\co(A)}\right]{\bf 1}_r\mid A\in M_n(\N|\Z_2)_r\}$ forms a basis for $\sTnr$. Now the assertion follows from the relation
$$\ol\rho_r\bigg(\Big(\pi_r\fkm^{A',{\bf0}}\left[{\sfK\atop\co(A)}\right]\Big)\bigg)=\fkm^{A',{\bf0}}\left[{\sfK\atop\co(A)}\right]{\bf 1}_r.$$
\end{proof}
For any $A\in  M_n(\bN|\bZ_2)_r$, let $\psi_A$ be the unique element in $\mathcal Q_\up(n,r)$ such that $\psi_A.{\bf 1}_r=X^{[A]}$. In other words, $\psi_A=\ol{\rho}_r^{-1}(X^{[A]})$. We now use the regular representation of $\mathcal Q_\up(n,r)$ to get a new presentation for $\mathcal Q_\up(n,r)$.

\begin{theorem}\label{Qnr}The queer $q$-Schur superalgebra $\mc Q_{\q}(n,r)$ has a basis
$$\{\psi_A\mid A\in  M_n(\bN|\bZ_2)_r,$$ and its generators
$$e_h:=\pi_r(\sfE_h),\;\; f_h:=\pi_r(\sfF_h), \;\;k_i:=\pi_r(\sfK_i),\;\; k_{\bar 1}:=\pi_r(\sfK_{\bar 1})\;\;(1\leq h<n,1\leq i\leq n)$$
have the following matrix representations relative to the basis:
\begin{enumerate} 
\item \quad\;\;$k_i\psi_{A} =\q^{\sum_{1\leq j\leq n}(a_{i,j}^0+a_{i,j}^1) }\psi_{A}$;
\item \quad\\\vspace{-4ex}
$$\aligned
e_\i\psi_{A}
=&\sum_{1\leq j\leq n;\, a_{\i+1,j}^0\neq 0} \q^{\sigma_\sce^+(\i,j,A)}
    [a_{\i ,j}^0+1]\psi_{(A^0+E_{\i,j}-E_{\i+1,j}|A^1)}\\
  &+\sum_{1\leq j\leq n;\, a_{\i+1,j}^1\neq 0} \q^{\sigma_\sce^-(\i,j,A)}
    [ a_{\i, j}^1+1]\psi_{(A^0|A^1+E_{\i,j}-E_{\i+1,j})}.
\endaligned$$
\item \quad\\\vspace{-5ex}
$$
\aligned
f_\i\psi_{A}=&\sum_{1\leq j\leq n;\, a_{\i,j}^0\neq 0} \q^{\sigma_\scf^+(\i,j,A)}
    [a_{\i+1,j}^0+1]\psi_{(A^0-E_{\i,j}+E_{\i+1,j}|A^1)}\\
  &+\sum_{1\leq j\leq n;\, a_{\i,j}^1\neq 0} \q^{\sigma_\sce^-(\i,j,A)}
    [ a_{\i+1,j}^1+1]\psi_{(A^0|A^1-E_{\i,j}+E_{\i+1,j})}.
\endaligned$$

\item\quad\\\vspace{-5ex}
$$\aligned
\quad k_{\bar 1}\psi_{A}
=&\sum_{1\leq j\leq n;\, a_{1,j}^1\neq 0}
                  (-1)^{\fks_j(A^1)} \q^{\sigma_\sck^+(\bar1,j,A) } [ a_{1,j}^0+1] \psi_{(A^0+E_{1,j}|A^1-E_{1,j})}\\
 &+ \sum_{1\leq j\leq n;\, a_{1,j}^0\neq 0}(-1)^{\fks_j(A^1)}
     \q^{\sigma_\sck^-(\bar1,j,A) } [a_{1,j}^1+1]\psi_{(A^0-E_{1,j}|A^1+E_{1,j})}.
\endaligned$$
\end{enumerate}
Here, for a matrix $B=(B^0|B^1)$ with $b_{i,j}^1=2$, $\psi_B=\frac{\up-\up^{-1}} {\up+\up^{-1}} 
\psi_{(B^0+2E_{i,j}|B^1-2E_{i,j})}$.
\end{theorem}
\begin{proof}
This follows immediately from Lemma \ref{genmul} and \cite[Lemma 5.1]{DZ}.
\end{proof}
\begin{remarks} (1)
As given in \cite[(9.5),(9.8)]{DW1}, the queer $q$-Schur superalgebra is the endomorphism algebra of the $r$-fold tensor superspace of the natural representation of $\qUq$ over the Hecke-Clifford algebra. It is natural to expect that the basis element $\psi_{(A^0|A^1)}$ should agree with the linear map $\phi_{(A^0|A^1)}$ (up to a signed power of $\up$). This identification is crucial to lifting the Schur--Weyl--Sergeev duality to the integral level.

(2) The integral theory developed in the paper has set down some foundation for establishing the theory of polynomial super representations of $\qUq$ in positive quantum characteristics.
\end{remarks}

%

\end{document}